\documentclass[a4paper,reqno,10pt]{amsart}

\usepackage{a4wide}

\usepackage{amssymb}
\usepackage{amstext}
\usepackage{amsmath}
\usepackage{amscd}
\usepackage{amsthm}
\usepackage{amsfonts}

\usepackage{tikz}
\usetikzlibrary{intersections, calc, arrows, cd}

\usepackage[bookmarks=true,draft=false,%
breaklinks,colorlinks,citecolor=green,linkcolor=red]{hyperref}

\usepackage{enumitem}
\usepackage[justification=centering]{caption}

\setenumerate[1]{label={\upshape(\arabic*)}}
\setenumerate[2]{label={\upshape(\alph*)}}
\numberwithin{equation}{section}
\numberwithin{figure}{section}

\newtheorem{theorem}{Theorem}[section]
\newtheorem{theoremi}{Theorem}

\newtheorem{corollary}[theorem]{Corollary}
\newtheorem{lemma}[theorem]{Lemma}
\newtheorem{proposition}[theorem]{Proposition}
\newtheorem{definition-proposition}[theorem]{Definition-Proposition}

\newtheorem{question}[theorem]{Question}

\theoremstyle{definition}
\newtheorem{definition}[theorem]{Definition}

\newtheorem{remark}[theorem]{Remark}
\newtheorem{example}[theorem]{Example}
\newtheorem{examplei}[theoremi]{Example}

\theoremstyle{plain}
\newtheorem{thm}[theorem]{Theorem}
\newtheorem{prp}[theorem]{Proposition}
\newtheorem{lem}[theorem]{Lemma}
\newtheorem{cor}[theorem]{Corollary}
\newtheorem{fct}[theorem]{Fact}

\theoremstyle{definition}
\newtheorem{dfn}[theorem]{Definition}
\newtheorem{rmk}[theorem]{Remark}
\newtheorem{ex}[theorem]{Example}
\newtheorem{cond}[theorem]{Condition}

\renewcommand{\AA}{\mathcal{A}}
\newcommand{\BB}{\mathcal{B}}
\renewcommand{\b}{\mathrm{b}}
\newcommand{\CC}{\mathcal{C}}
\newcommand{\DD}{\mathcal{D}}
\newcommand{\D}{\mathsf{D}}
\newcommand{\E}{\mathbb{E}}
\newcommand{\EE}{\mathcal{E}}

\newcommand{\FF}{\mathcal{F}}
\newcommand{\GG}{\mathcal{G}}

\newcommand{\K}{\mathsf{K}}
\newcommand{\LL}{\mathcal{L}}
\newcommand{\M}{\mathsf{M}}
\newcommand{\N}{\mathbb{N}}
\newcommand{\NN}{\mathcal{N}}

\newcommand{\RR}{\mathcal{R}}
\renewcommand{\SS}{\mathcal{S}}
\newcommand{\s}{\mathfrak{s}}
\newcommand{\TT}{\mathcal{T}}
\newcommand{\Z}{\mathbb{Z}}

\newcommand{\Ext}{\operatorname{Ext}\nolimits}

\newcommand{\op}{\operatorname{op}\nolimits}
\newcommand{\Image}{\operatorname{Im}\nolimits}
\newcommand{\Kernel}{\operatorname{Ker}\nolimits}
\newcommand{\Cokernel}{\operatorname{Coker}\nolimits}
\newcommand{\Ab}{\mathsf{Ab}}
\newcommand{\coker}{\Cokernel}
\newcommand{\im}{\Image}
\renewcommand{\ker}{\Kernel}

\newcommand{\un}{\underline}
\newcommand{\ov}{\overline}

\DeclareMathOperator{\moduleCategory}{\mathsf{mod}} \renewcommand{\mod}{\moduleCategory}

\DeclareMathOperator{\id}{\mathsf{id}}
\DeclareMathOperator{\Iso}{\mathsf{Iso}}
\DeclareMathOperator{\gp}{\mathsf{gp}}
\DeclareMathOperator{\simp}{\mathsf{sim}}
\DeclareMathOperator{\add}{\mathsf{add}}
\DeclareMathOperator{\udim}{\underline{\dim}}

\newcommand{\iso}{\cong}

\newcommand{\defl}{\twoheadrightarrow}

\newcommand{\la}{\langle}
\newcommand{\ra}{\rangle}
\newcommand{\ot}{\leftarrow}

\newcommand{\crel}{\sim_c}
\newcommand{\ceq}{\approx_c}

\newcommand{\sst}[1]{\substack{#1}}
\newenvironment{sbmatrix}{\left[\begin{smallmatrix}}{\end{smallmatrix}\right]}

\newcommand{\ol}{\overline}

\newcommand{\wti}{\widetilde}

\newcommand{\imply}{\Rightarrow}
\newcommand{\equi}{\Leftrightarrow}

\newcommand{\xr}[1]{\xrightarrow{\, #1 \, }}
\newcommand{\inj}{\hookrightarrow}
\newcommand{\surj}{\twoheadrightarrow}
\newcommand{\simto}{\xr{\sim}}
\newcommand{\dto}{\dashrightarrow}
\newcommand{\xdr}[1]{\xdashrightarrow[]{ #1 }}

\newcommand{\bbE}{\mathbb{E}}
\newcommand{\bbN}{\mathbb{N}}
\newcommand{\bbZ}{\mathbb{Z}}

\newcommand{\calL}{\mathcal{L}}
\newcommand{\calR}{\mathcal{R}}

\newcommand{\frs}{\mathfrak{s}}

\newcommand{\Sg}{\Sigma}

\newcommand{\Mon}{\mathsf{Mon}}

\newcommand{\catmod}{\moduleCategory}
\newcommand{\catA}{\mathcal{A}}

\newcommand{\catC}{\mathcal{C}}
\newcommand{\catD}{\mathcal{D}}

\newcommand{\catF}{\mathcal{F}}
\newcommand{\catN}{\mathcal{N}}
\newcommand{\catS}{\mathcal{S}}
\newcommand{\catT}{\mathcal{T}}

\newcommand{\catX}{\mathcal{X}}

\newcommand{\Fun}{\operatorname{Fun}\nolimits}

\DeclareMathOperator{\Mor}{Mor}

\DeclareMathOperator{\Ker}{Ker}

\DeclareMathOperator{\Cone}{Cone}

\newcommand{\wpb}{\arrow[rd,"{\mathrm{wPB}}",phantom]}

\DeclareMathOperator{\Serre}{\mathsf{Serre}}
\DeclareMathOperator{\Face}{\mathsf{Face}}
\newcommand{\face}{\mathrm{face}}
\DeclareMathOperator{\Proj}{\mathsf{Proj}}

\newcommand{\ETCat}{\mathsf{ETCat}}
\DeclareMathOperator{\exa}{ex}

\newcommand{\arr}[1]{\arrow[{#1}]}
\newcommand{\gen}[1]{\la #1 \ra}
\newenvironment{bsmatrix}{\left[\begin{smallmatrix}}{\end{smallmatrix}\right]}
\newenvironment{enur}{\begin{enumerate}[label={\upshape(\roman*)}]}{\end{enumerate}}
\newenvironment{enua}{\begin{enumerate}[label={\upshape(\arabic*)}]}{\end{enumerate}}

\tikzset{
  symbol/.style={
    draw=none,
    every to/.append style={
      edge node={node [sloped, allow upside down, auto=false]{$#1$}}}
  },
  whitev/.style={circle, fill=white, draw=black, inner sep=1.5pt, outer sep=0pt}
}

\makeatletter
\newcommand*{\da@rightarrow}{\mathchar"0\hexnumber@\symAMSa 4B }
\newcommand*{\da@leftarrow}{\mathchar"0\hexnumber@\symAMSa 4C }
\newcommand*{\xdashrightarrow}[2][]{%
  \mathrel{%
    \mathpalette{\da@xarrow{#1}{#2}{}\da@rightarrow{\,}{}}{}%
  }%
}
\newcommand{\xdashleftarrow}[2][]{%
  \mathrel{%
    \mathpalette{\da@xarrow{#1}{#2}\da@leftarrow{}{}{\,}}{}%
  }%
}
\newcommand*{\da@xarrow}[7]{%
  \sbox0{$\ifx#7\scriptstyle\scriptscriptstyle\else\scriptstyle\fi#5#1#6\m@th$}%
  \sbox2{$\ifx#7\scriptstyle\scriptscriptstyle\else\scriptstyle\fi#5#2#6\m@th$}%
  \sbox4{$#7\dabar@\m@th$}%
  \dimen@=\wd0 %
  \ifdim\wd2 >\dimen@
    \dimen@=\wd2 %
  \fi
  \count@=2 %
  \def\da@bars{\dabar@\dabar@}%
  \@whiledim\count@\wd4<\dimen@\do{%
    \advance\count@\@ne
    \expandafter\def\expandafter\da@bars\expandafter{%
      \da@bars
      \dabar@ 
    }%
  }%
  \mathrel{#3}%
  \mathrel{%
    \mathop{\da@bars}\limits
    \ifx\\#1\\%
    \else
      _{\copy0}%
    \fi
    \ifx\\#2\\%
    \else
      ^{\copy2}%
    \fi
  }%
  \mathrel{#4}%
}
\makeatother
\begin{document}
\title[Grothendieck monoids of extriangulated categories]{Grothendieck monoids of extriangulated categories}

\author[H. Enomoto]{Haruhisa Enomoto}
\address{Graduate School of Science, Osaka Metropolitan University, 1-1 Gakuen-cho, Naka-ku, Sakai, Osaka 599-8531, Japan}
\email{henomoto@omu.ac.jp}

\author[S. Saito]{Shunya Saito}
\address{Graduate School of Mathematics, Nagoya University, Chikusa-ku, Nagoya. 464-8602, Japan}
\email{m19018i@math.nagoya-u.ac.jp}

\begin{abstract}
  We study the Grothendieck monoid (a monoid version of the Grothendieck group) of an extriangulated category, and give some results which are new even for abelian categories. First, we classify Serre subcategories and dense 2-out-of-3 subcategories using the Grothendieck monoid. Second, in good situations, we show that the Grothendieck monoid of the localization of an extriangulated category is isomorphic to the natural quotient monoid of the original Grothendieck monoid. This includes the cases of the Serre quotient of an abelian category and the Verdier quotient of a triangulated category.
  As a concrete example, we introduce an intermediate subcategory of the derived category of an abelian category, which lies between the abelian category and its one shift.
  We show that intermediate subcategories bijectively correspond to torsionfree classes in the abelian category, and then compute the Grothendieck monoid of an intermediate subcategory.
\end{abstract}

\maketitle

\tableofcontents

\section{Introduction}
In the representation theory of algebras, we often consider extension-closed subcategories of a triangulated category which are neither abelian nor triangulated. An \emph{extriangulated category}, introduced by Nakaoka--Palu \cite{NP}, is an appropriate framework to consider such a class of subcategories. Extriangulated categories unify both exact categories and triangulated categories, and we have the notion of \emph{conflations} in an extriangulated category, which generalize conflations (short exact sequences) in an exact category and triangles in a triangulated category.

The \emph{Grothendieck group} is the classical and basic invariant for both a triangulated category and an exact category. As a natural generalization, one can define the Grothendieck group $\K_0(\CC)$ of an extriangulated category $\CC$ as in \cite{H, ZZ}.
On the other hand, the \emph{Grothendieck monoid}, a natural monoid version of the Grothendieck group, has been studied for an exact category as a more sophisticated invariant (\cite{BeGr, JHP, Saito}) than the Grothendieck group.
The aim of this paper is to study the Grothendieck monoid of an extriangulated category.

Let $\CC$ be an extriangulated category. The Grothendieck monoid $\M(\CC)$ can be naturally defined as the universal commutative monoid equipped with a function $[-] \colon \CC \to \M(\CC)$ which is additive on each conflation. For a triangulated category, the Grothendieck monoid coincides with the Grothendieck group (Example \ref{ex:gro-mon}).
As a first use of the Grothendieck monoid, we show that \emph{Serre subcategories} of an extriangulated category $\CC$ can be classified purely combinatorially. Here a subcategory $\SS$ of $\CC$ is \emph{Serre} if for any conflation $X \to Y \to Z \dashrightarrow$ in $\CC$, we have $Y \in \SS$ if and only if both $X \in \SS$ and $Z \in \SS$.

\begin{theoremi}[= Proposition \ref{prop:serre-face-bij}]\label{thm:a}
  Let $\CC$ be a skeletally small extriangulated category. Then there is a bijection between the following two sets:
  \begin{enumerate}
    \item The set of Serre subcategories of $\CC$.
    \item The set of \emph{faces} of $\M(\CC)$
  \end{enumerate}
  Here a \emph{face $F$} of a monoid $M$ is a submonoid of $M$ such that $a + b \in F$ if and only if both $a \in F$ and $b \in F$ for any $a, b \in M$.
\end{theoremi}
This generalizes \cite{Saito} where $\CC$ is assumed to be an exact category.
We also establish a classification of \emph{dense 2-out-of-3 subcategories} of $\CC$, that is, a subcategory $\DD$ of $\CC$ satisfying $\add \DD = \CC$ and the 2-out-of-3 property for conflations.
\begin{theoremi}[= Theorem \ref{thm:dense-2-3-bij}, Corollary \ref{cor:dense-2-3-k0}]\label{thm:b}
  Let $\CC$ be a skeletally small extriangulated category. Then there is a bijection between the following sets:
  \begin{enumerate}
    \item The set of dense 2-out-of-3 subcategories of $\CC$.
    \item The set of cofinal subtractive submonoid of $\M(\CC)$ (see Definition \ref{def:subt-cofin}).
    \item The set of subgroups of $\K_0(\CC)$ such that $\rho^{-1}(H)$ is cofinal in $\M(\CC)$, where $\rho \colon \M(\CC) \to \K_0(\CC)$ is the natural map.
  \end{enumerate}
\end{theoremi}
Applying Theorem \ref{thm:b} to a triangulated category, we immediately obtain Thomason's classification \cite{thomason} of dense triangulated subcategories via subgroups of the Grothendieck group (Corollary \ref{cor:thomason}).
In \cite{matsui}, 
Matsui classified dense resolving subcategories of an exact category \emph{containing a generator} 
via certain subgroups of the Grothendieck group (see also \cite{ZZ}).
Applying Theorem \ref{thm:b} to an exact category,
we obtain a classification of \emph{all} dense resolving subcategories,
which generalizes \cite{matsui, ZZ} (Corollary \ref{cor:matsui}).

Next, we consider the behavior of the Grothendieck monoid with respect to an \emph{exact localization} of Nakaoka--Ogawa--Sakai \cite{NOS}.
An exact localization $\CC/\NN$ of an extriangulated category $\CC$ by an extension-closed subcategory $\NN$ of $\CC$ is (if exists) the universal extriangulated category which sends $\NN$ to $0$ (Proposition \ref{prop:ex-loc-univ2}). Typical examples are Serre quotients of abelian categories and Verdier quotients of triangulated categories.
Under certain assumptions, where an exact localization exists, we show that this localization ``commutes with the monoid quotients'' as follows.
\begin{theoremi}[= Corollary \ref{cor:mon loc ET}]\label{thm:c}
  Let $\CC$ be a skeletally small extriangulated category and $\NN$ an extension-closed subcategory of $\CC$. Under some conditions (see Corollary \ref{cor:mon loc ET}), we have the following coequalizer diagram of commutative monoids:
  \[
    \begin{tikzcd}
      \M(\NN) \rar[shift left, "\M(\iota)"] \rar[shift right, "0"']
      & \M(\CC) \rar["\M(Q)"] & \M(\CC/\NN),
    \end{tikzcd}
  \]
  where $\iota \colon \NN \hookrightarrow \CC$ and $Q \colon \CC \to \CC/\NN$ are the inclusion and the localization functor respectively. In particular, $\M(\CC/\NN)$ is isomorphic to the quotient monoid $\M(\CC)/\Image \M(\iota)$ (see Definition \ref{def:quot-submon}).
  This can be applied to the following cases.
  \begin{enur}
    \item $\CC$ is a triangulated category and $\NN$ is a thick subcategory of $\CC$. In this case, $\CC/\NN$ is the usual Verdier quotient of a triangulated category.
    \item $\CC$ is a Frobenius category and $\NN$ is the subcategory of all projective objects in $\CC$. In this case, $\CC/\NN$ is the usual (triangulated) stable category.
    \item $\CC$ is an abelian category and $\NN$ is a Serre subcategory of $\CC$. In this case, $\CC/\NN$ is the usual Serre quotient of an abelian category.
  \end{enur}
  Moreover, if $\NN$ is a Serre subcategory (e.g. {\upshape (iii)}), then $\M(\iota) \colon \M(\NN) \to \M(\CC)$ is an injection.
\end{theoremi}
The above coequalizer diagram immediately yields the right exact sequence of the Grothendieck group (Corollary \ref{cor:loc-gr}):
\[
  \begin{tikzcd}
    \K_0(\NN) \rar & \K_0(\CC) \rar & \K_0(\CC/\NN) \rar & 0.
  \end{tikzcd}
\]
This unifies the well-known results for the abelian case and the triangulated case. Moreover, the injectivity of $\M(\NN) \to \M(\CC)$ for a Serre subcategory $\NN$ is quite remarkable, since it fails for the Grothendieck groups (see e.g. Example \ref{ex:not inj K0}).

Next, we address a categorification of a \emph{monoid localization}, which makes certain elements of a monoid invertible (see Definition \ref{def:mon-loc}).
For this purpose,
we study \emph{intermediate subcategories} of the derived category in detail,
which also gives a concrete example of our study described above for an extriangulated category which is neither abelian nor triangulated.
Let $\AA$ be a skeletally small abelian category and $\D^\b(\AA)$ the bounded derived category of $\AA$. Then a subcategory $\CC$ of $\D^\b(\AA)$ is called an \emph{intermediate subcategory} if $\AA \subseteq \CC \subseteq \AA[1] * \AA$ holds and $\CC$ is closed under extensions and direct summands. We show in Theorem \ref{thm:inter-torf-bij} that there is a bijection between torsionfree classes $\FF$ in $\AA$ and intermediate subcategories $\CC$ of $\D^\b(\AA)$, where the corresponding intermediate subcategory is given by $\FF[1] * \AA$. This is illustrated as follows.
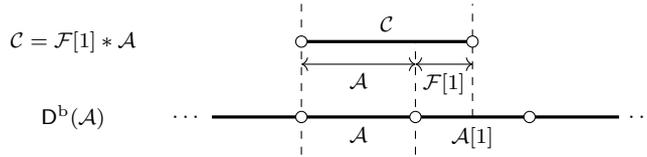
\begin{figure}[htp]
  \centering
  \begin{tikzpicture}[xscale=1.5, every node/.style={font=\footnotesize}]
    \draw[dashed] (0,-0.5) -- (0,1.5);
    \draw[dashed] (1,-0.5) -- (1,1);
    \draw[dashed] (1.5,0) -- (1.5,1.5);

    \draw[<->] (1,0.7) -- (1.5,0.7) node[midway, below] {$\FF[1]$};
    \draw[<->] (0,0.7) -- (1,0.7) node[midway, below] {$\AA$};    
    
    \node at (-2, 0) {$\D^\b(\AA)$};
    \node at (-2, 1) {$\CC = \FF[1] * \AA$};

    \node[whitev] (0) at (0,0) {}; 
    \node[whitev] (1) at (1,0) {}; 
    \node[whitev] (2) at (2,0) {};

    \node (3) at (3,0) {$\cdots$};
    \node (-1) at (-1,0) {$\cdots$};

    \node[whitev] (c0) at (0,1) {};
    \node[whitev] (c1) at (1.5,1) {};

    \draw[very thick] (0) -- (1) node[midway, below] {$\AA$};
    \draw[very thick] (1) -- (2) node[midway, below] {$\AA[1]$};
    \draw[very thick] (-1) -- (0);
    \draw[very thick] (2) -- (3);
    \draw[very thick] (c0) -- (c1) node[midway, above] {$\CC$};

  \end{tikzpicture}
  \caption{An intermediate subcategory}
  \label{fig:inter}
\end{figure}

We show that the Grothendieck monoid of an intermediate subcategory $\FF[1] * \AA$ can be described using a monoid localization (see Definition \ref{def:mon-loc}) as follows.
\begin{theoremi}[= Theorem \ref{thm:mon-of-inter}]
  Let $\AA$ be a skeletally small abelian category and $\FF$ a torsionfree class in $\AA$. Then $\M(\FF[1] * \AA)$ is isomorphic to the monoid localization $\M(\AA)_{\M_\FF}$ of $\M(\AA)$ with respect to a subset $\M_\FF := \{[F] \mid F \in \FF\} \subseteq \M(\AA)$.
\end{theoremi}
This enables us to compute $\M(\FF[1] * \AA)$ in concrete situations. Also, this can be interpreted as follows: \emph{the inclusion of a subcategory $\AA \hookrightarrow \FF[1] * \AA$ categorifies the monoid localization $\M(\AA) \to \M(\AA)_{\M_\FF}$} in the sense that by applying the functor $\M(-) \colon \ETCat \to \Mon$ we obtain the monoid localization. Moreover, we show that \emph{all the monoid localization of $\M(\AA)$ appear in this way} (Remark \ref{rem:mon-loc-categorify}).
This result explains why it is natural to study Grothendieck monoids in the generality of extriangulated categories.
The Grothendieck monoid of an exact category is reduced, that is, $0$ is the only invertible element, while that of a triangulated category is a group, that is, every element is invertible.
In particular, most of the monoid localizations of the Grothendieck monoid do not appear as the Grothendieck monoid if we only consider exact categories and triangulated categories.
Thus, to realize a localization as the Grothendieck monoid,
we have to consider extriangulated categories which are neither abelian nor triangulated.

Finally, we consider Serre subcategories of $\CC := \FF[1]*\AA$ and the behavior of an exact localization of $\CC$ in detail:
\begin{theoremi}[= Proposition \ref{prop:inter-serre-bij}, Theorem \ref{thm:inter-serre-loc}]\label{thm:d}
  Let $\AA$ be a skeletally small abelian category and $\FF$ a torsionfree class in $\AA$, and put $\CC := \FF[1] * \AA$.
  \begin{enumerate}
    \item The assignment $\SS \mapsto \FF[1] * \SS$ gives a bijection between the set of Serre subcategories $\SS$ of $\AA$ containing $\FF$ and the set of Serre subcategories of $\CC$.
    \item Let $\SS$ be a Serre subcategory of $\AA$ containing $\FF$. Then we can apply Theorem \ref{thm:c} to an exact localization $\CC/(\FF[1] * \SS)$, and we have an exact equivalence $\AA/\SS \simeq \CC/(\FF[1] * \SS)$ which makes the following diagram commute:
    \begin{equation}\label{eq:intro}
      \begin{tikzcd}
        \AA \rar[hookrightarrow] \dar & \FF[1] * \AA \dar \\
        \AA / \SS \rar["\sim"] & (\FF[1] * \AA)\big / (\FF[1] * \SS) .
      \end{tikzcd}
    \end{equation}
  \end{enumerate}
\end{theoremi}
\begin{examplei}\label{ex:intro}
  Let $k$ be a field and $Q$ a quiver $1 \ot 2 \ot 3$, and let $\AA := \mod kQ$. Then the Auslander-Reiten quiver of $\D^\b(\AA)$ is as follows:
  
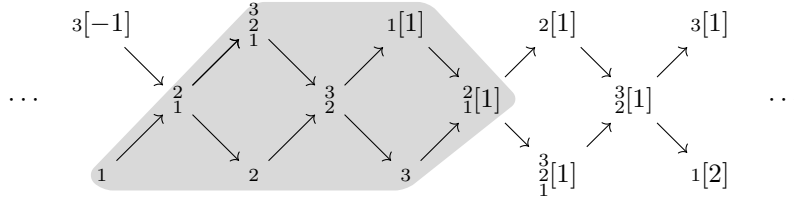
\[
    \begin{tikzpicture}
      \fill[gray!30, rounded corners]
      (0,-.2) -- (-.2,0) -- (2,2.3) -- (4.3,2.3) -- (5.5,1) -- (4,-.2) -- cycle;
      \node at (-1,1) {$\cdots$};
      \node (3m1) at (0,2) {$\sst{3}[-1]$};
      \node (1) at (0,0) {$\sst{1}$};
      \node (21) at (1,1) {$\sst{2 \\ 1}$};
      \node (321) at (2,2) {$\sst{3 \\ 2 \\ 1}$};
      \node (2) at (2,0) {$\sst{2}$};
      \node (32) at (3,1) {$\sst{3\\2}$};
      \node (3) at (4,0) {$\sst{3}$};
      \node (1p1) at (4,2) {$\sst{1}[1]$};
      \node (21p1) at (5,1) {$\sst{2\\1}[1]$};
      \node (321p1) at (6,0) {$\sst{3\\2\\1}[1]$};
      \node (2p1) at (6,2) {$\sst{2}$[1]};
      \node (32p1) at (7,1) {$\sst{3\\2}$[1]};
      \node (3p1) at (8,2) {$\sst{3}$[1]};
      \node (1p2) at (8,0) {$\sst{1}$[2]};
      \node at (9,1) {$\cdots$};

      \draw[->] (3m1) -- (21);
      \draw[->] (1) -- (21);
      \draw[->] (21) -- (321);
      \draw[->] (21) -- (2);
      \draw[->] (321) -- (32);
      \draw[->] (2) -- (32);
      \draw[->] (32) -- (3);
      \draw[->] (21) -- (321);
      \draw[->] (32) -- (1p1);
      \draw[->] (3) -- (21p1);
      \draw[->] (1p1) -- (21p1);
      \draw[->] (21p1) -- (2p1);
      \draw[->] (21p1) -- (321p1);
      \draw[->] (2p1) -- (32p1);
      \draw[->] (321p1) -- (32p1);
      \draw[->] (32p1) -- (3p1);
      \draw[->] (32p1) -- (1p2);
    \end{tikzpicture}
  \]
  Consider the torsionfree class $\FF := \add \{ \sst{1}, \sst{2\\1}\}$ in $\AA$. Then an intermediate subcategory $\CC := \FF[1]* \AA$ is the additive closure of the gray region.
  By computing the monoid localization, we obtain $\M(\CC) \iso \M(\AA)_{\M_\FF} \iso \Z \oplus \Z \oplus \N$ (see Example \ref{ex:toy} for details).
  There are exactly two Serre subcategories of $\AA$ containing $\FF$, namely, $\SS := \add \{ \sst{1}, \sst{2\\1}, \sst{2}\}$ and $\AA$. Therefore, $\CC$ has exactly two Serre subcategories $\FF[1] * \SS$ and $\CC$. For example, if we consider $\SS$, we have an exact equivalence $\AA/\SS \simeq \CC/ (\FF[1] * \SS)$, and by taking the Grothendieck monoid of \eqref{eq:intro}, we obtain the following commutative diagram of monoids:
  \[
    \begin{tikzcd}
      \N \oplus \N \oplus \N \rar[hookrightarrow] \dar[twoheadrightarrow] & \Z \oplus \Z \oplus \N \dar[twoheadrightarrow] \\
      \frac{\N \oplus \N \oplus \N}{\N \oplus \N \oplus 0} \iso \N \rar["\sim"] & \frac{\Z \oplus \Z \oplus \N}{\Z \oplus \Z \oplus 0} \iso \N.
    \end{tikzcd}
  \]
\end{examplei}

\medskip
\noindent
{\bf Organization.}
In Section \ref{sec:2}, we define and construct the Grothendieck monoid of an extriangulated category.
In Section \ref{sec:3}, we give classifications of some classes of subcategories of an extriangulated category using its Grothendieck monoid, including Theorems \ref{thm:a} and \ref{thm:b}.
In Section \ref{sec:4}, we discuss the relation between the exact localization of an extriangulated category and the monoid quotient, and prove Theorem \ref{thm:c}.
In Section \ref{sec:5}, we study intermediate subcategories of the derived category of an abelian category.
In Section \ref{sec:6}, we give some questions about the Grothendieck monoid.
In Appendix \ref{app:mon}, we collect basic facts about commutative monoids which we shall use in this paper.

\medskip
\noindent
{\bf Conventions.}
Throughout this paper, \emph{all subcategories are full and closed under isomorphisms}. By $A \iso B$ for two objects $A$ and $B$ in a category $\CC$, we mean that $A$ and $B$ are isomorphic in $\CC$.

For a skeletally small category $\CC$, we denote by $\Iso \CC$ the set of all isomorphism classes of objects in $\CC$. For an object $X$ in $\CC$, the isomorphism class of $X$ is denoted by $[X] \in \Iso\CC$.

For a subcategory $\DD$ of an additive category $\CC$, we denote by $\add \DD$ the subcategory of $\CC$ consisting of direct summands of direct sums of objects in $\DD$. This is the smallest additive subcategory of $\CC$ closed under direct summands containing $\DD$.

We denote by $\Mor \catC$ the class of all morphisms in $\CC$.
For a subclass $S\subseteq \Mor \catC$ and objects $X,Y\in\catC$,
we denote by $S(X,Y)$ the class of all morphisms in $S$ from $X$ to $Y$.
We simply write $\CC(X,Y) := (\Mor\CC)(X,Y)$.

By a \emph{monoid} $M$ we mean a \emph{commutative} monoid with a unit, and we always use an additive notation: the operation is denoted by $+$ with its unit $0$. A \emph{homomorphism} between monoids is a map which preserves both the addition and $0$. We denote by $\N$ the monoid of non-negative integers.
For a subset $S$ of a monoid $M$, we denote by $\la S \ra_\N$ the submonoid of $M$ generated by $S$. Similarly, $\la S \ra_\Z$ denotes the subgroup of $M$ generated by $S$ if $M$ is an abelian group.

We denote by $\Mon$ the category of (commutative) monoids and monoid homomorphisms, and by $\Ab$ the category of abelian groups and group homomorphisms.

For a noetherian ring $\Lambda$, we denote by $\mod\Lambda$ the abelian category of finitely generated right $\Lambda$-modules.

\section{The Grothendieck monoid of an extriangulated category}\label{sec:2}
In this section, we first give basic definitions and properties about extriangulated categories. Next, we define the Grothendieck monoid $\M(\CC)$ of an extriangulated category $\CC$ by the universal property, and then give a construction of $\M(\CC)$.

\subsection{Preliminaries on extriangulated categories}
In this paper, we omit the precise definitions and axioms of extriangulated categories. We refer the reader to \cite{NP,LN} for the basics of extriangulated categories.

An extriangulated category $(\CC, \E, \s)$ consists of the following data satisfying certain axioms:
\begin{itemize}
  \item $\CC$ is an additive category.
  \item $\E$ is an additive bifunctor $\E \colon \CC^{\op} \times \CC \to \Ab$.
  \item $\s$ is a correspondence which associates to each $\delta \in \E(Z,X)$ an equivalence class of complexes $[X \to Y \to Z]$ in $\CC$.
\end{itemize}
Here two complexes $[X \xrightarrow{f} Y \xrightarrow{g} Z]$ and $[X \xrightarrow{f'} Y' \xrightarrow{g'} Z ]$ are \emph{equivalent} if there is an isomorphism $y \colon Y \xrightarrow{\sim} Y'$ which makes the following diagram commute:
\[
\begin{tikzcd}
  X \dar[equal]\rar["f"] & Y \rar["y"] \dar["y"', "\sim" sloped] & Z \dar[equal]\\
  X \rar["f'"'] & Y' \rar["g'"'] & Z\rlap{.}
\end{tikzcd}
\]
For an extriangulated category $(\CC, \E,\s)$, we call a complex $X \xrightarrow{f} Y \xrightarrow{g} Z$ a \emph{conflation} if its equivalence class is equal to $\s(\delta)$ for some $\delta \in \E(Z,X)$. In this case, we call $f$ an \emph{inflation} and $g$ a \emph{deflation}, and say that this complex \emph{realizes} $\delta$. We write this situation as follows:
\[
\begin{tikzcd}
  X \rar["f"] & Y \rar["g"] & Z \rar[dashed, "\delta"] & ,
\end{tikzcd}
\]
and we also call this diagram a conflation in $\CC$.
In what follows, we often write $\CC = (\CC, \E, \s)$ for an extriangulated category $(\CC,\E,\s)$ for simplicity.

Triangulated categories and Quillen's exact categories can be naturally considered as extriangulated categories as follows:
\begin{enumerate}
  \item Let $\TT$ be a triangulated category with shift functor $\Sigma$. Then by setting $\E(Z, X) := \TT(Z, \Sigma X)$, we may regard $\TT$ as an extriangulated category. In this case, for each $h \in \E(Z,X) = \TT(Z, \Sigma X)$, its realization $\s(h)$ is given by $[X \xr{f}Y \xr{g} Z]$ which makes $X \xr{f} Y \xr{g} Z \xr{h} \Sigma X$ a triangle in $\TT$.
  \item Let $\EE$ be a skeletally small exact category. We have a functor $\Ext_\EE^1(-,-) : \EE^{\op} \times \EE \to \Ab$, where $\Ext^1_\EE(Z, X)$ is the set of equivalence classes of conflations (in the sense of exact categories) in $\EE$ of the form $0 \to X \to Y \to Z \to 0$. Then by setting $\E := \Ext^1_\EE$ and $\s := \id$, we may regard $\EE$ as an extriangulated category.
\end{enumerate}
\emph{Throughout this paper, we always regard triangulated categories and exact categories (in particular, abelian categories) as extriangulated categories in this way.}

For later use, we fix some terminologies for extriangulated categories.
Let $\catC = (\CC, \E, \s)$ be an extriangulated category.
For any $\delta \in \E(Z,X)$ and any morphisms $x\in \CC(X, X')$ and $z\in \CC(Z', Z)$,
we denote $\E(Z,x)(\delta)$ and $\E(z,X)(\delta)$
briefly by $x_*\delta$ and $z^*\delta$ respectively.

A \emph{morphism of conflations} from $X\xr{f} Y\xr{g} Z \xdr{\delta}$ to $X'\xr{f'} Y'\xr{g'} Z' \xdr{\delta'}$
is a triplet $\left(x \colon X\to X', y \colon Y\to Y', z \colon Z\to Z'\right)$ of morphisms in $\catC$
satisfying $yf=f'x$, $zg=g'y$ and $x_*\delta = z^*\delta'$.
We often write a morphism of conflations $(x,y,z)$ as follows:
\[
\begin{tikzcd}
X \arr{r,"f"} \arr{d,"x"'} & Y \arr{r,"g"} \arr{d,"y"} & Z \arr{r,"\delta",dashed} \arr{d,"z"} & {}\\
X' \arr{r,"f'"'} & Y \arr{r,"g'"'} & Z' \arr{r,"\delta'"',dashed} & {}.
\end{tikzcd}
\]

For two subcategories $\CC_1$ and $\CC_2$ of an extriangulated category $\CC$, we denote by $\CC_1 * \CC_2$ the subcategory of $\CC$ consisting of $X \in \CC$ such that there is a conflation
\[
  \begin{tikzcd}
    C_1 \rar & X \rar & C_2 \rar[dashed] & \ 
  \end{tikzcd}
\]
in $\CC$ with $C_1\in \CC_1$ and $C_2 \in \CC_2$.

We recall some properties of a subcategory $\catD$ of an extriangulated category $\catC$.
\begin{itemize}
\item 
$\catD$ is \emph{closed under finite direct sums} if it contains a zero object of $\catC$
and for any objects $X,Y\in \catD$, 
their direct sum $X\oplus Y$ in $\catC$ belongs to $\catD$.
We also say that $\catD$ is an \emph{additive subcategory} in this case.
\item 
$\catD$ is \emph{closed under direct summands} if $X\oplus Y \in\catD$ implies $X$, $Y\in\catD$ for any objects $X,Y\in \catC$.
\item 
$\catD$ is \emph{closed under extensions} if $\DD$ is additive and $\DD * \DD \subseteq \DD$ holds,
that is, for any conflation $X \to Y \to Z \dto$ in $\catC$,
we have that $X,Z\in\catD$ implies $Y\in\catD$.
We also say that $\catD$ is an \emph{extension-closed subcategory} in this case.
\end{itemize}
An extension-closed subcategory $\catD$ of an extriangulated category $\catC$
also has the structure of an extriangulated category induced from that of $\CC$,
see \cite[Remark 2.18]{NP}. In this structure, a conflation $X \to Y \to Z \dto$ in $\DD$ is precisely the conflation in $\CC$ with $X$, $Y$, and $Z$ in $\DD$.
When we consider extension-closed subcategories of an extriangulated category, we always regard $\DD$ as an extriangulated category in this way.

One of the main topics of this paper is the notion of \emph{Serre subcategories}, which is a natural generalization of that in abelian categories.
\begin{definition}
  Let $\CC$ be an extriangulated category and $\catD$ an additive subcategory of $\CC$. Then $\DD$ is called a \emph{Serre subcategory} if for any conflation
  \[
    \begin{tikzcd}
      X \rar & Y \rar & Z \rar[dashed] & \ 
    \end{tikzcd}
  \]
  in $\CC$, we have that both $X$ and $Z$ are in $\catD$ if and only if $Y\in\catD$.
\end{definition}
In particular, a Serre subcategory is extension-closed, so it can be regarded as an extriangulated category in itself.

We also need the natural notion of functors between extriangulated categories, namely, an \emph{exact functor}:
\begin{definition}[{\cite[Definition 2.32]{BS}}]
  Let $(\CC_i, \E_i, \s_i)$ be an extriangulated category for $i=1,2,3$.
\begin{enua}
\item
An \emph{exact functor $(F,\phi) \colon (\CC_1, \E_1, \s_1) \to (\CC_2, \E_2, \s_2)$} is a pair of an additive functor $F \colon \CC_1 \to \CC_2$ and a natural transformation $\phi \colon \E_1 \Rightarrow \E_2 \circ (F^{\op} \times F)$ such that for every conflation
  \[
    \begin{tikzcd}
      X \rar["f"] & Y \rar["g"] & Z \rar[dashed, "\delta"] & \ 
    \end{tikzcd}
  \]
  in $\CC_1$, the following is a conflation in $\CC_2$:
  \[
    \begin{tikzcd}
      FX \rar["Ff"] & FY \rar["Fg"] & FZ \rar[dashed, "\phi_{Z, X}(\delta)"] & \ .
    \end{tikzcd}
  \]
We often write $F = (F,\phi) \colon \catC_1 \to \catC_2$ in this case.
\item
For two exact functors $F_1=(F_1,\phi_1)\colon \catC_1 \to \catC_2$ 
and $F_2=(F_2, \phi_2)\colon \catC_2 \to \catC_3$,
their composition is defined by 
$F_2\circ F_1 = (F_2\circ F_1, \phi_2\cdot \phi_1) : \catC_1 \to \catC_3$,
where 
\[
(\phi_2\cdot \phi_1)_{C,A}\colon \bbE_1(C,A) \xr{(\phi_1)_{C,A}} \bbE_2(FC,FA) 
\xr{(\phi_2)_{FC,FA}} \bbE_3(GFC,GFA).
\]
\item
Let $F=(F,\phi)$ and $G=(G,\psi)$ be two exact functors $\catC_1 \to \catC_2$.
A \emph{natural transformation} $\eta\colon F \Rightarrow G$ of \emph{exact functors}
is a natural transformation of additive functors
such that, for any conflation $A \xr{x} B \xr{y} C \xdr{\delta}$ in $\catC$,
the following is a morphism of conflations in $\DD$:
\[
\begin{tikzcd}
FA \arr{r,"Fx"} \arr{d,"\eta_A"'} & FB \arr{r,"Fy"} \arr{d,"\eta_B"'} & FC \arr{r,"\phi_{C,A}(\delta)"} \arr{d,"\eta_C"',dashed} &{}\\
GA \arr{r,"Gx"'} & GB \arr{r,"Gy"'} & GC \arr{r,"\psi_{C,A}(\delta)"',dashed} &{}.
\end{tikzcd}
\]
\item
$\ETCat$ is the category of skeletally small extriangulated categories and exact functors.
\end{enua}  
\end{definition}
\emph{Horizontal compositions} and \emph{vertical compositions}
of natural transformations of exact functors are defined
by those for natural transformations of additive functors.
The compositions are also natural transformations of exact functors.
In particular,
we can define an \emph{exact equivalence} of extriangulated categories
in the same way as that of additive categories.

\begin{rmk}\label{rmk:ex funct}
Let $F=(F,\phi)\colon \catC \to \catD$ be an exact functor of extriangulated categories,
and let $(a,b,c)$ be a morphism of conflations
\[
\begin{tikzcd}
A_1 \arr{r,"x_1"} \arr{d,"a"'} & B_1 \arr{r,"y_1"} \arr{d,"b"'} & C_1 \arr{r,"\delta_1",dashed} \arr{d,"c"'} & {}\\
A_2 \arr{r,"x_2"'} & B_2 \arr{r,"y_2"'} & C_2 \arr{r,"\delta_2"',dashed} & {}
\end{tikzcd}
\]
in $\catC$.
Then $(Fa,Fb,Fc)$ gives a morphism of conflations
\[
\begin{tikzcd}
FA_1 \arr{r,"Fx_1"} \arr{d,"Fa"'} & FB_1 \arr{r,"Fy_1"} \arr{d,"Fb"'} & FC_1 \arr{r,"\phi_{C_1,A_1}(\delta_1)",dashed} \arr{d,"Fc"'} &[1cm]{}\\
FA_2 \arr{r,"Fx_2"'} & FB_2 \arr{r,"Fy_2"'} & FC_2 \arr{r,"\phi_{C_2,A_2}(\delta_2)"',dashed} &{}
\end{tikzcd}
\]
in $\catD$ by the naturality of $\phi$.
\end{rmk}

Later, we need the following observation about exact functors to exact categories, which can be easily proved by considering the extriangulated structure on exact categories.
\begin{lemma}\label{lem:exact-func-to-exact-cat}
  Let $\CC = (\CC, \E, \s)$ be an extriangulated category, $\EE$ an exact category, and $F \colon \CC \to \EE$ an additive functor. Then the following are equivalent.
  \begin{enumerate}
    \item There is some $\phi$ which makes $(F, \phi)$ an exact functor between extriangulated categories.
    \item For each conflation $X \xr{f} Y \xr{g} Z \dto $ in $\CC$, the following is a conflation in $\EE$:
    \[
      \begin{tikzcd}
        0 \rar & FX \rar["Ff"] & FY \rar["Fg"] & FZ \rar & 0.
      \end{tikzcd}
    \]
  \end{enumerate}
  Moreover, in this case, $\phi$ in {\upshape (1)} is uniquely determined.
\end{lemma}
Due to this fact, we call an additive functor $F \colon \CC \to \EE$ satisfying (2) above just an \emph{exact functor}.

\subsection{Definition and construction of Grothendieck monoid}
For a triangulated category or an exact category $\CC$, its \emph{Grothendieck group $\K_0(\CC)$} is a basic invariant and has been studied and used in various areas. Recently, the \emph{Grothendieck monoid of an exact category}, a natural monoid version of the Grothendieck group, has been introduced by Berenstein--Greenstein in \cite{BeGr} to study Hall algebras, and its relation to the categorical property of an exact category was studied in \cite{JHP}.

Let $\CC$ be an extriangulated category. We can naturally generalize these constructions to an extriangulated category $\CC$. The Grothendieck group $\K_0(\CC)$ was formulated by Zhu--Zhuang \cite{ZZ} and Haugland \cite{H} and was used to classify certain subcategories of an extriangulated category. The aim of this paper is to define and investigate the \emph{Grothendieck monoid $\M(\CC)$ of $\CC$}.

First, we define the Grothendieck monoid by the universal property. Recall that all monoids are commutative.

\begin{definition}\label{def:grmon}
  Let $\CC$ be a skeletally small extriangulated category.
  \begin{enumerate}
    \item We say that a map $f \colon \Iso\CC \to M$ for a monoid $M$ \emph{respects conflations} if it satisfies the following conditions:
   \begin{enumerate}
     \item $f[0] = 0$ holds.
     \item For every conflation
     \[
     \begin{tikzcd}
      X \rar & Y \rar & Z \rar[dashed] & \
     \end{tikzcd}
     \]
     in $\CC$, we have $f[Y] = f[X] + f[Z]$ in $M$.
   \end{enumerate}
   \item A \emph{Grothendieck monoid} $\M(\CC)$ is a monoid $\M(\CC)$ together with a map $\pi \colon \Iso\CC \to \M(\CC)$ which satisfies the following universal property:
   \begin{enumerate}
     \item $\pi$ respects conflations in $\CC$.
     \item Every map $f\colon \Iso\CC \to M$ for a monoid $M$ which respects conflations in $\CC$ uniquely factors through $\pi$, that is, there exists a unique monoid homomorphism $\ov{f} \colon \M(\CC) \to M$ which satisfies $f = \ov{f} \pi$.
     \[
     \begin{tikzcd}
       \Iso\CC \dar["\pi"'] \rar["f"] & M \\
       \M(\CC) \ar[ru, "\ov{f}"', dashed]
     \end{tikzcd}
     \]
   \end{enumerate}
   By abuse of notation, we often write $\pi[X] =: [X] \in \M(\CC)$ for $X \in \CC$.
  \end{enumerate}
\end{definition}
Since the Grothendieck monoid of $\CC$ is defined by the universal property, we must show that it actually exists. For later use, we give an explicit construction of $\M(\CC)$.

\begin{definition}
  Let $\CC$ be an extriangulated category and $A$ and $B$ objects in $\CC$.
  \begin{enumerate}
    \item $A$ and $B$ are \emph{conflation-related}, abbreviated by \emph{c-related}, if there is a conflation
    \[
      \begin{tikzcd}
        X \rar & Y \rar & Z \rar[dashed] & \
      \end{tikzcd}
    \]
    such that either of the following holds in $\CC$:
    \begin{enumerate}
      \item $Y \iso A$ and $X \oplus Z \iso B$, or
      \item $Y \iso B$ and $X \oplus Z \iso A$.
    \end{enumerate}
    In this case, we write $A \crel B$.
    \item $A$ and $B$ in $\CC$ are \emph{conflation-equivalent}, abbreviated by \emph{c-equivalent}, if there are objects $A_0, A_1, \dots, A_n$ in $\CC$ for some $n \geq 0$ satisfying $A_0 \iso A$, $A_n \iso B$, and $A_i \crel A_{i+1}$ for each $i$. In this case, we write $A \ceq B$.
  \end{enumerate}
\end{definition}
It is immediate that $\crel$ and $\ceq$ induce binary relations $\crel$ and $\ceq$ on the set $\Iso\CC$ by putting $[A] \crel [B]$ (resp. $[A] \ceq [B]$) if $A \crel B$ (resp. $A \ceq B$). Then clearly $\ceq$ is an equivalence relation on $\Iso\CC$ generated by $\crel$.
On the other hand, $\Iso\CC$ can be regarded as a monoid with the addition given by $[A] + [B] = [A \oplus B]$. Then we can state the following explicit construction of the Grothendieck monoid.
\begin{proposition}\label{prop:gro-mon-ceq}
  Let $\CC = (\CC, \E, \s)$ be a skeletally small extriangulated category. Then the following hold.
  \begin{enumerate}
    \item $\ceq$ is a monoid congruence on $\Iso\CC$ (see Definition \ref{def:mon-congr} for a monoid congruence).
    \item The quotient monoid $\Iso\CC/{\ceq}$ together with the natural projection $\pi \colon \Iso\CC \defl \Iso\CC/{\ceq}$ gives a Grothendieck monoid of $\CC$.
  \end{enumerate}
\end{proposition}
\begin{proof}
  (1)
  We must show that $[A] \ceq [B]$ implies $[A] + [C] \ceq [B] + [C]$ holds in $\Iso\CC$ for every $A$, $B$, and $C$ in $\CC$. By the definition of $\ceq$, it clearly suffices to show that $[A] \crel [B]$ implies $[A] + [C] \crel [B] + [C]$, so suppose that $[A] \crel [B]$ holds.
  Then there is a conflation
  \[
    \begin{tikzcd}
      X \rar["f"] & Y \rar["g"] & Z \rar[dashed] & \
    \end{tikzcd}
  \]
  such that either (a) $Y \iso A$ and $X \oplus Z \iso B$ or (b) $Y \iso B$ and $X \oplus Z \iso A$. On the other hand, we have the following conflation:
  \[
    \begin{tikzcd}
      C \rar[equal] & C \rar & 0 \rar[dashed] & \
    \end{tikzcd}
  \]
  since $\s$ is an additive realization \cite[Definition 2.10]{NP}. Moreover, since conflations are closed under finite direct sums (because $\s$ is an additive realization), we obtain the following conflation:
  \[
    \begin{tikzcd}
      X \oplus C \rar["f \oplus \id_C"] & Y \oplus C \rar["{[g, \, 0]}"] & Z \rar[dashed] & \
    \end{tikzcd}
  \]
  Therefore, we obtain $Y \oplus C \crel (X \oplus C) \oplus Z \iso  (X \oplus Z) \oplus C$. This implies $[Y] + [C] \crel [X \oplus Z] + [C]$ in $\Iso\CC$. Therefore, in either case, we obtain $[A] + [C] \crel [B] + [C]$.

  (2)
  By (1), we obtain the quotient monoid $\Iso\CC / {\ceq}$ and the natural monoid homomorphism $\pi \colon \Iso\CC \defl \Iso\CC / {\ceq}$.
  We first show that $\pi \colon \Iso\CC \defl \Iso\CC/{\ceq}$ respects conflations. Since $\pi$ preserves the additive unit, we have $\pi[0] = 0$. Let $X \to Y \to Z \dashrightarrow$ be a conflation in $\CC$. Then we have $Y \crel X \oplus Z$.
  It follows that $[Y] \crel [X \oplus Z] = [X] + [Z]$ holds in $\Iso\CC$,
  and hence we obtain $\pi[Y] = \pi[X] + \pi[Z]$ in $\Iso\CC/{\ceq}$.

  Next, suppose that we have a map $f \colon \Iso\CC \to M$ for a monoid $M$ which respects conflations, and we will show that there is a unique monoid homomorphism $\ov{f} \colon \Iso\CC/{\ceq} \to M$ satisfying $f = \ov{f} \pi$. The uniqueness is clear since $\pi$ is surjective, so we only show the existence of such $\ov{f}$.
  To obtain a well-defined map $\ov{f} \colon \Iso\CC/{\ceq} \to M$ satisfying $f = \ov{f} \pi$, it clearly suffices to show that $[A] \crel [B]$ in $\Iso\CC$ implies $f[A] = f[B]$, so suppose $[A] \crel [B]$. Then there is a conflation $X \to Y \to Z \dashrightarrow$ such that either (a) $Y \iso A$ and $X \oplus Z \iso B$ or (b) $Y \iso B$ and $X \oplus Z \iso A$. On the other hand, since $f$ respects conflations, we have $f[Y] = f[X] + f[Z]$. Moreover, we have the following split conflation
  \[
    \begin{tikzcd}
      X \rar["{\begin{sbmatrix} \id_X \\ 0\end{sbmatrix}}"] & X \oplus Z \rar["{[0,\, \id_Z]}"] & Z \rar[dashed, "0"] & \
    \end{tikzcd}
  \]
  because $\s$ is an additive realization. Therefore, $f[X] + f[Z] = f[X \oplus Z]$ holds, and hence $f[Y] = f[X \oplus Z]$. Thus $f[A] = f[B]$ holds in either case.
\end{proof}
In what follows, we often identify $\M(\CC)$ with $\Iso\CC/{\ceq}$ and write $[X] \in \Iso\CC/{\ceq}$ to represent an element for $X \in \CC$.

The assignment $\CC \mapsto \M(\CC)$ actually gives a functor:
\begin{proposition}\label{prp:ex fun hom}
  We have a functor $\M(-) \colon \ETCat \to \Mon$ defined as follows.
  \begin{itemize}
    \item To $\CC\in \ETCat$, we associate the Grothendieck monoid $\M(\CC) \in \Mon$.
    \item To an exact functor $F \colon \CC \to \DD$, we associate a monoid homomorphism $\M(F) \colon \M(\CC) \to \M(\DD)$ defined by $[C] \mapsto [F(C)]$.
  \end{itemize}
\end{proposition}
Since the proof is straightforward by using the defining universal property of the Grothendieck monoid, we omit it.

\begin{remark}\label{rem:k0-as-gp}
  The Grothendieck group $\K_0(\CC)$ of an extriangulated category can be defined similarly by the universal property (see also \cite{H,ZZ}). The relation between $\M(\CC)$ and $\K_0(\CC)$ is as follows.
  To any commutative monoid $M$, we can associate its \emph{group completion $\gp M$}, which is the universal abelian group together with the monoid homomorphism $M \to \gp M$ (see Definition \ref{def:mon-loc} and Proposition \ref{prop:gp-compl-univ}).
  Then the defining properties of $\M(\CC)$ and $\K_0(\CC)$ immediately show that $\K_0(\CC)$ is isomorphic to $\gp\M(\CC)$. Moreover, the group completion defines a functor $\gp \colon \Mon \to \Ab$, and we have a natural isomorphism of functors $\K_0(-) \simeq \gp \circ \M(-) \colon \ETCat \to \Ab$.
\end{remark}

\begin{example}\label{ex:gro-mon}
  The following are typical examples of the computation of the Grothendieck monoid. For more examples, we refer the reader to Section \ref{sec:5}.
  \begin{enumerate}
    \item Let $\EE$ be an exact category. Then our $\M(\EE)$ (as an extriangulated category) coincides with the Grothendieck monoid of an exact category which is studied in \cite{BeGr, JHP}.
    For example, \cite[Theorem 4.12]{JHP} implies that $\M(\EE)$ is a free monoid if and only if $\EE$ satisfies the Jordan-H\"older-like property. In particular, if $\EE$ is an abelian length category with $n$ simple objects, then $\M(\EE) \cong \N^n$ holds.
    \item Let $\TT$ be a triangulated category. Then \emph{$\M(\TT)$ becomes a group, and hence $\M(\TT)$ and $\K_0(\TT)$ coincide}. Indeed, for every element $[X] \in \M(\TT)$, there is a triangle
    \[
      \begin{tikzcd}
        X \rar & 0 \rar & \Sigma X \rar[equal] & \Sigma X.
      \end{tikzcd}
    \]
    This means that we have a conflation $X \to  0 \to \Sigma X \dashrightarrow$, which implies $[X] + [\Sigma X] = 0$ in $\M(\TT)$. Therefore, every element in $\M(\TT)$ is invertible, that is, $\M(\TT)$ is a group.
    In general, the converse do not hold: there is an extriangulated category which is not triangulated but whose Grothendieck monoid is a group (see Corollary \ref{cor:a1-a-k0} for example).
  \end{enumerate}
\end{example}

\section{Classifying subcategories via Grothendieck monoid}\label{sec:3}
Throughout this section, $\catC=(\catC, \bbE, \frs)$ is a skeletally small extriangulated category.
In this section, we give classifications of several subcategories of $\CC$ via its Grothendieck monoid $\M(\CC)$. More precisely, we consider the following assignments:
\begin{itemize}
  \item For a subcategory $\catD$ of $\catC$,
  we define a subset $\M_\DD$ of $\M(\catC)$ by
  \[
  \M_{\catD}:=\{[D]\in\M(\catC) \mid D\in \catD  \}.
  \]
  \item For a subset $N$ of $\M(\catC)$,
  we define a subcategory $\DD_N$ of $\catC$ by
  \[
  \catD_{N}:=\{X\in\catC \mid [X]\in N  \}.
  \]
\end{itemize}
In what follows, we will show that these maps give bijections between certain subcategories of $\CC$ and certain subsets of $\M(\CC)$.
We freely identify $\M(\CC)$ with $\Iso\CC/{\ceq}$ by Proposition \ref{prop:gro-mon-ceq}.

\subsection{Serre subcategories and faces}\label{s:Serre face}
In this subsection, we establish a bijection between the set of Serre subcategories of $\catC$
and the set of faces of $\M(\catC)$.
Most results in this subsection are direct generalizations of \cite{Saito}.
However, we give more sophisticated proofs using c-equivalences 
and new results such as Proposition \ref{prp:Serre c-eq}.
\begin{dfn}
A subcategory $\catD$ of $\catC$ is said to be \emph{closed under c-equivalences}
if for any conflation $A \to B \to C \dto$,
we have that $B\in \catD$ if and only if $A\oplus C \in \catD$.
We also say that $\catD$ is a \emph{c-closed subcategory} for short.
\end{dfn}

The following lemma follows immediately from Proposition \ref{prop:gro-mon-ceq}. We freely use this characterization in what follows.
\begin{lem}\label{lem:c-closed-char}
The following are equivalent for a subcategory $\catD$ of $\catC$.
\begin{enua}
\item
$\catD$ is closed under c-equivalences.
\item
For every $X, Y \in \CC$, if $X \crel Y$ holds and $X$ belongs to $\DD$,
then $Y$ also belongs to $\DD$.
\item
For every $X, Y \in \CC$, if $[X] = [Y]$ holds in $\M(\CC)$ and $X$ belongs to $\DD$,
then $Y$ also belongs to $\DD$.
\end{enua}
\end{lem}

The relation between the subcategories $\catD_N$ defined above 
and c-closed subcategories is as follows.
Note that a subcategory is not assumed to be additive.
\begin{prp}\label{prp:c-closed bij}
The following hold for a skeletally small extriangulated category $\CC$.
\begin{enua}
\item
$\catD_{N}$ is closed under c-equivalences for any subset $N$ of $\M(\catC)$.
\item
The assignments $\catD \mapsto \M_{\catD}$ and $N\mapsto \catD_N$ give
mutually inverse bijections between
the set of c-closed subcategories of $\CC$ and the power set of $\M(\CC)$.
\end{enua}
\end{prp}
\begin{proof}
(1)
Let $X\in\catD_N$ and $Y\in \catC$, and suppose that $[X] = [Y]$ in $\M(\CC)$.
Then we have $[Y] = [X] \in N$, and hence $Y$ also belongs to $\catD_N$.
This proves that $\catD_{N}$ is closed under c-equivalences.

(2)
First, we prove that $\M_{\catD_N}=N$ holds for each subset $N$ of $\M(\CC)$.
Clearly, we have $\M_{\catD_N}\supseteq N$.
Take $[X] \in \M_{\catD_N}$.
Then there is an object $Y\in \catD_N$ such that $[X]=[Y]$ holds in $\M(\CC)$.
Since $\catD_N$ is closed under c-equivalences by (1),
we have that $X$ is also in $\catD_N$, that is, $[X]\in N$.
Hence, we obtain $\M_{\catD_N}=N$.

Next, we prove that $\catD_{\M_{\catD}}=\catD$ holds
for a c-closed subcategory $\catD$ of $\catC$.
Clearly, we have $\catD_{\M_{\catD}} \supseteq \catD$.
Take $X\in \catD_{\M_{\catD}}$, then we have $[X]\in \M_\catD$,
so there exists an object $Y \in \catD$ satisfying $[X]=[Y]$ in $\M(\CC)$.
Since $\catD$ is closed under c-equivalences, we obtain $X\in\catD$.
Therefore, we obtain $\catD_{\M_{\catD}}=\catD$, which completes the proof.
\end{proof}

Now we will consider \emph{Serre subcategories} of an extriangulated category.
Recall that a subcategory $\catD$ of $\catC$ is called a \emph{Serre subcategory} 
if it is additive and for any conflation $X \to Y \to Z \dto$ in $\catC$,
we have that $X,Z\in\catD$ if and only if $Y\in\catD$.
We denote by $\Serre \CC$ the set of Serre subcategories of $\CC$.

The following characterizes Serre subcategories using c-equivalences.
\begin{prp}\label{prp:Serre c-eq}
A subcategory of $\catC$ is a Serre subcategory
if and only if it is closed under finite direct sums, direct summands, and c-equivalences.
\end{prp}
\begin{proof}
It is clear that a Serre subcategory is closed under finite direct sums and direct summands.
Let $A\to B\to C \dto$ be a conflation in $\catC$.
For any additive subcategory $\catD$ closed under direct summands,
$A\oplus C$ belongs to $\catD$ if and only if both $A$ and $C$ belong to $\catD$.
Thus we can conclude that $\catD$ is Serre if and only if it is closed under c-equivalences.
%
%
\end{proof}

Now we can establish a bijection between Serre subcategories and faces. 
\begin{prp}\label{prop:serre-face-bij}
The bijection in Proposition \ref{prp:c-closed bij} (2) restricts to
the bijection between the set $\Serre \catC$ of Serre subcategories of $\catC$
and the set $\Face\M(\catC)$ of faces of $\M(\catC)$ (see Definition \ref{def:face}).
\end{prp}
\begin{proof}
Let $\catS$ be a Serre subcategory of $\catC$ and $F$ a face of $\M(\catC)$.
We already know $\catD_{\M_{\catS}}=\catS$ and $\M_{\catD_{F}}=F$ 
by Propositions \ref{prp:c-closed bij} and \ref{prp:Serre c-eq}.
Hence we only need to show that $\M_{\catS}$ is a face and $\catD_F$ is a Serre subcategory.

We first prove that $\M_{\catS}$ is a face.
Note that $\catS$ is closed under finite direct sums, direct summands, and c-equivalences
by Proposition \ref{prp:Serre c-eq}.
It is clear that $\M_{\catS}$ is a submonoid of $\M(\catC)$
since $\catS$ is closed under direct sums.
Suppose that $[X]+[Y]\in \M_{\catS}$ for some objects $X,Y\in\catC$.
By the definition of $\M_{\catS}$, there exists an object $Z\in\catS$
such that $[Z]=[X]+[Y]=[X\oplus Y]$.
Then we have $X\oplus Y\in\catS$ because $\catS$ is closed under c-equivalences.
Since $\catS$ is closed under direct summands,
both $X$ and $Y$ belong to $\catS$, and hence $[X], [Y]\in \M_{\catS}$.
This proves that $\M_{\catS}$ is a face of $M(\catC)$.

Next, we prove that $\catD_F$ is a Serre subcategory.
It is obvious that $\catD_F$ is closed under finite direct sums
since $F$ is a submonoid of $\M(\catC)$.
We already know that $\catD_F$ is c-closed by Proposition \ref{prp:c-closed bij} (1).
Thus it is enough to show that $\catD_F$ is closed under direct summands
by Proposition \ref{prp:Serre c-eq}.
Let $X$ and $Y$ be objects in $\catC$ with $X\oplus Y \in \catD_F$.
Then $[X]+[Y]=[X\oplus Y]\in F$.
Because $F$ is a face of $\M(\catC)$,
both $[X]$ and $[Y]$ belong to $F$,
which implies both $X$ and $Y$ belong to $\catD_F$.
Therefore $\catD_F$ is a Serre subcategory of $\catC$.
\end{proof}

Finally,
we compare $\M_{\catD}$ with $\M(\catD)$ for an extension-closed subcategory $\catD \subseteq \catC$.
The natural inclusion functor $\iota \colon \catD \inj \catC$
induces the monoid homomorphism $\M(\iota) \colon \M(\catD) \to \M(\catC)$ 
by Proposition \ref{prp:ex fun hom}.
Clearly, the image of $\M(\iota)$ coincide with $\M_{\catD}$.
Thus we have a surjective monoid homomorphism $\M(\catD) \surj \M_{\catD}$.
This monoid homomorphism is not injective in general, as the following example shows.

\begin{example}\label{ex:not inj}
Let $\catA$ be a skeletally small abelian category.
Then $\catA$ can be regarded as an extension-closed subcategory of its bounded derived category $\D^\b(\catA)$.
The natural inclusion functor $\AA \hookrightarrow \D^\b(\AA)$ induces a monoid homomorphism
\[
\M(\catA) \to \M(\D^\b(\catA))=\K_0(\D^\b(\catA)) \iso \K_0(\catA)
\]
by Example \ref{ex:gro-mon} (2).
In fact, the composition $\M(\AA) \to \K_0(\AA)$ of the above maps coincides with the group completion of $\M(\catA)$ (see Remark \ref{rem:k0-as-gp}). 
Consider the polynomial ring $k[T]$ over a field $k$.
Then the monoid homomorphism $\M(\catmod k[T]) \to K_0(\catmod k[T])$ given above is not injective.
Indeed,
$[k[T]/(T)]$ is non-zero in $\M(\catmod k[T])$ by \cite[Proposition 3.5]{JHP},
but it is zero in $\K_0(\catmod k[T])$ because there is a short exact sequence
\[
\begin{tikzcd}
  0 \rar & k[T] \rar["T"] & k[T] \rar & k[T]/(T) \rar & 0.
\end{tikzcd}
\]
\end{example}

In spite of this example, if we consider Serre subcategories, then the natural monoid homomorphism is injective:
\begin{prp}\label{prp:Serre inj}
Let $\catS$ be a Serre subcategory of $\catC$ and $\iota \colon \SS \to \CC$ the inclusion functor.
Then the monoid homomorphism 
\[
\M(\iota) \colon \M(\catS) \to \M(\catC)
\]
is injective.
In particular, it induces an isomorphism $\M(\catS) \simto \M_{\catS} \subseteq \M(\catC)$ of monoids.
\end{prp}
\begin{proof}
Suppose that $A, B \in\catS$ satisfies $\M(\iota)[A] = \M(\iota)[B]$ in $\M(\catC)$,
that is, $A\ceq B$ in $\catC$.
There is a sequence of objects $A_0=A, A_1, \dots, A_n=B$ 
such that $A_i \crel A_{i+1}$ in $\catC$ for all $i$.
Then $A_i\in\catS$ for all $i$ since $\catS$ is c-closed.
Thus it is enough to show that $A \crel B$ in $\catC$ implies $A \crel B$ in $\catS$.
Since $A \crel B$ in $\catC$,
there is a conflation
\begin{equation}\label{diag:Serre inj}
\begin{tikzcd}
  X \rar & Y \rar & Z \rar[dashed] & \
\end{tikzcd}
\end{equation}
in $\catC$ satisfying either
(a) $Y\iso A$ and $X\oplus Z \iso B$
or (b) $Y\iso B$ and $X\oplus Z \iso A$.
Since $A,B\in\catS$ and $\SS$ is closed under direct summands, we have $X,Y,Z \in \catS$ in both cases.
Then the sequence \eqref{diag:Serre inj} is a conflation in $\catS$,
which implies $A \crel B$ in $\catS$.
\end{proof}

This injectivity is remarkable since it is false for Grothendieck groups and one has to consider the higher $K$-group $\K_1$ to deal with its failure.

\begin{example}\label{ex:not inj K0}
Consider the subcategory $\catS$ of $\catmod k[T]$ consisting of 
finitely generated torsion $k[T]$-modules.
It is clearly a Serre subcategory.
The natural inclusion functor $\catS \inj \catmod k[T]$ induces
an injective monoid morphism $\M(\catS) \inj \M(\catmod k[T])$ on the Grothendieck monoids
by Proposition \ref{prp:Serre inj}.
On the other hand,
it induces a zero morphism $\K_0(\catS) \xr{0} \K_0(\catmod k[T])$ on the Grothendieck groups.
Indeed, every object in $\SS$ is a finite direct sum of finitely generated indecomposable torsion $k[T]$-modules, and such a module $M$ has a free resolution
\[
  \begin{tikzcd}
    0 \rar & k[T] \rar["f"] & k[T] \rar & M \rar & 0
  \end{tikzcd}
\]
for some polynomial $f\in k[T]$
by the structure theorem for finitely generated modules over a principal ideal domain.
This implies $[M]=0$ in $\K_0(\catmod k[T])$.
\end{example}

\subsection{Dense 2-out-of-3 subcategories and cofinal subgroups}
In this subsection, we give classifications of dense 2-out-of-3 subcategories of $\CC$ in terms of $\M(\CC)$ and $\K_0(\CC)$, which generalize Thomason's classification of dense triangulated subcategories of a triangulated category \cite{thomason} in terms of $\K_0(\CC)$.
\begin{definition}
  Let $\DD$ be an additive subcategory of $\CC$.
  \begin{enumerate}
    \item $\DD$ is a \emph{dense} subcategory if $\add \DD = \CC$ holds, that is, for every $C \in \CC$, there is some $C' \in \CC$ satisfying $C \oplus C' \in \DD$.
    \item $\DD$ is a \emph{2-out-of-3} subcategory if it satisfies 2-out-of-3 for conflations, that is, if two of three objects $X, Y, Z$ in a conflation $X \to Y\to Z\dto$ belong to $\DD$, then so does the third.
  \end{enumerate}
\end{definition}
We note that 2-out-of-3 subcategories closed under direct summands are called \emph{thick subcategories}, see Definition \ref{dfn:thick}.
\begin{example}
  Let $\TT$ be a triangulated subcategory. Then a subcategory of $\TT$ is a 2-out-of-3 subcategory if and only if it is a triangulated subcategory (see e.g. \cite[1.1]{thomason}).
\end{example}
\begin{remark}
  Let $\DD$ be a 2-out-of-3 subcategory of $\CC$. If $X \oplus Y \in \DD$ and $X \in \DD$, then $Y \in \DD$ by a split conflation $X \to X \oplus Y \to Y \dto$. We will freely use this property in what follows. 
\end{remark}

We can relax the 2-out-of-3 condition of dense 2-out-of-3 subcategories by the following observation.
\begin{proposition}[{\cite[Lemma 5.5]{ZZ}}]
  Let $\DD$ be a dense additive subcategory of $\CC$. Then the following are equivalent.
  \begin{enumerate}
    \item For every conflation $X \to Y \to Z \dto$ in $\CC$, if $X$ and $Y$ belong to $\DD$, then so does $Z$.
    \item For every conflation $X \to Y \to Z \dto$ in $\CC$, if $Y$ and $Z$ belong to $\DD$, then so does $X$.
  \end{enumerate}
\end{proposition}

A key observation in this subsection is as follows.
\begin{proposition}\label{prop:dense-2-3-c-cl}
  Let $\DD$ be a dense 2-out-of-3 subcategory of $\CC$. Then $\DD$ is closed under c-equivalences.
\end{proposition}
\begin{proof}
  Take any conflation
  \begin{equation}\label{eq:original}
    \begin{tikzcd}
      X \rar["f"] & Y \rar["g"] & Z \rar[dashed] & \ 
    \end{tikzcd}
  \end{equation}
  in $\CC$. It suffices to show that $Y$ belongs to $\DD$ if and only if so does $X \oplus Z$.

  First, suppose that $Y$ belongs to $\DD$. Since $\DD$ is dense, there is some $W \in \CC$ satisfying $Y \oplus Z \oplus W \in \DD$. By taking the direct sum of \eqref{eq:original} and a split conflation $Z \to Z \oplus W \to W \dto$, we obtain the following conflation.
  \begin{equation}\label{eq:trick1}
    \begin{tikzcd}[ampersand replacement=\&, column sep = large]
      X \oplus Z \rar["{\begin{sbmatrix}f & 0 \\ 0 & \id_Z \\ 0 & 0\end{sbmatrix}}"] \&
      Y \oplus Z \oplus W \rar["{\begin{sbmatrix}g & 0 & 0 \\ 0 & 0 & \id_W \end{sbmatrix}}"] \&
      Z \oplus W \rar[dashed] \& \ 
    \end{tikzcd}
  \end{equation}
  Since $\DD$ is 2-out-of-3, $Y \oplus (Z \oplus W) \in \DD$ and $Y \in \DD$ implies $Z \oplus W \in \DD$. Therefore, \eqref{eq:trick1} implies that $X \oplus Z$ belongs to $\DD$.

  Conversely, suppose that $X \oplus Z$ belongs to $\DD$. By taking the direct sum of \eqref{eq:original} and a split conflation $Z \to Z \oplus X \to X \dto$, we obtain the following conflation:
  \begin{equation}\label{eq:trick2}
    \begin{tikzcd}[ampersand replacement=\&, column sep = large]
      X \oplus Z \rar["{\begin{sbmatrix}f & 0 \\ 0 & 0 \\ 0 & \id_Z \end{sbmatrix}}"] \&
      Y \oplus X \oplus Z \rar["{\begin{sbmatrix} 0 & \id_X & 0 \\ g & 0 & 0 \end{sbmatrix}}"] \&
      X \oplus Z \rar[dashed] \& \ 
    \end{tikzcd}
  \end{equation}
  Since $\DD$ is 2-out-of-3 and $X \oplus Z \in \DD$, we have $Y \oplus (X \oplus Z) \in \DD$, which implies $Y \in \DD$ by $X \oplus Z \in \DD$.
\end{proof}
Now we can state the following classification of dense 2-out-of-3 subcategories.
\begin{theorem}\label{thm:dense-2-3-bij}
  Let $\CC$ be a skeletally small extriangulated category. Then the bijection in Proposition \ref{prp:c-closed bij} induces a bijection between the following two sets.
  \begin{itemize}
    \item The set of dense 2-out-of-3 subcategories of $\CC$.
    \item The set of cofinal subtractive submonoids of $\M(\CC)$ (see Definition \ref{def:subt-cofin}).
  \end{itemize}
\end{theorem}
\begin{proof}
  Due to Proposition \ref{prp:c-closed bij} (2) and Proposition \ref{prop:dense-2-3-c-cl}, we only have to check the following well-definedness of maps:
  \begin{enumerate}
    \item $\M_\DD$ is a cofinal subtractive submonoid for a dense 2-out-of-3 subcategory $\DD$.
    \item $\DD_N$ is a dense 2-out-of-3 subcategory of $\CC$ for a cofinal subtractive submonoid $N$.
  \end{enumerate}
  
  (1)
  Let $\DD$ be a dense 2-out-of-3 subcategory of $\CC$. Since $\DD$ is closed under direct sums, $\M_\DD$ is a submonoid of $\M(\CC)$.
  To show that $\M_\DD$ is cofinal in $\M(\CC)$, take any $[C] \in \M(\CC)$. Since $\CC$ is dense, there is some $C'$ satisfying $C \oplus C' \in \DD$. This implies $[C] + [C'] = [C \oplus C'] \in \M_\DD$. Thus $\M_\DD$ is cofinal in $\M(\CC)$.

  Next, to show that $\M_\DD$ is subtractive, suppose that $x + y$ and $x$ belong to $\M_\DD$. Take $X, Y \in \CC$ satisfying $[X] =x $ and $[Y] = y$. Then $[X \oplus Y]$ and $[X]$ belong to $\M_\DD$.
  Since $\DD$ is c-closed by Proposition \ref{prop:dense-2-3-c-cl}, we have $\DD = \DD_{\M_\DD}$ by Proposition \ref{prp:c-closed bij}. Therefore, $X \oplus Y$ and $X$ belong to $\DD$. Since $\DD$ is 2-out-of-3, we obtain $Y \in \DD$. Thus $y = [Y] \in \M_\DD$ holds.

  (2)
  Let $N$ be a cofinal subtractive submonoid of $\M(\DD)$. To show that $\DD_N$ is dense, take any $C \in \CC$. Since $N$ is cofinal, there is some $C' \in \CC$ satisfying $[C \oplus C'] = [C] + [C'] \in N$. Thus $C \oplus C' \in \DD_N$ holds.

  Next, we will check that $\DD_N$ is 2-out-of-3. Take any conflation $X \to Y \to Z \dto$ in $\CC$. Then we have $[Y] = [X] + [Z]$ in $\M(\CC)$. If $X$ and $Z$ belong to $\DD_N$, then $[X]$ and $[Z]$ belong to $N$, and hence so does $[Y] = [X] + [Z]$ since $N$ is a submonoid. Thus $Y$ belongs to $\DD_N$.
  Similarly, if $X$ and $Y$ belong to $\DD_N$, then $[X]$ and $[Y] = [X] + [Z]$ belong to $N$, and hence so does $[Z]$ since $N$ is subtractive. Thus $Z$ belongs to $\DD_N$. The same argument works if $Y$ and $Z$ belong to $\DD_N$.
\end{proof}

As a corollary, we can immediately deduce the following classification of dense triangulated subcategories.
\begin{corollary}[{\cite[Theorem 2.1]{thomason}}]\label{cor:thomason}
  Let $\TT$ be a skeletally small triangulated category. Then there exists a bijection between the following two sets:
  \begin{itemize}
    \item The set of dense triangulated subcategories of $\TT$.
    \item The set of subgroups of $\K_0(\TT)$.
  \end{itemize}
\end{corollary}
\begin{proof}
  Since $\M(\TT) \iso \K_0(\TT)$ holds by Example \ref{ex:gro-mon} (2) and dense triangulated subcategories of $\TT$ are precisely dense 2-out-of-3 subcategories, we only have to check that a subset of $\K_0(\TT)$ is a cofinal subtractive submonoid if and only if it is a subgroup. This follows from Remark \ref{rem:grp-subt-cof}.
\end{proof}

Due to Theorem \ref{thm:dense-2-3-bij}, the classification of dense 2-out-of-3 subcategories can be reduced to the classification of cofinal subtractive submonoids of the Grothendieck monoid.
Since it is easier to deal with subgroups of a group than with submonoids of a monoid, we study the relation between cofinal subtractive submonoids of a monoid and subgroups of its group completion. The main observation can be summarized as follows:
\begin{proposition}\label{prop:cof-subt-subgrp}
  Let $M$ be a monoid and $\rho \colon M \to \gp M$ its group completion.
  Consider the following two maps:
  \[
    \begin{tikzcd}
      \{ \text{submonoids of $M$} \}
      \rar["\Phi", shift left] & \lar["\Psi", shift left]
      \{ \text{subgroups of $\gp M$} \},
    \end{tikzcd}
  \]
  where $\Phi(N) = \la \rho(N) \ra_\Z$ and $\Psi (H) = \rho^{-1}(H)$.
  Then the following hold.
  \begin{enumerate}
    \item $\Psi\Phi(N) \supseteq N$ and $\Phi\Psi(H) \subseteq H$ hold for every $N$ and $H$.
    \item If $N$ is a cofinal subtractive submonoid of $M$, then $\Psi\Phi(N) = N$ holds.
    \item $\Phi$ and $\Psi$ induce bijections between the following sets:
    \begin{itemize}
      \item The set of cofinal subtractive submonoids of $M$.
      \item The set of subgroups $H$ of $\gp M$ such that $\rho^{-1}(H)$ is cofinal in $M$.
    \end{itemize}
    \item $\Phi$ and $\Psi$ induce bijections between the following sets for each cofinal subset $S$ of $M$:
    \begin{itemize}
      \item The set of subtractive submonoids of $M$ containing $S$.
      \item The set of subgroups of $\gp M$ containing $\rho(S)$.
    \end{itemize}
  \end{enumerate}
\end{proposition}
\begin{proof}
  It is clear that $\Psi(H)$ is a subtractive submonoid of $M$ for every subgroup $H$ of $G$.
  
  (1) Clear from definitions.

  (2)
  Let $N$ be a cofinal subtractive submonoid of $M$, and put $H:= \la \rho(N) \ra_\Z$. It suffices to show $\rho^{-1}(H) \subseteq N$ by (1).
  Suppose $x \in \rho^{-1}(H)$. Then since $\rho(x) \in H = \la \rho(N) \ra_\Z$ and $N$ is a submonoid of $M$, there are elements $n_1$ and $n_2$ in $N$ such that $\rho(x) = \rho(n_1) - \rho(n_2)$, so $\rho(x + n_2) = \rho(n_1)$. Therefore, there is an element $m \in M$ such that $x + n_2 + m = n_1 + m$ in $M$ (see the argument below Definition \ref{def:mon-loc}).
  Then, since $N$ is cofinal in $M$, there is $m' \in M$ satisfying $m + m' \in N$. Then we have $x + (n_2 + m + m') = n_1 + m + m'$. Now $x \in N$ follows since $N$ is subtractive and $x + (n_2 + m + m')$ and $n_2 + m + m'$ belong to $N$.

  (3)
  Let $H$ be a subgroup of $\gp M$. We claim that if $\rho^{-1}(H)$ is cofinal in $M$, then $\Phi\Psi(H) = H$ holds. It suffices to show $H \subseteq \Phi\Psi(H)$ by (1). Let $h \in H$, then $h = \rho(x) - \rho(y)$ holds for some $x, y \in M$. Since $\rho^{-1}(H)$ is cofinal, there is some $y' \in M$ such that $y + y' \in \rho^{-1}(H)$. Then $h = \rho(x + y') - \rho(y + y')$ and $\rho(y + y') \in H$ hold, which imply $\rho(x + y') = h + \rho(y + y') \in H$. It follows that both $x + y'$ and $y + y'$ belong to $\rho^{-1}(H)$, so we obtain $h = \rho(x + y') - \rho(y + y') \in \la \rho (\rho^{-1}(H)) \ra_\Z = \Phi\Psi(H)$.
 
  Therefore, by (2) and the above argument, we only have to check that $\Phi$ and $\Psi$ are well-defined on these two sets. This is immediate from (2) and definitions.

  (4)
  Every subtractive submonoid containing $S$ is cofinal, and every subgroup $H$ of $\gp M$ containing $\rho(S)$ satisfies that $\rho^{-1}(H)$ is cofinal by $S \subseteq \rho^{-1}(H)$. Therefore, the two sets in (4) are subsets of the corresponding two sets in (3). Thus it suffices to observe that $\Phi$ and $\Psi$ are well-defined on the two sets in (4), which is immediate from definitions.
\end{proof}

By combining this with Theorem \ref{thm:dense-2-3-bij}, we immediately obtain the following result.
\begin{corollary}\label{cor:dense-2-3-k0}
  Let $\CC$ be a skeletally small extriangulated category. Then there is a bijection between the following two sets, where $\rho \colon \M(\CC) \to \K_0(\CC)$ is the group completion.
  \begin{itemize}
    \item The set of dense 2-out-of-3 subcategories of $\CC$.
    \item The set of subgroups $H$ of $\K_0(\CC)$ such that $\rho^{-1}(H)$ is cofinal in $\M(\CC)$.
  \end{itemize}
\end{corollary}

Using this observation, we can obtain all dense 2-out-of-3 subcategories in an abelian length category with finitely many simples. First, recall the following description of the Grothendieck monoid.
Let $\AA$ be an abelian length category, that is, an abelian category such that every object has a composition series. Suppose that $\{S_1, \dots S_n\}$ is the set of all non-isomorphic simple objects in $\AA$. Then for $C \in \AA$, define $\udim C := (x_1, \dots , x_n) \in \N^n$, where $x_i$ is the multiplicity of $S_i$ in the composition series of $C$ (this is well-defined due to the Jordan-H\"older theorem). Then $\udim$ respects conflations, and moreover, it induces the following isomorphisms of monoids and groups:
\begin{equation}\label{eq:JHP}
  \begin{tikzcd}
    \M(\AA) \rar["\udim", "\sim"'] \dar["\rho"'] & \N^n \dar[hookrightarrow, "\iota"] \\
    \K_0(\AA) \rar["\udim", "\sim"'] & \Z^n \rlap{,}
  \end{tikzcd}
\end{equation}
where $\rho$ is the group completion and $\iota$ is the natural inclusion.

\begin{corollary}
  Let $\AA$ be an abelian length category with $n$ simple objects up to isomorphism. Then there are bijections between the following two sets:
  \begin{itemize}
    \item The set of dense 2-out-of-3 subcategories $\CC$ of $\AA$.
    \item The set of subgroups $H$ of $\Z^n$ containing a strictly positive element.
  \end{itemize}
  Here an element $x = (x_1, \dots, x_n)$ of $\Z^n$ is \emph{strictly positive} if $x_i >0$ for all $i$.
  The maps are given by $\CC \mapsto \la\{\udim C \mid C \in \CC \}\ra_\Z$ and $H \mapsto \{ C \in \CC \mid \udim C \in H\}$.
\end{corollary}
\begin{proof}
  By Corollary \ref{cor:dense-2-3-k0} and the isomorphisms in \eqref{eq:JHP}, we only have to observe that $\N^n \cap H$ is cofinal in $\N^n$ if and only if $H$ contains a strictly positive element for a subgroup $H$ of $\Z^n$. This is straightforward, so we omit it.
\end{proof}

Certain classes of dense 2-out-of-3 subcategories were classified via the Grothendieck group in \cite{matsui} (for the exact case) and \cite{ZZ} (for the extriangulated case).
We explain that their results can be immediately deduced from ours. Let us explain some terminology to state them.
\begin{definition}
  Let $\CC$ be an extriangulated category. Then a set $\GG$ of objects in $\CC$ is called a \emph{generator} if for every $C \in \CC$ there is a conflation $X \to G \to C \dto$ in $\CC$ with $G \in \GG$.
\end{definition}
Now we can deduce their results as follows.
\begin{corollary}[{\cite[Theorem 2.7]{matsui}, \cite[Theorem 5.7]{ZZ}}]\label{cor:matsui}
  Let $\CC$ be a skeletally small extriangulated category and $\GG$ a generator of $\CC$. Then there is a bijection between the following two sets:
  \begin{enur}
    \item The set of dense 2-out-of-3 subcategories of $\CC$ containing $\GG$.
    \item The set of subgroups of $\K_0(\CC)$ containing the image of $\GG$.
  \end{enur}
\end{corollary}
\begin{proof}
  We will show that these sets are in bijection with the following one:
  \begin{enur}
    \item[(iii)] The set of subtractive submonoids of $\M(\CC)$ containing the image of $\GG$. 
  \end{enur}

  Denote by $[\GG] \subseteq \M(\CC)$ the image of $\GG$ in $\M(\CC)$, then $[\GG]$ is cofinal in $\M(\CC)$. Indeed, for every $[C] \in \M(\CC)$, there is a conflation $X \to G \to C \dto$ in $\CC$, and thus $[C] + [X] = [G] \in [\GG]$ holds in $\M(\CC)$.
  Therefore, by Proposition \ref{prop:cof-subt-subgrp} (4), we have a bijection between (ii) and (iii), and every submonoid in (iii) is cofinal.
  Therefore, (i) and (iii) are subsets of the two sets in Theorem \ref{thm:dense-2-3-bij}. Hence it suffices to observe the following well-definedness, which are immediate from definitions:
  If $\DD$ is a dense 2-out-of-3 subcategory containing $\GG$, then $\M_\DD$ contains $[\GG]$, and if $N$ is a submonoid of $\M(\CC)$ containing $[\GG]$, then $\DD_N$ contains $\GG$.
\end{proof}

\begin{remark}
  In \cite[Theorem 5.1]{H}, the above classification of dense 2-out-of-3 subcategories via the Grothendieck group was generalized to an \emph{$n$-exangulated category}, which is a generalization of extriangulated category. Therefore, it may be possible to establish classifications similar to this paper for an $n$-exangulated categories using the Grothendieck monoid of an $n$-exangulated category.
\end{remark}
\section{Grothendieck monoids and localization of extriangulated categories}\label{sec:4}
The purpose of this section is to prove Theorem \ref{thm:mon loc ET},
which describes the Grothendieck monoid of the localization of extriangulated categories as a monoid quotient of the Grothendieck monoid.
\emph{Throughout this section, we assume that
every category, functor, and subcategory is additive}.
\subsection{Localization of extriangulated categories}\label{ss:loc ET}
We first recall the localization of an extriangulated category 
following \cite{NOS}.
\begin{dfn}\label{def:ex-loc}
Let $\catC$ be an extriangulated category,
and let $S\subseteq \Mor\catC$ be a class of morphisms.
A pair $(\catC_S,Q)$ of an extriangulated category $\catC_S$
and an exact functor $Q \colon \catC \to \catC_S$
is the \emph{exact localization of $\catC$ with respect to $S$}
if it satisfies the following conditions:
\begin{enur}
\item
$F(s)$ is an isomorphism in $\catC_S$ for any $s\in S$.
\item
For any extriangulated category $\DD$ and any exact functor $F\colon \catC \to \catD$ 
such that $F(s)$ is an isomorphism for any $s\in S$,
there exists a unique exact functor $F_S \colon \catC_S \to \catD$
satisfying $F=F_S\circ Q$.
\end{enur} 
\end{dfn}
If the exact localization exists,
it is unique up to exact isomorphisms.
Note that the exact localization is closed under exact \emph{isomorphisms}, but not closed under exact \emph{equivalences}, see Remark \ref{rmk:2-loc} below.

Nakaoka--Ogawa--Sakai \cite{NOS} constructs the exact localization of an extriangulated category 
by a class of morphisms under some assumptions.
We only recall the construction of the localization of an extriangulated category 
by the class of morphisms determined by a thick subcategory, as we shall explain.

From now on, $\catC=(\catC,\bbE,\frs)$ is an extriangulated category.

\begin{dfn}\label{dfn:thick}
A subcategory $\catN$ of $\catC$ is a \emph{thick} subcategory
if it satisfies the following conditions:
\begin{enur}
\item
$\catN$ is closed under direct summands.
\item
$\catN$ satisfies 2-out-of-3 for conflations in $\catC$,
that is, if two of three objects $A, B, C$ in a conflation $A \to B\to C\dto$ 
belong to $\catN$,
then so does the third.
\end{enur}
\end{dfn}

For a thick subcategory $\catN \subseteq \catC$,
we set the following classes of morphisms:
\begin{align*}
\calL &:=\{\ell\in \Mor\catC \mid \text{there is a conflation $A \xr{\ell} B \to N \dto $ with $N\in\catN$} \},\\
\calR &:=\{r\in \Mor\catC \mid \text{there is a conflation $N \to A \xr{r} B \dto $ with $N\in\catN$} \}.
\end{align*}
We define $S_{\catN}$ to be
the class of all finite compositions of morphisms in $\calL$ and $\calR$.
We can easily check $\calL \circ \calL \subseteq \calL$ and $\calR \circ \calR \subseteq \calR$.
Thus a morphism $s$ in $S_{\catN}$ is of the form
$s=\cdots \ell_{n-1} r_n \ell_{n+1} r_{n+2} \cdots $ for some $\ell_i\in \calL$ and $r_j \in \calR$.
The thick subcategory $\catN$ can be recovered from $S_{\catN}$ since we have
\begin{align}\label{eq:thick N}
\catN=\{A\in \catC \mid \text{both $A\to 0$ and $0\to A$ belong to $S_{\catN}$} \}
\end{align}
by \cite[Lemma 4.5]{NOS}.
In the following, we consider 
the exact localization $\catC/\catN :=\catC_{S_{\catN}}$ of $\catC$ with respect to $S_{\catN}$.
This localization satisfies the following natural universal property.
\begin{proposition}\label{prop:ex-loc-univ2}
  Let $\CC$ be an extriangulated category and $\NN$ a thick subcategory of $\CC$. Suppose that the exact localization $Q \colon \CC \to \CC / \NN := \CC_{S_\NN}$ exists. Then it satisfies the following conditions.
  \begin{enur}
    \item $Q(N) \iso 0$ holds for every $N \in \NN$.
    \item For any extriangulated category $\DD$ and an any exact functor $F \colon \CC \to \DD$ such that $F(N) \iso 0$ holds for every $N \in \NN$, there exists a unique exact functor $F_\NN \colon \CC/\NN \to \DD$ satisfying $F=F_\NN \circ Q$.
  \end{enur}
\end{proposition}
\begin{proof}
  By comparing the claimed properties with Definition \ref{def:ex-loc}, it suffices to show the following claim: for an extriangulated category $\DD$ and an exact functor $F \colon \CC \to \DD$, we have that $F (N) \iso 0$ holds for every $N \in \NN$ if and only if $F(s)$ is an isomorphism for every $s \in S_\NN$.

  To see the ``only if'' part, suppose that $F(N) \iso 0$ holds for every $N \in \NN$. It suffices to check that $F(\ell)$ and $F(r)$ are isomorphisms for $\ell \in \calL$ and $r \in \calR$ respectively. By the definition of $\calL$, there is a conflation $A \xr{\ell} B \to N \dto $ with $N\in\catN$. Since $F$ is an exact functor, we obtain a conflation $F(A) \xr{F(\ell)} F(B) \to F(N) \dto$ in $\DD$. Then $F(N) \iso 0$ implies that $F(\ell)$ is an isomorphism in $\DD$ by considering the associated long exact sequence (cf. \cite[Corollary 3.12]{NP}). The same proof applies to $r \in \calR$.

  To see the ``if'' part, suppose that $F(s)$ is an isomorphism for every $s \in S_\NN$, and let $N \in \NN$. Then $0 \to N$ clearly belongs to $\calL \subseteq S_\NN$, so $0 \to F(N)$ is an isomorphism in $\DD$. Thus the assertion holds.
\end{proof}

Let $\ol{\catC}:=\catC/[\catN]$ be the quotient by the ideal $[\catN]$
consisting of morphisms which factor through objects in $\catN$,
and let $p \colon \catC \to \ol{\catC}$ be the canonical functor.
In what follows,
we write $\ol{f}:=p(f)$ for any morphism $f$ in $\catC$.
Set $\ol{S_{\catN}}:=p(S_{\catN})$
\footnote{Our notation $\ol{S_{\catN}}$ is different from the one in \cite{NOS}.
However, they coincide if (i) of Condition \ref{cond:NOS loc} is satisfied. 
See \cite[Lemma 3.2]{NOS}.}.
We recall a condition in \cite{NOS}
under which the exact localization $\catC/\catN=(\catC/\catN,Q)$ exists.
\begin{cond}\label{cond:NOS loc}
Let $\catN$ be a thick subcategory of $\catC$.
\begin{enur}
\item 
$f\in S_{\catN}$ holds for any split monomorphism $f \colon A \to B$ in $\catC$
such that $\ol{f}$ is an isomorphism in $\ol{\catC}$. (This is equivalent to the dual condition by \cite[Lemma 3.2]{NOS}: $f\in S_{\catN}$ holds for any split epimorphism $f \colon A \to B$ in $\catC$ such that $\ol{f}$ is an isomorphism in $\ol{\catC}$.)
\item
$\ol{S_{\catN}}$ satisfies 2-out-of-3 with respect to compositions in $\ol{\catC}$.
\item
$\ol{S_{\catN}}$ is a multiplicative system in $\ol{\catC}$.
\item
The set $\left\{\ol{t x s} \mid \text{$x$ is an inflation in $\catC$ and}\, s, t\in S_{\catN} \right\}$ is closed under compositions. 
Dually, the set $\left\{\ol{t y s} \mid \text{$y$ is a deflation in $\catC$ and}\, s, t\in S_{\catN} \right\}$ is closed under compositions.
\end{enur}
\end{cond}

\begin{fct}[{\cite[Theorem 3.5, Lemma 3.32]{NOS}}]\label{fct:NOS loc}
Let $\catN$ be a thick subcategory of $\catC$.
If it satisfies Condition \ref{cond:NOS loc},
then there exists the exact localization $\catC/\catN$ satisfying the following properties.
\begin{enua}
\item
$\catC/\catN$ is constructed as the category $\ol{S_{\catN}}^{-1}\ol{\catC}$ of fractions.
In particular, every morphism in $\catC/\catN$ can be described as a right or left roof of morphisms in $\ol{\catC}$.
\item
For any inflation $\alpha$ in $\catC/\catN$,
there exist an inflation $f$ in $\catC$ and isomorphisms $\beta, \gamma$ in $\catC/\catN$
satisfying $\alpha = \beta \circ Q(f) \circ \gamma$.
Dually, for any deflation $\alpha$ in $\catC/\catN$,
there exist a deflation $f$ in $\catC$ and isomorphisms $\beta, \gamma$ in $\catC/\catN$
satisfying $\alpha = \beta \circ Q(f) \circ \gamma$.
\end{enua}
\end{fct}

\begin{rmk}
Let us confirm 
that Condition \ref{cond:NOS loc} implies the conditions in \cite[Theorem 3.5]{NOS}.
Suppose that $S_{\catN}$ satisfies Condition \ref{cond:NOS loc}.
It is clear that $S_{\catN}$ satisfies (M0) in \cite[Section 3]{NOS}.
The condition (MR1), (MR2), and (MR4) in \cite[Theorem 3.5]{NOS} 
are nothing but (i), (iii), and (iv) of Condition \ref{cond:NOS loc}, respectively.
By \cite[Lemma 4.6]{NOS}, $S_{\catN}$ satisfies (M3) in \cite[Corollary 3.4]{NOS}.
Thus it also satisfies (MR3) by the condition (i) and \cite[Lemma 3.2, Claim 3.6]{NOS}.
Therefore $S_{\catN}$ satisfies all the conditions in \cite[Theorem 3.5]{NOS}.
\end{rmk}

\begin{rmk}\label{rmk:2-loc}
  Some readers may find the definition of exact localizations unsatisfactory
  since it is not preserved by exact equivalences.
  In fact, there is a notion of \emph{exact $2$-localizations}, which is preserved by exact equivalences.
  For two extriangulated categories $\catC$ and $\catD$,
  we denote by $\Fun^{\exa}(\catC,\catD)$ the category of exact functors $\catC \to \catD$
  and natural transformations of them.
  Any exact functor $F\colon \catC \to \catC'$ induces
  a functor $F^*:\Fun^{\exa}(\catC',\catD) \to \Fun^{\exa}(\catC,\catD)$
  defined by $F^*(G):=G\circ F$.
  
  For a class $S\subseteq \Mor\catC$ of morphisms in an extriangulated category $\catC$,
  the \emph{exact $2$-localization} of $\catC$ with respect to $S$
  is a pair $(\catC_S,Q)$ of an extriangulated category $\catC_S$
  and an exact functor $Q\colon \catC \to \catC_S$
  which satisfies the following conditions:
  \begin{enur}
  \item
  $F(s)$ is an isomorphism in $\catC_S$ for any $s\in S$.
  \item
  For any extriangulated category $\DD$ and an exact functor $F\colon \catC \to \catD$ 
  such that $F(s)$ is an isomorphism for any $s\in S$,
  there exist an exact functor $\wti{F}\colon \catC_S \to \catD$
  and a natural isomorphism $F \iso \wti{F} \circ Q$ of exact functors.
  \item
  The functor
  \[
  Q^*\colon \Fun^{\exa}(\catC_S, \catD) \to \Fun^{\exa}(\catC,\catD)
  \]
  is fully faithful for every extriangulated category $\catD$.
  \end{enur}

  Actually, the exact localization $\catC/\catN$ obtained by Fact \ref{fct:NOS loc}
  is also the exact $2$-localization by \cite[Theorem 3.5]{NOS}.
\end{rmk}

Although Condition \ref{cond:NOS loc} seems a little technical, there is a sufficient condition which we can easily verify,
see Fact \ref{fct:loc tri} and \ref{fct:loc ex}.
In this paper, we focus on the situation where Condition \ref{cond:NOS loc} is satisfied
so that we can use the additional properties (1)--(2) of Fact \ref{fct:NOS loc} freely.

\begin{rmk}\label{rmk:roof}
We do not assume in Condition \ref{cond:NOS loc} that $S_{\catN}$ is a multiplicative system.
However, any morphism $f \colon X\to Y$ in $\catC/\catN$
is of the form $f=Q(s)^{-1}Q(g)$ 
for some morphisms $g \colon X\to A$ in $\catC$ and $s \colon Y \to A$ in $S_{\catN}$.
Indeed,
let $\ol{Q} \colon \ol{\catC} \to \ol{S_{\catN}}^{-1}\ol{\catC} = \catC/\catN$
be the canonical functor.
For any morphism $f \colon X\to Y$ in $\catC/\catN$,
there are morphisms $g \colon X\to A$ in $\ol{\catC}$ 
and $\ol{s} \colon Y \to A$ in $\ol{S_{\catN}}$ 
satisfying $f=\ol{Q}(\ol{s})^{-1}\ol{Q}(\ol{g})$.
Since $\ol{Q}(\ol{\phi})=Q(\phi)$ holds for any morphism $\phi$ in $\catC$, the claim follows.
\end{rmk}

We introduce some convenient classes of thick subcategories.
\begin{dfn}
Let $\catN$ be a thick subcategory of $\catC$.
\begin{enua}
\item
$\catN$ is called \emph{biresolving}
if for any $C\in\catC$, there exist an inflation $C\to N$ and a deflation $N'\to C$
in $\catC$ with $N,N' \in \catN$.
\item
$\catN$ is called \emph{percolating}
if for any morphism $f \colon X\to Y$ in $\catC$ factoring through some object in $\catN$,
there exist a deflation $g \colon X \to N$ and an inflation $h \colon N \to Y$
satisfying $N\in\catN$ and $f=hg$
\footnote{This definition is different from \cite[Definition 4.28]{NOS},
but they are equivalent by \cite[Lemma 4.29]{NOS}.}.
\end{enua}
\end{dfn}

For the triangulated case, we have the following observation.
\begin{ex}\label{ex:tri-thick-bir-perco}
  Let $\TT$ be a triangulated category with shift functor $\Sigma$.
\begin{enua}
\item
A thick subcategory of $\catT$ in the sense of Definition \ref{dfn:thick}
coincides with the usual one,
that is, a subcategory of $\catT$ closed under cones, shifts, and direct summands.
We can easily check it by considering the following conflations:
\[
X \xr{f} Y \to \Cone(f) \dto,\quad
X \to 0 \to \Sg X \dto, \quad \text{and} \quad
\Sg^{-1} X \to 0 \to X \dto.
\]
\item
Any thick subcategory of $\TT$ is biresolving
because there exist conflations 
$C \xr{0} N \to N\oplus \Sg C \dto$ and $\Sg^{-1} C \oplus N \to  N \xr{0} C \dto$.
\item Similarly, any thick subcategory of $\TT$ is percolating because every morphism is both an inflation and a deflation in $\TT$.
\end{enua}
\end{ex}

Typical examples of percolating subcategories are Serre subcategories of \emph{admissible} extriangulated categories, as we shall explain.
A morphism $f\colon A \to B$ in $\catC$ is called \emph{admissible}
if it has a factorization $f=i\circ d$ such that $i$ is an inflation and $d$ is a deflation.
We call this factorization a \emph{deflation-inflation factorization}.
We also say that \emph{$\catC$ is admissible} if every morphism in $\catC$ is admissible.
For examples, abelian categories and triangulated categories are admissible.

\begin{example}\label{ex:adm-serre-perco}
  Let $\CC$ be an admissible extriangulated category. Then every Serre subcategory $\catN$ of $\catC$ is percolating.
  Indeed, let $f \colon X\to Y$ be a morphism in $\catC$ 
  having a factorization $X\xr{x} N \xr{y} Y$ with $N\in\catN$.
  Consider a deflation-inflation factorization $X \xr{d_1} M_1 \xr{i_1} N$ 
  of $x$.
  Since $\catN$ is a Serre subcategory, an inflation $i_1$ implies $M_1 \in \catN$.
  Then consider a deflation-inflation factorization $M_1 \xr{d_2} M_2 \xr{i_2} Y$
  of $M_1 \xr{i_1} N \xr{y} Y$.
  We have $M_2\in \catN$ by a deflation $d_2$,
  and $X \xr{d_2d_1} M_2 \xr{i_2} Y$ is a desired decomposition.
  It is obvious that Serre subcategories are thick, and thus $\catN$ is percolating.
\end{example}

The following two facts are useful conditions where Condition \ref{cond:NOS loc} is satisfied for biresolving and percolating subcategories.
\begin{fct}[{\cite[Propostion 4.26]{NOS}}]\label{fct:loc tri}
If $\catN$ is biresolving,
then Condition \ref{cond:NOS loc} is satisfied.
In this case, $\catC/\catN$ is a triangulated category.
\end{fct}

\begin{fct}[{\cite[Corollary 4.42]{NOS}}]\label{fct:loc ex}
Let $\catN$ be a thick subcategory of $\catC$.
Consider the following conditions:
\begin{description}
\item[(EL1)]
$\catN$ is percolating.
\item[(EL2)]
For any split monomorphism $f\colon A\to B$ in $\catC$ such that $\ol{f}$ is an isomorphism in $\ol{\catC}$,
there exist $N\in\catN$ and $j\colon N\to B$ in $\catC$
such that $[f\, j] \colon A\oplus N \to B$ is an isomorphism in $\catC$.
\item[(EL3)]
For every conflation $A\xr{f} B\xr{g}C \dto$ in $\CC$,  
both of the following hold:
\begin{align*}
\Ker\left(\CC(-,A)\xr{f \circ (-)}\CC(-,B)\right) &\subseteq [\catN](-,A),\\
\Ker\left(\CC(C,-)\xr{(-) \circ g}\CC(B,-)\right) &\subseteq [\catN](C,-).
\end{align*}

\item[(EL4)]
$\catC$ is admissible, namely, every morphism $f\colon A \to B$ in $\catC$ has a factorization
$f=i\circ d$ such that $i$ is an inflation and $d$ is a deflation.
\end{description}
Then the following statements hold.
\begin{enua}
\item
If {\upshape (EL1)--(EL3)} are satisfied, then Condition \ref{cond:NOS loc} is satisfied, and $\catC/\catN$ is an exact category.
\item
If {\upshape (EL1)--(EL4)} are satisfied, then
$\catC/\catN$ is an abelian category (endowed with the natural extriangulated structure).
\end{enua}
\end{fct}

The situation can be summarized as follows.
\[
  \begin{tikzcd}[row sep = 0]
    \text{$\NN$ is biresolving} \rar[Rightarrow, "\text{Fact \ref{fct:loc tri}}"]
    & \text{Condition \ref{cond:NOS loc} + $\CC/\NN$ is triangulated} \\ 
    \text{(EL1)--(EL3) (+ (EL4))} \rar[Rightarrow, "\text{Fact \ref{fct:loc ex}}"]
    & \text{Condition \ref{cond:NOS loc} + $\CC/\NN$ is exact (+ abelian)}
  \end{tikzcd}
\]
Our main interests are the above two cases, namely, the case where $\NN$ is biresolving, and the case where (EL1)--(EL3) are satisfied (so $\NN$ is percolating).

\begin{rmk}\label{rmk:EL}\hfill
\begin{enua}
\item
Consider the following condition:
\begin{description}
\item[(WIC)]
Let $h=gf$ be a morphism in $\catC$.
If $h$ is an inflation, then so is $f$.
Dually, if $h$ is a deflation, then so is $g$.
\end{description}
Then (WIC) implies (EL2) by \cite[Remark 4.31 (2)]{NOS}.
A triangulated category satisfies (WIC)
since every morphism is both an inflation and a deflation.
More generally, 
an extension-closed subcategory of a triangulated category which is closed under direct summands
satisfies (WIC),
see Lemma \ref{lem:WIC} below.
\item
(EL2) implies Condition \ref{cond:NOS loc} (i) by \cite[Lemma 3.2, 4.34]{NOS}.
\item Under the condition (EL1), the condition (EL3) is equivalent to that $\ov{f}$ (resp. $\ov{g}$) is a monomorphism (resp. an epimorphism) in $\ov{\CC}$ for every conflation $A\xr{f} B\xr{g}C \dto$ in $\CC$, see Lemma \ref{lem:P cond}.
\end{enua}
\end{rmk}

\begin{ex}\label{ex:two-cond}
  The following are typical examples of the case where $\NN$ is biresolving or $\NN$ satisfies (EL1)--(EL4).
\begin{enua}
\item
Let $\catT$ be a triangulated category.
A thick subcategory $\catN$ of $\catT$ 
is biresolving as mentioned in Example \ref{ex:tri-thick-bir-perco}, so the exact localization $\TT/\NN$ exists and becomes a triangulated category by Fact \ref{fct:loc tri}.
This coincides with the Verdier quotient of $\catT$ by $\catN$.
\item Let $\FF$ be a Frobenius exact category, and let $\NN$ be the subcategory of projective objects in $\FF$. Then $\NN$ is a biresolving subcategory of $\FF$. In this case, the exact localization $\FF/\NN$ coincides with the usual stable (triangulated) category $\un{\FF} = \FF/[\NN]$.
\item
Let $\catA$ be an abelian category.
A Serre subcategory $\catS$ of $\catA$
is percolating as mentioned in Example \ref{ex:adm-serre-perco}.
Moreover, it satisfies (EL1)--(EL4) in Fact \ref{fct:loc ex}. Indeed, (EL2) is satisfied by Remark \ref{rmk:EL} (1) since $\AA$ satisfies (WIC), and (EL3) is satisfied since every inflation and deflation in $\AA$ is a monomorphism and an epimorphism respectively.
Therefore, the exact localization $\catA/\catS$ exists and becomes an abelian category by Fact \ref{fct:loc ex}.
This coincides with the Serre quotient of $\catA$ by $\catS$.
\end{enua}
\end{ex}
We have the following convenient criterion for the condition (WIC).
\begin{lem}\label{lem:WIC}
Let $\catC$ be an extension-closed subcategory of a triangulated category $\catT$
with shift functor $\Sigma$.
If $\catC$ is closed under direct summands in $\TT$,
then it satisfies {\upshape (WIC)}.
\end{lem}
\begin{proof}
Let $f\colon A \to B$ and $g\colon B\to C$ be morphisms in $\catC$.
Suppose that $h:=gf$ is an inflation,
and thus there is a conflation $A \xr{h} C \xr{c} D\xdr{\delta}$ in $\catC$.
Then there exists a conflation $A \xr{f} B \xr{b} X \xdr{\tau}$ in $\catT$
by taking cone of $f$.
Applying \cite[Proposition 1.20]{LN} (see also \cite[Proposition 1.4.3]{Ne}) to $\catT$,
we obtain a morphism $x:X \to D$ in $\catT$
which gives a morphism of conflations
\[
\begin{tikzcd}
B \arr{d,"g"'} \arr{r,"b"} & X \arr{r,"\tau"}  \arr{d,"x"} & \Sigma A \arr{r,"-\Sg f",dashed} \arr{d,equal} & {} \\
C \arr{r,"c"'} & D \arr{r,"\delta"'} & \Sigma A \arr{r,"-\Sg h"',dashed} & {}
\end{tikzcd}
\]
in $\catT$ and makes $B \xr{\begin{bsmatrix}b \\ g\end{bsmatrix}} X\oplus C \xr{\begin{bsmatrix}x & -c\end{bsmatrix}} D \dto$ a conflation in $\catT$.
Since $B,D\in \catC$ and $\catC$ is extension-closed, we have $X\oplus C \in \catC$, which implies $X\in \catC$ because $\catC$ is closed under direct summands.
Hence $f$ is an inflation in $\catC$.
We can dually prove that, if $h$ is a deflation, then so is $g$, so we omit it.
\end{proof}

\subsection{The Grothendieck monoid of the localization of an extriangulated category}\label{ss:mon loc ET}
In this subsection, we investigate the relations between Grothendieck monoids and the exact localization.
Throughout this subsection, $\catC$ is an extriangulated category and $\catN$ is a thick subcategory of $\CC$.

The following is the main theorem of this section. We refer the reader to Definition \ref{def:quot-submon} for the quotient monoid by a submonoid.
\begin{thm}\label{thm:mon loc ET}
Let $\CC$ be a skeletally small extriangulated category and $\NN$ its thick subcategory.
Suppose that the following two conditions are satisfied:
\begin{enur}
\item
Condition \ref{cond:NOS loc} is satisfied.
Thus the exact localization $Q\colon \catC \to \catC/\catN$ exists,
and we can freely use Fact \ref{fct:NOS loc}.
\item
$S_{\catN}$ is saturated,
that is, for every morphism $f$ in $\CC$, we have
$f\in S_{\catN}$ if $Q(f)$ is an isomorphism in $\catC/\catN$.
\end{enur}
Then $\M(Q) \colon \M(\CC) \to \M(\CC/\NN)$ induces an isomorphism of monoids:
\[
\M(\catC)/\M_{\catN} \simto \M(\catC/\catN).
\]
\end{thm}
We first prove this theorem,
and then discuss the condition (ii) of this theorem. 

\begin{lem}\label{lem:sat mod}
Assume the same conditions as in Theorem \ref{thm:mon loc ET}.
Let $X$ and $Y$ be objects in $\catC$.
If $X\iso Y$ in $\catC/\catN$,
then $[X] \equiv [Y] \bmod\M_{\catN}$.
\end{lem}
\begin{proof}
Let $f\colon X\simto Y$ be an isomorphism in $\catC/\catN$.
We have morphisms $g\colon X\to A$ in $\catC$ and $s \colon Y \to A$ in $S_{\catN}$
such that $f=Q(s)^{-1}Q(g)$ by Remark \ref{rmk:roof}.
Then $Q(g)$ is also an isomorphism in $\catC/\catN$,
and thus $g\in S_{\catN}$ because $S_{\catN}$ is saturated.
Therefore it is enough to show that $[X] \equiv [Y] \bmod \M_{\catN}$
if there is a morphism $s\colon X\to Y$ in $S_{\catN}$.
Since a morphism in $S_\catN$ is a finite composition of morphisms in $\calL$ and $\calR$,
we may assume that $s\in \calL$ or $s\in \calR$.
If $s \in \calL$, then there is a conflation $X \xr{s} Y \to N \dto$ with $N \in \NN$. Then $[X] + [N] = [Y]$ holds in $\M(\CC)$, which implies $[X] \equiv [Y] \bmod \M_\NN$.
The case $s \in \calR$ is similar.
\end{proof}

\begin{proof}[Proof of Theorem \ref{thm:mon loc ET}]
The homomorphism $\M(Q) \colon \M(\catC) \to \M(\catC/\catN)$
induces a homomorphism $\phi \colon \M(\catC)/\M_{\catN} \to \M(\catC/\catN)$ satisfying $\phi([A] \bmod \M_{\catN}) = [Q(A)]$ by Proposition \ref{prp:mon quot}
since $Q(\catN)=0$ holds by Proposition \ref{prop:ex-loc-univ2}.
Clearly $\phi$ is surjective since $Q$ is the identity on objects.
We have to show that $\phi$ is injective.
Suppose that $[Q(A)]=[Q(B)]$ in $\M(\catC/\catN)$ for $A,B\in\catC$, or equivalently, $A \ceq B$ in $\CC/\NN$.
We want to show that $[A]\equiv [B] \bmod \M_{\catN}$.
It suffices to prove it for the case $A \crel B$ in $\catC/\catN$.

Since $A \crel B$ in $\catC/\catN$, there exists a conflation $X\to Y \to Z \dto $ in $\catC/\catN$
such that either (a) $X\oplus Z \iso A$ and $Y\iso B$ in $\catC/\catN$ or
(b) $X\oplus Z \iso B$ and $Y\iso A$ in $\catC/\catN$.
Clearly we only have to deal with the case (a).
By Fact \ref{fct:NOS loc} (2),
we can find a conflation $X'\xr{f} Y' \xr{g} Z' \dto$ in $\catC$ such that we have an isomorphism of conflations
\[
\begin{tikzcd}
X \arr{r} \arr{d,"\sim"'sloped,dash} & Y \arr{r} \arr{d,"\sim"'sloped,dash} & Z \arr{r,dashed} \arr{d,"\sim"'sloped,dash} & {}\\
X' \arr{r,"Q(f)"} & Y' \arr{r,"Q(g)"} & Z' \arr{r,dashed} & {}
\end{tikzcd}
\]
in $\catC/\catN$.
Thus, by Lemma \ref{lem:sat mod}, we have
\[
[A]= [X\oplus Z] \equiv [X'\oplus Z'] = [Y'] \equiv [Y] = [B] \bmod \M_{\catN}.
\]
This proves the injectivity of $\phi$.
\end{proof}

In the following, we discuss the kernel of the localization functor $Q \colon \CC \to \CC/\NN$
and study when $S_{\catN}$ is saturated.
The authors would like to thank Hiroyuki Nakaoka for sharing the results on the kernel of the localization functor.
We prepare some lemmas.
\begin{lem}\label{lem:Ker 1}
Suppose {\upshape (i)} of Condition \ref{cond:NOS loc} and let $A \in \catC$.
If there exists a split monomorphism $u\colon A \to N$ with $N \in \NN$,
then $A\in \catN$.
\end{lem}
\begin{proof}
The morphism $\ol{u}$ is a split monomorphism in $\ol{\catC}$ since so is $u$ in $\catC$.
Then the split monomorphism $A \xr{\ol{u}}N \iso 0$ should be an isomorphism in $\ol{\catC}$.
Thus (i) of Condition \ref{cond:NOS loc} implies that $u \colon A \to N$ belongs to $S_\NN$.
By composing this with $N \to 0$, which is in $S_\NN$, we obtain that $A \to 0$ is in $S_{\catN}$.
Dually, $0\to A$ is also in $S_{\catN}$
by considering a retraction $q\colon N\to A$ of $u$.
By \eqref{eq:thick N}, we conclude that $A\in\catN$.
\end{proof}

\begin{lem}\label{lem:Ker 2}
Suppose {\upshape (i)} of Condition \ref{cond:NOS loc}. Then the following hold.
\begin{enua}
\item
Let $N \xr{k} Y \xr{r} Z \xdr{\delta}$ be a conflation
with $N\in\catN$ and $\ol{r}=0$ in $\ol{\catC}$.
Then $Y\in \catN$.
\item
Dually,
let $X \xr{\ell} Y \xr{c} N \xdr{\delta}$ be a conflation
with $N\in\catN$ and $\ol{\ell}=0$ in $\ol{\catC}$.
Then $Y\in \catN$.
\end{enua}
In each case, all objects in the above conflations belong to $\catN$
since $\NN$ is a thick subcategory.
\end{lem}
\begin{proof}
We only prove (1) since (2) follows dually.
The assumption $\ol{r}=0$ implies that it has a factorization
$Y\xr{f} N' \xr{g} Z$ in $\catC$ with $N'\in\catN$.
Then there is a morphism of conflations
\[
\begin{tikzcd}
N \arr{d,equal} \arr{r,"a"} & M \arr{r,"c"} \arr{d,"b"} & N' \arr{d,"g"} \arr{r,"g^*\delta", dashed} & {}\\
N \arr{r,"k"'} & Y \arr{r,"r"'} & Z \arr{r,"\delta"', dashed} & {}.
\end{tikzcd}
\]
The top row implies $M\in \catN$ because $\catN$ is extension-closed.
We can choose $b\colon M\to Y$ so that
$M \xr{\begin{bsmatrix}c \\ b\end{bsmatrix}} Y\oplus N' \xr{\begin{bsmatrix}g & -r\end{bsmatrix}} Z \dto$ is a conflation by \cite[Proposition 1.20]{LN}.
This implies that the following diagram is a weak pullback diagram,
and hence a section $\phi \colon Y \to M$ of $b$ is induced:
\[
\begin{tikzcd}
Y \arr{rrd,bend left,"f"} \arr{rdd,bend right,equal} \arr{rd,"\phi",dashed}&&\\
&M \arr{r,"c"} \arr{d,"b"'} \wpb & N' \arr{d,"g"}\\
&Y \arr{r,"r"'} & Z.
\end{tikzcd}
\]
Since $M \in \NN$, we obtain $Y\in\catN$ by Lemma \ref{lem:Ker 1}.
\end{proof}

Set $\ol{\calL}:=p(\calL)$ and $\ol{\calR}:=p(\calR)$.
We have the following interpretation of (EL3) for the case of percolating subcategories.
\begin{lem}\label{lem:P cond}
Let $\catN$ be a percolating subcategory,
and let $f\colon A \to B$ be a morphism in $\catC$.
\begin{enua}
\item
$\ol{f}$ is a monomorphism in $\ol{\catC}$ if and only if
$\Ker\left(\CC(-,A)\xr{f \circ (-)}\CC(-,B)\right) \subseteq [\catN](-,A)$ holds.
\item
$\ol{f}$ is an epimorphism in $\ol{\catC}$ if and only if
$\Ker\left(\CC(B,-)\xr{(-) \circ f}\CC(A,-)\right) \subseteq [\catN](B,-)$ holds.
\end{enua}
\end{lem}
\begin{proof}
Since (2) is dual of (1), we only prove (1).
To show the ``only if'' part, suppose that $\ol{f}$ is a monomorphism in $\ol{\catC}$.
Let $x\in\Ker\left(\CC(X,A)\xr{f \circ (-)}\CC(X,B)\right)$,
that is, $x\colon X\to A$ satisfies $fx=0$ in $\catC$.
Because $\ol{f}$ is a monomorphism, we have $\ol{x}=0$ in $\ol{\catC}$.
This means $x\in [\catN](X,A)$.

To show the ``if'' part, suppose that $\Ker\left(\CC(-,A)\xr{f \circ (-)}\CC(-,B)\right) \subseteq [\catN](-,A)$ holds.
Let $x\colon X\to A$ be a morphism in $\catC$ with $\ol{fx}=0$ in $\ol{\catC}$, that is, $fx$ factors through an object in $\NN$.
Then $fx$ has a deflation-inflation factorization $X \xr{d} N \xr{i} B$
with $N\in\catN$ since $\catN$ is percolating.
In particular, we have a conflation $K \xr{\ell} X \xr{d} N \dto$.
Note that $\ell\in \calL$.
Then we have $f x \ell=i d \ell=0$,
and by the assumption, we conclude that $\ol{x\ell}=0$ in $\ol{\catC}$.
Since $\ol{\ell}\in \ol{\calL}$ is an epimorphism in $\ol{\catC}$ by \cite[Lemma 4.7 (2)]{NOS},
we have $\ol{x}=0$.
This proves that $\ol{f}$ is a monomorphism.
\end{proof}

\begin{cor}\label{cor:regular}
Suppose that either of the following conditions holds.
\begin{enur}
\item
$\catN$ is a biresolving subcategory.
\item
$\catN$ satisfies {\upshape (EL1)--(EL3)}.
\end{enur}
Then any morphism in $\ol{S_{\catN}}$ is both a monomorphism and an epimorphism in $\ol{\catC}$.
\end{cor}
\begin{proof}
By \cite[Lemma 4.7]{NOS}, every morphism in $\ov{\LL}$ (resp. $\ov{\RR}$) is an epimorphism (resp. a monomorphism) in $\ov{\CC}$.
Thus it suffices to show that every morphism in $\ov{\LL}$ is a monomorphism in each case, since the proof for $\ov{\RR}$ follows dually.

(i)
This case is precisely \cite[Lemma 4.24]{NOS}.

(ii)
In this case, $\NN$ is a percolating subcategory by (EL1).
Thus the assertion follows from (EL3) and Lemma \ref{lem:P cond},
since every morphism in $\LL$ is an inflation in $\CC$ by definition.
\end{proof}
Now we can show that the kernel of $\CC \to \CC/\NN$ coincides with $\NN$ under some conditions. Recall that a kernel $\ker F$ of an additive functor $F \colon \CC \to \DD$ consists of objects $C \in \CC$ satisfying $F(C) \iso 0$ in $\DD$.
\begin{prp}\label{prp:Ker}
Suppose that Condition \ref{cond:NOS loc} is satisfied, and that every morphism in $\ol{S_{\catN}}$ is a monomorphism in $\ol{\catC}$.
Then $\Ker(Q\colon \catC \to \catC/\catN)=\catN$ holds.
\end{prp}
\begin{proof}
We have $\Ker(Q\colon \catC \to \catC/\catN) \supseteq \catN$
by Proposition \ref{prop:ex-loc-univ2}.
Conversely, suppose that $X\in \catC$ satisfies $Q(X)\iso 0$ in $\catC/\catN$.
Then there exists $Y\in \catC$ satisfying $0\in \ol{S_{\catN}}(X,Y)$
since $\catC/\catN$ is constructed as the category of fractions $\ol{S_{\catN}}^{-1}\ol{\catC}$ by Fact \ref{fct:NOS loc}.
Hence we obtain $s\in S_{\catN}(X,Y)$ satisfying $\ol{s}=0$ in $\ov{\CC}$.
By the construction of $S_\NN$, we can write either $s=t\ell$ or $s=ur$ 
for some $t,u\in S_{\catN}$, $\ell\in\calL$, and $r\in\calR$.
We consider the case $s=t\ell$.
Since $\ol{t}$ is a monomorphism in $\ol{\catC}$ by the assumption,
$\ov{s} = \ov{t} \ov{\ell} = 0$ implies $\ol{\ell}=0$.
Thus Lemma \ref{lem:Ker 2} (2) implies $X\in\catN$.
The case $s = ur$ can be proved similarly using Lemma \ref{lem:Ker 2} (1).
\end{proof}
As a consequence, we can describe the kernel of the localization for the cases we are interested in.
\begin{cor}\label{cor:Ker}
Suppose that either of the following conditions holds.
\begin{enur}
\item
$\catN$ is a biresolving subcategory.
\item
$\catN$ satisfies {\upshape (EL1)--(EL3)}.
\end{enur}
Then $\Ker(Q\colon \catC \to \catC/\catN)=\catN$ holds.
\end{cor}
\begin{proof}
For each case, Condition \ref{cond:NOS loc} is satisfied by Fact \ref{fct:loc tri} and \ref{fct:loc ex}, and every morphism in $\ov{S_\NN}$ is a monomorphism by Corollary \ref{cor:regular}. Thus the assertion follows from Proposition \ref{prp:Ker}.
\end{proof}

Using this result, we next consider whether $S_\NN$ is saturated. Assume Condition \ref{cond:NOS loc}, and 
let $Q \colon \catC \to \catC/\catN$ be the localization functor.
Recall that $S_{\catN}$ is \emph{saturated} if
$f\in S_{\catN}$ holds for any morphism $f$ in $\catC$ such that $Q(f)$ is an isomorphism in $\catC/\catN$.
\begin{prp}\label{prp:sat tri}
If $\catN$ is a biresolving subcategory of $\CC$,
then $S_{\catN}$ is saturated.
\end{prp}
\begin{proof}
Let $f\colon X \to Y$ be a morphism in $\catC$ with $Q(f)$ is an isomorphism in $\catC/\catN$.
Since $\catN$ is biresolving,
there is a conflation $X \xr{x} N \xr{y} C \xdr{\delta}$ with $N\in\catN$.
Then there is a morphism of conflations
\begin{equation}\label{diag:po1}
\begin{tikzcd}
X \arr{r,"x"} \arr{d,"f"'} & N \arr{r,"y"} \arr{d,"g"} & C \arr{d,equal} \arr{r, dashed} &{}\\
Y \arr{r,"a"'} & M \arr{r,"b"'} & C \arr{r, dashed} &{}.
\end{tikzcd}
\end{equation}
We can choose a morphism $g\colon N\to M$ so that
\begin{equation}\label{eq:sat tri}
X \xr{\begin{bsmatrix}
x\\ f
\end{bsmatrix}}
N \oplus Y
\xr{\begin{bsmatrix}
g & -a
\end{bsmatrix}}
M \xdr{b^*\delta}
\end{equation}
is a conflation by \cite[Proposition 1.20]{LN}.
We have a factorization $f = \begin{sbmatrix} 0 & \id_Y \end{sbmatrix} \circ \begin{sbmatrix} x\\ f\end{sbmatrix}$ for $\begin{sbmatrix} 0 & \id_Y \end{sbmatrix} \colon N \oplus Y \to Y$ and $\begin{sbmatrix} x\\ f\end{sbmatrix} \colon X \to N \oplus Y$.
Clearly, $\begin{sbmatrix} 0 & \id_Y \end{sbmatrix}$ belongs to $\calR$.
Thus it is enough to show that $\begin{sbmatrix} x\\ f\end{sbmatrix}$ belongs to $\calL$,
that is, $M\in\catN$.
Applying $Q$ to \eqref{diag:po1},
we have a morphism of conflation
\[
\begin{tikzcd}[column sep = large]
X \arr{r,"Q(x)"} \arr{d,"Q(f)"'} & N \arr{r,"Q(y)"} \arr{d,"Q(g)"} & C \arr{d,equal} \arr{r, dashed} &{}\\
Y \arr{r,"Q(a)"'} & M \arr{r,"Q(b)"'} & C \arr{r, dashed} &{}.
\end{tikzcd}
\]
in $\catC/\catN$ by Remark \ref{rmk:ex funct}.
Since $Q(f)$ is an isomorphism by the assumption,
$Q(g)$ is an isomorphism in $\catC/\catN$ by \cite[Corollary 3.6]{NP}.
Thus we obtain $M \iso N \iso 0$ in $\catC/\catN$, which shows $M\in \Ker Q$.
Since $\Ker Q = \NN$ holds by Corollary \ref{cor:Ker}, we obtain $M \in \NN$.
This proves $\begin{sbmatrix} x\\ f\end{sbmatrix}\in \calL$ by \eqref{eq:sat tri}.
\end{proof}
Unfortunately, for the percolating case, the authors do not know whether the condition (EL1)--(EL3) implies that $S_\NN$ is saturated, although we have the following criterion.
\begin{prp}\label{prp:sat ex}
Suppose that $\catN$ satisfies {\upshape (EL1)--(EL3)}.
Then the following conditions are equivalent.
\begin{enua}
\item
$S_{\catN}$ is saturated.
\item
For a morphism $f\colon X\to Y$ in $\catC$, if $Q(f)$ is an isomorphism in $\catC/\catN$, then $f$ is admissible in $\CC$, that is, there exists a deflation-inflation factorization of $f$. 
\end{enua}
\end{prp}
\begin{proof}
(1) $\Rightarrow$ (2):
We have $S_{\catN}=\calL\circ\calR$ by \cite[Lemma 4.37]{NOS}.
Hence if $Q(f)$ is an isomorphism in $\catC/\catN$,
then $f\in S_{\catN}=\calL\circ\calR$ since $S_\NN$ is saturated,
which proves (2) since a morphism in $\calL$ (resp. $\calR$) is an inflation (resp. a deflation).

(2) $\Rightarrow$ (1):
Let $f\colon X \to Y$ be a morphism in $\catC$ such that $Q(f)$ is an isomorphism in $\catC/\catN$.
There is a deflation-inflation factorization $X \xr{d} A \xr{i} Y$ of $f$ by (2).
In particular,
we have conflations $K\xr{k}X \xr{d} A \dto$ and $A\xr{i} Y \xr{c} C \dto$.
Since $Q$ is an exact functor, we obtain conflations $Q(K) \xr{Q(k)} Q(X) \xr{Q(d)} Q(A) \dto$ and $Q(A)\xr{Q(i)} Q(Y) \xr{Q(c)} Q(C) \dto$ in $\CC/\NN$.
In particular, $Q(d)$ is a deflation in $\catC/\catN$.
Thus $Q(d)$ is an epimorphism because $\catC/\catN$ is an exact category by Fact \ref{fct:loc ex}.
On the other hand,
$Q(d)$ is a split monomorphism since $Q(f)=Q(i) \circ Q(d)$ is an isomorphism.
Thus $Q(d)$ is an isomorphism in $\catC/\catN$.
Dually, $Q(i)$ is also an isomorphism in $\catC/\catN$.
Hence both $Q(K)$ and $Q(C)$ are isomorphic to $0$ in $\catC/\catN$,
which implies $K,C \in \Ker Q = \catN$ by Corollary \ref{cor:Ker}.
We conclude that $d\in\calR$ and $i\in\calL$,
and hence $f=i\circ d \in S_{\catN}$.
\end{proof}
Therefore, if one assumes in addition that (EL4) holds, that is, $\CC$ is admissible, then we can show the saturatedness of $S_\NN$. We can summarize our results of the saturatedness as follows.
\begin{cor}\label{cor:sat}
Suppose that either of the following conditions holds.
\begin{enur}
\item
$\catN$ is a biresolving subcategory.
\item
$\catN$ satisfies {\upshape (EL1)--(EL4)}.
\end{enur}
Then $S_{\catN}$ is saturated.
\end{cor}
\begin{proof}
  (i) This is Proposition \ref{prp:sat tri}.

  (ii)
  Since $\NN$ satisfies (EL1)--(EL3), we can apply Proposition \ref{prp:sat ex}. In addition, every morphism in $\CC$ is admissible by (EL4), and hence the condition in Proposition \ref{prp:sat ex} (2) automatically holds. Therefore, $S_\NN$ is saturated.
\end{proof}
As a consequence, we can deduce the following result for the cases we are interested in.
\begin{cor}\label{cor:mon loc ET}
Suppose that either of the following conditions holds.
\begin{enur}
\item
$\catN$ is a biresolving subcategory.
\item
$\catN$ satisfies {\upshape (EL1)--(EL4)}.
\end{enur}
Then the monoid homomorphism $\M(Q) \colon \M(\CC) \to \M(\CC/\NN)$ for the exact localization $Q \colon \CC \to \CC/\NN$ induces an isomorphism of monoids:
\[
\M(\catC)/\M_{\catN} \simto \M(\catC/\catN).
\]
\end{cor}
\begin{proof}
This follows from Theorem \ref{thm:mon loc ET} and Corollary \ref{cor:sat}.
\end{proof}

\begin{rmk}
Ogawa \cite{ogawa} developed localization theory of triangulated categories 
with respect to an \emph{extension-closed subcategory}.
We briefly introduce his results and explain the relation to ours.
Let $\catT$ be a triangulated category and $\catN$ an extension-closed subcategory of $\catT$.
We consider $\catT$ as the extriangulated category $(\catT, \bbE, \frs)$.
Then there is a subfunctor $\bbE_{\catN} \subseteq \bbE$ satisfying the following:
\begin{itemize}
\item 
$\catT^{\catN}:=(\catT,\bbE_{\catN}, \frs_{|\bbE_{\catN}})$ is an extriangulated category.
\item
$\catN$ is a thick subcategory of $\catT^{\catN}$
and satisfies Condition \ref{cond:NOS loc}.
In particular, the exact localization $\catT^{\catN}/\catN$ exists by Fact \ref{fct:NOS loc}.
\item
$S_\catN$ is saturated in $\catT^{\catN}$.
\end{itemize}
Thus we can apply Theorem \ref{thm:mon loc ET} and obtain a monoid isomorphism
$\M\left(\catT^{\catN}\right)/\M_{\catN} \simto \M\left(\catT^{\catN}/\catN\right)$
for any extension-closed subcategory $\catN$ of $\catT$.
However, we do not know a relation between $\M(\catT)$ and $\M\left(\catT^{\catN}\right)$
at the moment.
\end{rmk}

\subsection{Applications}
In this subsection, we give applications of our description Theorem \ref{thm:mon loc ET} of the Grothendieck monoid of the exact localization.

First, by applying this to the abelian case, we obtain the following consequence on the Serre quotient of an abelian category.
\begin{corollary}\label{cor:serre-quot}
  Let $\AA$ be a skeletally small abelian category, $\SS$ a Serre subcategory of $\AA$, and $\iota \colon \SS \hookrightarrow \AA$ and $Q \colon \AA \to \AA/\SS$ the inclusion and the localization functor respectively. Then the following holds.
  \begin{enumerate}
    \item $\M(\iota) \colon \M(\SS) \to \M(\AA)$ is injective, so we have an isomorphism $\M(\SS) \iso \M_\SS$.
    \item $\M(Q) \colon \M(\AA) \to \M(\AA/\SS)$ induces an isomorphism of monoids $\M(\AA)/\M_\SS \cong \M(\AA/\SS)$.
  \end{enumerate}
\end{corollary}
\begin{proof}
  (1) This is Proposition \ref{prp:Serre inj}.

  (2)
  Example \ref{ex:two-cond} (3) shows that $\SS$ satisfies (EL1)--(EL4). Therefore, we can apply Corollary \ref{cor:mon loc ET} to this setting.
\end{proof}

This may be seen as a ``short exact sequence'' of monoids:
\begin{equation}\label{eq:gro-mon-ses}
  \begin{tikzcd}
    0 \rar & \M(\SS) \rar & \M(\AA) \rar & \M(\AA/\SS) \rar & 0.
  \end{tikzcd}
\end{equation}

This description of $\M(\AA/\SS)$ gives the following description of Serre subcategories of $\AA/\SS$, which seems to be a folklore.
\begin{corollary}
  Let $\AA$ be a skeletally small abelian category and $\SS$ a Serre subcategory of $\AA$. Then there is a bijection between the following two sets:
  \begin{enumerate}
    \item $\Serre (\AA/\SS)$.
    \item $\{ \SS' \in \Serre \AA \mid \SS \subseteq \SS' \}$.
  \end{enumerate}
\end{corollary}
\begin{proof}
  There is a bijection between $\Serre(\AA/\SS)$ and $\Face \M(\AA/\SS)$ by Proposition \ref{prop:serre-face-bij}. On the other hand, $\M(\AA/\SS)$ is isomorphic to the quotient monoid $\M(\AA)/\M_\SS$ by Corollary \ref{cor:serre-quot} (2).
  Thus Proposition \ref{prp:face quot} (3) shows that $\Face (\M(\AA)/\M_\SS)$ is in bijection with the set of faces of $\M(\AA)$ containing $\M_\SS$. Since $\Face \M(\AA)$ are in bijection with $\Serre \AA$ by Proposition \ref{prop:serre-face-bij} again and since $\SS$ corresponds to $\M_\SS$ in this bijection, we conclude that $\Face(\M(\AA)/\M_\SS)$ are in bijection with (2).
  The situation can be summarized as the following figure.
  \[
    \begin{tikzcd}
      \Serre(\AA/\SS) \ar[rr, leftrightarrow, dashed] \dar[leftrightarrow] & & \{ \SS' \in \Serre \AA \mid \SS \subseteq \SS' \} \rar[symbol=\subseteq] \dar[leftrightarrow]
      & \Serre \AA \dar[leftrightarrow] \\
      \Face \M(\AA/\SS) \rar[leftrightarrow] &
      \Face(\M(\AA)/\M_\SS) \rar[leftrightarrow] &
      \{ F \in \Face \M(\AA) \mid \M_\SS \subseteq F \} \rar[symbol=\subseteq]
      & \Face \M(\AA)
    \end{tikzcd}
  \]
\end{proof}

It is well-known that the similar sequence to \eqref{eq:gro-mon-ses} for the Grothendieck group is only right exact. Actually, we can deduce it using our result as follows.
\begin{cor}\label{cor:loc-gr}
Suppose that Condition \ref{cond:NOS loc} are satisfied and $S_{\catN}$ is saturated (e.g. the assumption in Corollary \ref{cor:sat} holds).
Then the following sequence is exact:
\[
\begin{tikzcd}
  \K_0(\catN) \rar["{\K_0(\iota)}"] & \K_0(\catC) \rar["{\K_0(Q)}"] &
  \K_0(\catC/\catN) \rar & 0,
\end{tikzcd}
\]
where $\iota \colon \catN \inj \catC$ is the inclusion functor
and $Q \colon \catC \to \catC/\catN$ is the exact localization functor.
\end{cor}
\begin{proof}
The diagram
\[
\begin{tikzcd}
  \M(\catN) \rar[shift left, "\M(\iota)"] \rar[shift right, "0"'] &
  \M(\catC) \rar["\M(Q)"] & \M(\catC/\catN)
\end{tikzcd}
\]
is a coequalizer diagram in $\Mon$
by Theorem \ref{thm:mon loc ET} (see also Proposition \ref{prp:mon quot}).
Since the group completion functor $\gp \colon \Mon \to \Ab$
is a left adjoint of the forgetful functor $\Ab \inj \Mon$ by Proposition \ref{prp:adj grp},
it preserves colimits.
Hence we have a coequalizer diagram
\[
  \begin{tikzcd}
    \K_0(\catN) \rar[shift left, "\K_0(\iota)"] \rar[shift right, "0"'] &
    \K_0(\catC) \rar["\K_0(Q)"] & \K_0(\catC/\catN)
  \end{tikzcd}
\]
in $\Ab$ because $\gp\circ \M = \K_0$ by Remark \ref{rem:k0-as-gp}.
This is nothing but the claimed exact sequence.
\end{proof}
\section{Intermediate subcategories of the derived category}\label{sec:5}
In this section, we give a concrete example of an extriangulated category and its thick subcategory which satisfies the condition (EL1)--(EL4) in Fact \ref{fct:loc ex}, \emph{intermediate subcategories of the derived category}, and compute its Grothendieck monoid and the exact localization.

\emph{Throughout this section, $\AA$ denotes a skeletally small abelian category, and we assume that every category, functor, and subcategory is additive.}
We denote by $\D^\b(\AA)$ the bounded derived category of $\AA$, which we regard as an extriangulated category. We also denote by $H^i \colon \D^\b(\AA) \to \AA$ the $i$-th cohomology functor. We often identify $\AA$ with the essential image of the canonical embedding $\AA \hookrightarrow \D^\b(\AA)$, that is, the subcategory of $\D^\b(\AA)$ consisting of $X$ such that $H^i(X) = 0$ for $i \neq 0$.

The following observation is useful throughout this section, and can be proved easily by using the truncation functor.
\begin{lemma}\label{lem:intermediate}
  Let $\AA$ be a skeletally small abelian category and $\BB$ and $\BB'$ subcategories of $\AA$.
  \begin{enumerate}
    \item We have the following equality in $\D^\b(\AA)$:
    \[
      \BB[1] * \BB' = \{
        X \in \D^\b(\AA) \mid 
        \text{$H^i(X) = 0$ for $i \neq 0,-1$, $H^{-1}(X) \in \BB$, and $H^0(X) \in \BB'$} \}.
    \]
    \item For every $X \in \AA[1] * \AA$, we have the following conflation in $\D^\b(\AA)$, which is natural on $X$:
    \[
      \begin{tikzcd}
        H^{-1}(X)[1] \rar & X \rar & H^0(X) \rar[dashed] & {}. 
      \end{tikzcd}
    \]
    Moreover, every conflation $A[1] \to X \to B \dto$ with $A, B\in \AA$ is isomorphic to the above conflation.
    \item For every $X \in \AA[1] * \AA$, there is some $Y \in \D^\b(\AA)$ with $X \iso Y$ in $\D^\b(\AA)$ such that $Y$ is a complex concentrated in degree $0$ and $-1$.
  \end{enumerate}
\end{lemma}

\subsection{Classification of intermediate subcategories}
Let us introduce the main topic of this section.
\begin{definition}
  Let $\AA$ be a skeletally small abelian category. An \emph{intermediate subcategory} of $\D^\b(\AA)$ is a subcategory $\CC$ satisfying the following conditions:
  \begin{enumerate}
    \item $\AA \subseteq \CC \subseteq \AA[1] * \AA$ holds.
    \item $\CC$ is closed under extensions in $\D^\b(\AA)$.
    \item $\CC$ is closed under direct summands in $\D^\b(\AA)$.
  \end{enumerate}
 By (2), we regard an intermediate subcategory as an extriangulated category.
\end{definition}
See Figure \ref{fig:inter} for the intuition of this notion. The simplest example of an intermediate subcategory is $\AA$ itself, and generally intermediate subcategories are larger than $\AA$ but not too large so that they are contained in $\AA[1] * \AA$.

First, we will see that intermediate subcategories can be described using torsionfree classes in $\AA$.
Here a subcategory $\FF$ of $\AA$ is called a \emph{torsionfree class} if $\FF$ is closed under subobjects and extensions.
Note that torsionfree classes do not necessarily correspond to torsion classes and torsion pairs in general.
\begin{theorem}\label{thm:inter-torf-bij}
  Let $\AA$ be a skeletally small abelian category. Then the following hold.
  \begin{enumerate}
    \item If $\FF$ is a torsionfree class in $\AA$, then $\FF[1] * \AA$ is an intermediate subcategory of $\D^\b(\AA)$.
    \item If $\CC$ is an intermediate subcategory of $\D^\b(\AA)$, then $H^{-1}(\CC)$ is a torsionfree class in $\AA$, and we have $\CC = H^{-1}(\CC)[1] * \AA$.
    \item The assignments given in {\upshape (1)} and {\upshape (2)} give bijections between the set of torsionfree classes in $\AA$ and that of intermediate subcategories of $\D^\b(\AA)$.
  \end{enumerate}
\end{theorem}
\begin{proof}
  (1)
  Let $\FF$ be a torsionfree class in $\AA$. Recall from Lemma \ref{lem:intermediate} (1) that $\FF[1] * \AA$ consists of $X \in \D^\b(\AA)$ satisfying $H^i(X) = 0$ for $i \neq 0, -1$ and $H^{-1}(X) \in \FF$.
  We clearly have $\AA \subseteq \FF[1] * \AA \subseteq \AA[1] * \AA$. Next, we show that $\FF[1] * \AA$ is closed under extensions in $\D^\b(\AA)$.
  Suppose that there is a conflation (triangle)
  \[
    \begin{tikzcd}
      X \rar & Y \rar & Z \rar[dashed] & \ 
    \end{tikzcd}
  \]
  in $\D^\b(\AA)$ with $X, Z \in \FF[1] * \AA$. It clearly suffices to show that $H^{-1}(Y) \in \FF$. Since $H$ is a cohomological functor, we have the following exact sequence in $\AA$:
  \[
    \begin{tikzcd}
      0 = H^{-2}(Z) \rar & H^{-1}(X) \rar & H^{-1}(Y) \rar & H^{-1}(Z).
    \end{tikzcd}
  \]
  Since we have $H^{-1}(X), H^{-1}(Z) \in \FF$ and $\FF$ is closed under subobjects and extensions, we obtain $H^{-1}(Y) \in \FF$. Since clearly $H^i(Y) = 0$ for $i \neq 0, -1$, we obtain $Y \in \FF[1] * \AA$, so $\CC$ is closed under extensions in $\D^\b(\AA)$.

  Since $H^i(X \oplus Y) \iso H^i(X) \oplus H^i(Y)$ holds and $\FF$ is closed under direct summands in $\AA$, we can easily check that $\FF[1] * \AA$ is closed under direct summands in $\D^\b(\AA)$.

  (2)
  Suppose that $\CC$ is an intermediate subcategory of $\D^\b(\AA)$, that is, $\AA \subseteq \CC \subseteq \AA[1] * \AA$ holds and $\CC$ is closed under extensions and direct summands in $\D^\b(\AA)$.
  We first show the equality $\CC = H^{-1}(\CC)[1] * \AA$. Lemma \ref{lem:intermediate} (2) implies $\CC \subseteq H^{-1}(\CC)[1] * \AA$. To show the converse, it suffices to show $H^{-1}(\CC)[1] \subseteq \CC$ since $\AA \subseteq \CC$ and $\CC$ is closed under extensions.

  Take any $C \in \CC$ and we will see $H^{-1}(C)[1] \in \CC$. By Lemma \ref{lem:intermediate} (3), we may assume that $C = C^{\bullet}$ is a complex concentrated in degree $0$ and $-1$, that is, $C$ is of the form $C = [\cdots \to 0 \to C^{-1} \to C^0 \to 0 \to \cdots]$.
  Lemma \ref{lem:intermediate} (2) gives the following triangle:
  \[
    \begin{tikzcd}
      H^{-1}(C)[1] \rar["f"] & C \rar & H^0(C) \rar[dashed] & \ .
    \end{tikzcd}
  \]
  On the other hand, define the cochain map $g \colon C^0 \to C$ by the following:
  \[
    \begin{tikzcd}
      C^0 \dar["g"'] & \cdots \rar & 0 \rar\dar & 0 \rar\dar & C^0 \rar\dar[equal] & 0 \rar\dar & \cdots \\
      C & \cdots \rar & 0 \rar & C^{-1} \rar & C^0 \rar & 0 \rar & \cdots \\
    \end{tikzcd}
  \]
  Then by consider the mapping cocone of $[f,\ g] \colon H^{-1}(C) [1] \oplus C^0 \to C$, we obtain the following triangle in $\D^\b(\AA)$:
  \[
    \begin{tikzcd}
      A \rar & H^{-1}(C)[1] \oplus C^0 \rar["{[f,\ g]}"] & C \rar & A[1]. 
    \end{tikzcd}
  \]
  We claim that $A$ belongs to $\AA$. To show that, consider the cohomology long exact sequence induced from the above triangle:
  \[
    \begin{tikzcd}[row sep = 0]
      0 = H^{-2}(C) \rar & H^{-1}(A) \rar & H^{-1}(H^{-1}(C)[1] \oplus C^0) \rar["\ov{f}"] & H^{-1}(C) \\
      \rar & H^0(A) \rar & H^0(H^{-1}(C)[1] \oplus C^0) \rar["\ov{g}"] & H^0(C)\\
      \rar & H^1(A) \rar & H^1(H^{-1}(C)[1] \oplus C^0) = 0
    \end{tikzcd}
  \]
  Clearly we have $H^i(A) = 0$ for $i \neq -1, 0, 1$. On the other hand, $\ov{f}$ is isomorphic to $H^{-1}(f)$, which is an isomorphism. Hence $H^{-1}(A) = 0$ holds.
  Moreover, $\ov{g}$ is isomorphic to $H^0(g) \colon C^0 \to H^0(C)$, which is surjective. Hence $H^{1}(A) = 0$ holds.
  Therefore, we have shown that $H^i(A) = 0$ for $i \neq 0$. Thus $A \in \AA$ holds.

  It follows that $H^{-1}(C)[1] \oplus C^0$ belongs to $\CC$, since $\AA \subseteq \CC$ and $\CC$ is closed under extensions. Therefore, we obtain $H^{-1}(C)[1] \in \CC$ because $\CC$ is closed under direct summands.
  Thus we have shown that $\CC = H^{-1}(\CC)[1] * \AA$ holds.

  We claim that $H^{-1}(\CC)$ is a torsionfree class in $\AA$. Put $\FF := H^{-1}(\CC)$ for simplicity, and then we have $\CC = \FF[1] * \AA$ by the above argument.
  Suppose that we have a short exact sequence
  \[
    \begin{tikzcd}
      0 \rar & X \rar & Y \rar & Z \rar & 0
    \end{tikzcd}
  \]
  in $\AA$. Note that this gives a triangle $X \to Y \to Z \to X[1]$ in $\D^\b(\AA)$.
  If $X$ and $Z$ belong to $\FF$, then $X[1]$ and $Z[1]$ belongs to $\FF[1] \subseteq \CC$. Therefore, the conflation
  \[
    \begin{tikzcd}
      X[1] \rar & Y[1] \rar & Z[1] \rar[dashed] & \ 
    \end{tikzcd}
  \]
  in $\D^\b(\AA)$ implies that $Y[1] \in \CC$ since $\CC$ is closed under extensions. Then we obtain $Y = H^{-1}(Y[1]) \in H^{-1}(\CC) = \FF$.

  Similarly, suppose that $Y$ belongs to $\FF$. Then we have a conflation
  \[
    \begin{tikzcd}
      Z \rar & X[1] \rar & Y[1] \rar[dashed] & \ 
    \end{tikzcd}
  \]
  in $\D^\b(\AA)$, and we have $Z \in \AA \subseteq \CC$ and $Y[1] \in \FF[1] \subseteq \CC$. Therefore, $X[1] \in \CC$ holds since $\CC$ is closed under extensions. Hence we obtain $X = H^{-1}(X[1]) \in H^{-1}(\CC) = \FF$. Therefore, $\FF$ is a torsionfree class in $\AA$.

  (3)
  Let $\FF$ be a torsionfree class in $\AA$. Then Lemma \ref{lem:intermediate} (1) shows $H^{-1}(\FF[1] * \AA) \subseteq \FF$. Since we clearly have $\FF \subseteq H^{-1}(\FF[1] * \AA)$, we obtain $H^{-1}(\FF[1] * \AA) = \FF$.
  Conversely, if $\CC$ is an intermediate subcategory of $\D^\b(\AA)$, then we have  $\CC = H^{-1}(\CC)[1] * \AA$ by the proof of (2). Thus the assignments in (1) and (2) are mutually inverse.
\end{proof}

\subsection{Grothendieck monoid of an intermediate subcategory}
By the previous subsection, an intermediate subcategory of $\D^\b(\AA)$ is of the form $\FF[1] * \AA$ for a torsionfree class $\FF$ in $\AA$.
Next, we will calculate the Grothendieck monoid of this extriangulated category. Note that the inclusion functor $\AA \hookrightarrow \FF[1] * \AA$ induces a monoid homomorphism $\M(\AA) \to \M(\FF[1] * \AA)$.
\begin{theorem}\label{thm:mon-of-inter}
  Let $\AA$ be a skeletally small abelian category and $\FF$ a torsionfree class in $\AA$, and put $\CC := \FF[1] * \AA$. Then the natural monoid homomorphism $\M(\AA) \to \M(\CC)$ induces the following isomorphism of monoids
  \[
    \M(\AA)_{\M_\FF} \iso \M(\CC),
  \]
  where the left hand side is the localization of $\M(\AA)$ with respect to $\M_\FF$ (see Definition \ref{def:mon-loc}).
\end{theorem}
\begin{proof}
  Let $\iota \colon \M(\AA) \to \M(\CC)$ be the natural monoid homomorphism. We check that $\iota$ satisfies the universal property of the localization of $\M(\AA)$ with respect to $\M_\FF$.

  First, we show that every element in $\iota(\M_\FF)$ is invertible. Take any $[F] \in \M_\FF$ with $F \in \FF$, then $\iota[F] = [F]$ in $\M(\CC)$. Now we have a conflation
  \[
    \begin{tikzcd}
      F \rar & 0 \rar & F[1] \rar[dashed] & \ 
    \end{tikzcd}
  \]
  in $\CC = \FF[1] * \AA$. Therefore, $[F] + [F[1]] = 0$ holds in $\M(\CC)$, that is, $\iota[F] \in \M(\CC)$ is invertible.

  Next, we will check the universal property. Let $\varphi \colon \M(\AA) \to M$ be a monoid homomorphism such that $\varphi[F]$ is invertible in $M$ for every $F \in \FF$.
  We have to show that there is a unique monoid homomorphism $\ov{\varphi} \colon \M(\CC)\to M$ which makes the following diagram commute:
  \[
    \begin{tikzcd}
      \M(\AA) \dar["\varphi"'] \rar["\iota"] & \M(\CC) \ar[ld, "\ov{\varphi}", dashed] \\
      M
    \end{tikzcd}
  \]
  We first check the uniqueness. Suppose that there is such a map $\ov{\varphi}$. Let $C \in \CC$ be any object. Then since $\CC = \FF[1] * \AA$, there is a conflation in $\CC$
  \[
    \begin{tikzcd}
      H^{-1}(C)[1] \rar & C \rar & H^0(C) \rar[dashed] & \
    \end{tikzcd}
  \]
  with $H^{-1}(C) \in \FF$ by Lemma \ref{lem:intermediate}. Thus $[C] = [H^0(C)] + [H^{-1}(C)[1]] = [H^0(C)] - [H^1(C)]$ in $\M(\CC)$.
  Since $\ov{\varphi}$ is a monoid homomorphism, it preserves the inverse. Therefore, we must have
  \[
    \ov{\varphi}[C] = \ov{\varphi}[H^0(C)] - \ov{\varphi}[H^1(C)]
    = \ov{\varphi} \iota [H^0(C)] - \ov{\varphi} \iota [H^1(C)]
    = \varphi[H^0(C)] - \varphi[H^1(C)].
  \]
  Therefore, the uniqueness of $\ov{\varphi}$ holds.

  Next, we will construct $\ov{\varphi}$. Consider the following map $\psi \colon \Iso\CC \to M$:
  \[
    \psi [X] := \varphi [H^0(X)] - \varphi [H^{-1}(X)].
  \]
  Note that $\varphi[H^{-1}(X)]$ has an inverse in $M$ since $H^{-1}(X) \in \FF$.
  We will show that $\psi$ respects conflations in $\CC$. Clearly $\psi[0] = 0$ holds.
  Take any conflation 
  \[
    \begin{tikzcd}
      X \rar & Y \rar & Z \rar[dashed] & \ 
    \end{tikzcd}
  \]
  in $\CC$. Since this is a triangle in $\D^\b(\AA)$, we obtain the following long exact sequence in $\AA$:
  \[
    \begin{tikzcd}[row sep = 0]
      H^{-2}(Z) = 0 \rar & H^{-1}(X) \rar & H^{-1}(Y) \rar & H^{-1}(Z) \\
      \rar & H^0(X) \rar & H^0(Y) \rar & H^0(Z) \rar & H^2(X) = 0
    \end{tikzcd}
  \]
  By decomposing this long exact sequence into short exact sequences in $\AA$, we can easily obtain the equality
  \[
    [H^{-1}(X)] + [H^{-1}(Z)] + [H^0(Y)] = [H^{-1}(Y)] + [H^0(X)] + [H^0(Z)]
  \]
  in $\M(\AA)$. Therefore, by applying $\varphi$, we obtain
  \[
    \varphi[H^{-1}(X)] + \varphi[H^{-1}(Z)] + \varphi[H^0(Y)] = \varphi[H^{-1}(Y)] + \varphi[H^0(X)] + \varphi[H^0(Z)],
  \]
  which clearly implies $\psi[Y] = \psi[X] + \psi[Z]$ in $M$. Thus $\psi$ respects conflations, so it induces a monoid homomorphism $\ov{\varphi} \colon \M(\CC) \to M$. Moreover, for any $A \in \AA$, we have $\ov{\varphi}\iota[A] = \varphi[H^0(A)] - \varphi[H^{-1}(A)] = \varphi[A] - 0 = \varphi[A]$, and hence $\ov{\varphi}\iota = \varphi$ holds.
\end{proof}

As an immediate consequence, we obtain the following example of an extriangulated category whose Grothendieck monoid is a group:
\begin{corollary}\label{cor:a1-a-k0}
  Let $\AA$ be a skeletally small abelian category. Then the natural homomorphism $\M(\AA) \to \M(\AA[1] * \AA)$ induces an isomorphism $\K_0(\AA) \iso \M(\AA[1] * \AA)$.
\end{corollary}
\begin{proof}
  By Theorem \ref{thm:mon-of-inter}, we have that $\M(\AA[1] * \AA)$ is isomorphic to the localization of $\M(\AA)$ with respect to the whole set $\M(\AA)$. This is nothing but the group completion of $\M(\AA)$, which is $\K_0(\AA)$ (see Remark \ref{rem:k0-as-gp}).
\end{proof}

The above corollary says that the inclusion $\AA \hookrightarrow \AA[1] * \AA$ \emph{categorifies} the group completion $\M(\AA) \to \K_0(\AA)$. Furthermore, we can realize all the localization of $\M(\AA)$ in this way, as follows.
\begin{rmk}\label{rem:mon-loc-categorify}
Consider the localization $\M(\catA)_S$ of $\M(\catA)$ 
with respect to any subset $S \subseteq \M(\catA)$.
We have a monoid isomorphism $\M(\catA)_S \iso \M(\catA)_{\gen{S}_{\face}}$
by Lemma \ref{lem:face loc 1}.
Then $\catF :=\catD_{\gen{S}_{\face}}$ is a Serre subcategory of $\catA$ 
and $\M_{\catF}=\gen{S}_{\face}$ by Proposition \ref{prop:serre-face-bij}.
In particular, it is a torsionfree class of $\catA$.
By Theorem \ref{thm:mon-of-inter}, we have a monoid isomorphism
\[
\M(\catA)_S \iso \M(\catA)_{\gen{S}_{\face}}=\M(\catA)_{\M_{\catF}}\iso \M(\catF[1]*\catA).
\]
Thus we can realize all localizations of $\M(\catA)$ 
as the Grothendieck monoids of intermediate subcategories of $\D^\b(\AA)$.
Therefore, the natural inclusion $\catA \inj \catC$ to an intermediate subcategory
yields a categorification of a localization of the monoid $\M(\catA)$.
\end{rmk}

Next, we describe the Grothendieck monoid of an intermediate subcategory of $\D^\b(\AA)$ when $\AA$ is an abelian length category with finitely many simples. 
Recall that an \emph{abelian length category} is a skeletally small abelian category such that every object has a finite length. For an abelian length category $\AA$, we denote by $\simp \AA$ the set of isomorphism classes of simple objects in $\AA$.
The following observation on the localization of a monoid can be easily checked.
\begin{lemma}\label{lem:loc-monoid-of-sum}
  Let $M$ and $N$ be monoids. Then the localization of $M \oplus N$ with respect to $M$ is isomorphic to $(\gp M) \oplus N$.
\end{lemma}

\begin{corollary}\label{cor:mon-of-inter-of-length}
  Let $\AA$ be an abelian length category and $\FF$ a torsionfree class in $\AA$. Define $\simp_\FF \AA$ as follows:
  \[
    \simp_\FF \AA := \{ [S] \in \simp\AA \mid \text{$S$ appears as a composition factor of a some object in $\FF$} \}.
  \]
  Then we have an isomorphism of monoids
  \[
    \M(\FF[1] * \AA) \iso \Z^{\oplus \simp_\FF \AA} \oplus 
    \N^{\oplus (\simp\AA \setminus \simp_\FF\AA)}.
  \]
\end{corollary}
\begin{proof}
  Since $\AA$ is an abelian length category and the Jordan-H\"older theorem holds in $\AA$, we have that $\M(\AA)$ is a free commutative monoid with the basis $\{ [S] \in \M(\AA) \mid [S] \in \simp\AA\}$, see \cite[Corollary 4.10]{JHP}. Thus we have an isomorphism $\M(\AA) \iso \N^{\oplus \simp\AA}$, which sends $[A] \in \M(\AA)$ to $\sum n_i [S_i]$, where $n_i$ is the multiplicity of $S_i$ in the composition series of $M$. In the rest of this proof, we identify $\M(\AA)$ with $\N^{\oplus \simp\AA}$.

  Theorem \ref{thm:mon-of-inter} implies that $\M(\FF[1] * \AA)$ is the localization of $\M(\AA)$ with respect to $\M_\FF$. By Lemma \ref{lem:face loc 1}, this localization coincides with that with respect to the smallest face $\la \M_\FF \ra_\face$ of $\M(\AA)$ containing $\M_\FF$. It is easily checked that the following holds:
  \[
    \la \M_\FF \ra_{\face} = \N^{\oplus \simp_\FF \AA} \subseteq \N^{\oplus \simp_\FF \AA} \oplus \N^{\oplus (\simp\AA \setminus \simp_\FF\AA)} = \M(\AA).
  \]
  Then Lemma \ref{lem:loc-monoid-of-sum} immediately deduces the assertion.
\end{proof}

\begin{example}\label{ex:toy}
  Consider Example \ref{ex:intro}.
  Since $\AA$ is length, $\M(\AA) = \N [S_1] \oplus \N[S_2] \oplus \N[S_3]$, where $S_i = \sst{i}$ is the simple module corresponding to each vertex $i$.
  Then $\simp_\FF\AA = \{ S_1, S_2\}$, so Corollary \ref{cor:mon-of-inter-of-length} implies $\M(\CC) = \Z [S_1] \oplus \Z[S_2] \oplus \N[S_3]$.
  For example, $[S_2]$ is invertible in $\M(\CC)$ although $[S_2[1]] \not \in \CC$, which can be checked alternatively as follows. We have $[S_2] = [\sst{2\\1}] + [\sst{1}[1]]$ in $\M(\CC)$ by a conflation
  \[
    \begin{tikzcd}
      \sst{2\\1} \rar & \sst{2} \rar & \sst{1}[1] \rar[dashed] & \ .
    \end{tikzcd}
  \]
  On the other hand, $[\sst{2\\1}]$ and $[\sst{1}[1]]$ have inverses $[\sst{2\\1}[1]]$ and $[\sst{1}]$ in $\M(\CC)$ respectively. Thus $[S_2]$ has an inverse $[\sst{2\\1}[1]] + [\sst{1}]$.
\end{example}

\subsection{Serre subcategories of an intermediate subcategory}
In this subsection, we classify Serre subcategories of an intermediate subcategory and study its localization.

First, we classify Serre subcategories of an intermediate subcategory via Serre subcategories of the original abelian category.
\begin{proposition}\label{prop:inter-serre-bij}
  Let $\AA$ be a skeletally small abelian category and $\FF$ a torsionfree class in $\AA$.
  Then there is a bijection between the following two sets.
  \begin{enumerate}
    \item $\Serre(\FF[1] * \AA)$.
    \item $\{\SS \in \Serre \AA \mid \FF \subseteq \SS \}$.
  \end{enumerate}
  The maps are given as follows: for $\DD$ in {\upshape (1)}, we consider $\AA \cap \DD$ in {\upshape (2)}, and for $\SS$ in {\upshape (2)} we consider $\FF[1] * \SS$ in {\upshape (1)}.
\end{proposition}
\begin{proof}
  Let $\DD$ be a Serre subcategory of $\FF[1] * \AA$. Then clearly $\AA \cap \DD$ is a Serre subcategory of $\AA$ since $\AA$ is an extension-closed subcategory of $\FF[1] * \AA$. Moreover, for any $F \in \FF$, we have a conflation
  \begin{equation}\label{eq:f-0-f1}
    \begin{tikzcd}
      F \rar & 0 \rar & F[1] \rar[dashed] & \ 
    \end{tikzcd}
  \end{equation}
  in $\FF[1] * \AA$, which implies that $F \in \DD$ and $F[1] \in \DD$ since $\DD$ is a Serre subcategory of $\FF[1] * \AA$ and $0 \in \DD$. Thus $F \in \AA \cap \DD$, and hence $\FF \subseteq \AA \cap \DD$.

  Conversely, let $\SS$ be a Serre subcategory of $\AA$ with $\FF \subseteq \SS$, and put $\DD := \FF[1] * \SS$. By Lemma \ref{lem:intermediate} (1), we can describe $\DD$ as follows:
  \[
    \FF[1] * \SS = \{
      X \in \FF[1] * \AA \mid H^0(X) \in \SS \}.
  \]
  We claim that $\DD$ is a Serre subcategory of $\FF[1] * \AA$. Suppose that there is a conflation
  \[
    \begin{tikzcd}
      X \rar["f"] & Y \rar["g"] & Z \rar[dashed] & \ 
    \end{tikzcd}
  \]
  in $\FF[1] * \AA$. Then we obtain the following long exact sequence in $\AA$:
  \[
    \begin{tikzcd}[row sep = 0]
      H^{-1}(Z) \rar["\delta"] & H^0(X) \rar["H^0(f)"] & H^0(Y) \rar["H^0(g)"] & H^0(Z) \rar & H^1(X) = 0
    \end{tikzcd}
  \]
  Thus we obtain the following short exact sequence:
  \[
    \begin{tikzcd}
      0 \rar & \coker \delta \rar & H^0(Y) \rar["H^0(g)"] & H^0(Z) \rar & 0.
    \end{tikzcd}
  \]
  If $X$ and $Z$ belong to $\FF[1] * \SS$, then we have $H^0(X),H^0(Z) \in \SS$ by $X, Z \in \FF[1] * \AA$, and $\coker \delta \in \SS$ since $\SS$ is closed under quotients in $\AA$. Thus $H^0(Y) \in \SS$ holds since $\SS$ is closed under extensions.
  Similarly, if $Y$ belongs to $\FF[1] * \SS$, then $H^0(Y) \in \SS$, and hence $\coker \delta, H^0(Z) \in \SS$ since $\SS$ is closed under subobjects and quotients. Moreover, we have another short exact sequence:
  \[
    \begin{tikzcd}
      0 \rar & \im \delta \rar & H^0(X) \rar & \coker \delta \rar & 0.
    \end{tikzcd}
  \]
  Since $\im \delta$ is a quotient of $H^{-1}(Z) \in \FF \subseteq \SS$, we have $\im \delta \in \SS$. Thus $H^0(X) \in \SS$ holds since $\SS$ is closed under extensions. Therefore, we obtain $X, Z \in \FF[1] * \SS$, so $\FF[1] * \SS$ is a Serre subcategory of $\FF[1] * \AA$.

  Finally, we check that these maps are inverse to each other. Let $\DD$ be a Serre subcategory of $\FF[1] * \AA$, and we will see $\DD = \FF[1] * (\AA \cap \DD)$. By the conflation \eqref{eq:f-0-f1}, we have $\FF[1] \subseteq \DD$. Thus we obtain $\FF[1] * (\AA\cap \DD) \subseteq \DD$ since $\DD$ is closed under extensions. Conversely, let $X \in \DD$. Then since $X \in \DD \subseteq \FF[1] * \AA$, we have the conflation
  \[
    \begin{tikzcd}
      F[1] \rar & X \rar & A \rar[dashed] & \ 
    \end{tikzcd}
  \]
  in $\FF[1] * \AA$ with $F \in \FF$ and $A \in \AA$. Since $\DD$ is a Serre subcategory of $\FF[1] * \AA$, we have in addition $A \in \DD$, and hence $A \in \AA \cap \DD$. Thus $X \in \FF[1] * (\AA \cap \DD)$.
  Next, let $\SS$ be a Serre subcategory of $\AA$ containing $\FF$. Clearly we have $\SS \subseteq \FF[1] * \SS$, so $\SS \subseteq \AA \cap (\FF[1] * \SS)$. Conversely, let $A \in \AA \cap (\FF[1] * \SS)$. Then we have an isomorphism $A \iso H^0(A)$, and $H^0(A) \in \SS$ holds by Lemma \ref{lem:intermediate} (1). Therefore we obtain $A \in \SS$.
\end{proof}

\begin{remark}
  The bijection in Proposition \ref{prop:inter-serre-bij} can also be constructed purely combinatorially as follows: Proposition \ref{prop:serre-face-bij}, Theorem \ref{thm:mon-of-inter}, and Proposition \ref{prp:face loc} give the following diagram consisting of bijections.
  \[
    \begin{tikzcd}
      \Serre (\FF[1] * \AA) \dar["\M_{(-)}"'] & \{\SS \in \Serre(\AA) \mid \FF \subseteq \SS \}
      \dar["\M_{(-)}"]
      \\
      \Face \M(\FF[1] * \AA) = \Face \M(\AA)_{\M_\FF} 
      \rar[leftrightarrow, "\sim"] & \{ F \in \Face \M(\AA) \mid \M_\FF \subseteq F\}
    \end{tikzcd}
  \]
  Therefore, we can think of Proposition \ref{prop:inter-serre-bij} as a \emph{categorification} of Proposition \ref{prp:face loc} about faces of the localization of a monoid.
\end{remark}

Next, we discuss the localization of an intermediate subcategory by its Serre subcategory. It is always possible and yields an abelian category:
\begin{proposition}\label{prop:inter-loc}
  Let $\AA$ be a skeletally small abelian category, $\FF$ a torsionfree class in $\AA$, and $\DD$ a Serre subcategory of $\FF[1] * \AA$. Then $\DD$ satisfies conditions {\upshape (EL1)--(EL4)} in Fact \ref{fct:loc ex}. Therefore, the exact localization $(\FF[1] * \AA) \big/\DD$ exists and is an abelian category.
\end{proposition}
\begin{proof}
  (EL1) Since $\DD$ is a Serre subcategory of $\FF[1] * \AA$ and $\FF[1] * \AA$ is admissible as shown in (EL4) below, it is a percolating subcategory by Example \ref{ex:adm-serre-perco}.

  (EL2)
  Lemma \ref{lem:intermediate} (1) implies that $\FF[1] * \AA$ is closed under direct summands in $\D^\b(\AA)$. Therefore, $\FF[1] * \AA$ satisfies the condition (WIC) by Lemma \ref{lem:WIC}, and thus (EL2) is satisfied by \cite[Remark 4.31 (2)]{NOS}.

  (EL3)
  First, we will check the first condition in (EL3).
  Let $f \colon X \to Y$ be an inflation in $\FF[1] * \AA$ and $\varphi \colon W\to X$ in $\FF[1] * \AA$ a morphism satisfying $f \varphi = 0$.
  We have the following triangle in $\D^\b(\AA)$ with $Z \in \FF[1] * \AA$:
  \[
    \begin{tikzcd}
      & W \dar["\varphi"] \ar[rd, "0"] \ar[ld, dashed, "\ov{\varphi}"'] & \\
      Z[-1] \rar["h"'] & X \rar["f"'] & Y \rar & Z.
    \end{tikzcd}
  \]
  Since $f\varphi = 0$ and $h$ is a weak kernel of $f$ in $\D^\b(\AA)$, it follows that $\varphi$ factors through $h$, so we obtain the above $\ov{\varphi} \colon W \to Z[-1]$ satisfying $\varphi = h \ov{\varphi}$.
  Then, since $Z \in \FF[1] * \AA$, we have the following triangle in $\D^\b(\AA)$ with $F \in \FF$ and $A \in \AA$.
  \[
    \begin{tikzcd}
      & W \dar["\ov{\varphi}"] \ar[rd, "0"] \ar[ld, dashed, "\psi"'] & \\
      F \rar["a"'] & Z[-1] \rar["b"'] & A[-1] \rar & F[1].
    \end{tikzcd}
  \]
  Here $b \ov{\varphi} = 0$ holds since both $\D^\b(\AA)(\FF[1], \AA[-1])$ and $\D^\b(\AA)(\AA, \AA[-1])$ vanish. Therefore, we obtain the above $\psi$ satisfying $\ov{\varphi} = a \psi$. Thus $\varphi = h \ov{\varphi} = h a \psi$ holds, so $\varphi$ factors through the object $F \in \FF$. On the other hand, by the proof of Proposition \ref{prop:inter-serre-bij}, we have $\FF \subseteq \DD$.
  Hence $\varphi$ factors through an object in $\DD$.

  Next, we will prove the second condition. Let $g \colon Y \to Z$ be a deflation in $\FF[1] * \AA$ and $\varphi \colon Z \to W$ in $\FF[1] * \AA$ a morphism satisfying $\varphi g = 0$.
  We have the following triangle in $\D^\b(\AA)$ with $X \in \FF[1] * \AA$:
  \[
    \begin{tikzcd}
      X \rar & Y \rar["g"] \ar[rd, "0"'] & Z \rar["h"] \dar["\varphi"] & X[1] \ar[dl, dashed, "\ov{\varphi}"] \\
      & & W
    \end{tikzcd}
  \]
  Since $\varphi g = 0$ and $h$ is a weak cokernel of $g$ in $\D^\b(\AA)$, it follows that $\varphi$ factors through $h$, so we obtain the above $\ov{\varphi} \colon X[1] \to W$ satisfying $\varphi = \ov{\varphi} h$.
  Then, since $W \in \FF[1] * \AA$, we have the following triangle in $\D^\b(\AA)$ with $F \in \FF$ and $A \in \AA$.
  \[
    \begin{tikzcd}
      & X[1] \dar["\ov{\varphi}"] \ar[rd, "0"] \ar[ld, dashed, "\psi"'] & \\
      F[1] \rar["a"'] & W \rar["b"'] & A \rar & F[2].
    \end{tikzcd}
  \]
  Here $b \ov{\varphi} = 0$ holds since $X[1] \in \FF[2] * \AA[1]$ and both $\D^\b(\AA)(\FF[2], \AA)$ and $\D^\b(\AA)(\AA[1], \AA)$ vanish. Therefore, we obtain the above $\psi$ satisfying $\ov{\varphi} = a \psi$. Thus $\varphi = \ov{\varphi} h = a \psi h$ holds, so $\varphi$ factors through the object $F[1] \in \FF[1]$. On the other hand, by the proof of Proposition \ref{prop:inter-serre-bij}, we have $\FF[1] \subseteq \DD$.
  Hence $\varphi$ factors through an object in $\DD$, and thus (EL3) is satisfied.

  (EL4)
  The proof of this part is essentially the same as the proof of the fact that the heart of a $t$-structure is an abelian category. We denote by $(-)^{\leq 0}, (-)^{\geq 1} \colon \D^\b(\AA) \to \AA$ the truncation functors with respect to the standard $t$-structure of $\D^\b(\AA)$.

  Let $f \colon X \to Y$ be any morphism in $\FF[1] * \AA$. By taking the mapping cocone $K$ of $f$ and the truncation of $K$ and using the octahedral axiom, we obtain the following commutative diagram consisting of triangles in $\D^\b(\AA)$:
  \begin{equation}\label{eq:octa}
    \begin{tikzcd}
      K^{\leq 0} \dar \rar[equal] & K^{\leq 0} \dar \\
      K \rar \dar & X \rar["f"] \dar["p"'] & Y \rar \dar[equal] & K[1] \dar \\
      K^{\geq 1} \rar \dar & W \rar["i"'] \dar & Y \rar & K^{\geq 1} [1] \\
      K^{\leq 0} [1] \rar[equal] & K^{\leq 0} [1]
    \end{tikzcd}
  \end{equation}
  We claim that $f = i p$ is a deflation-inflation factorization. It suffices to show the following three assertions.
  \begin{enumerate}
    \item $K^{\leq 0} \in \FF[1] * \AA$
    \item $K^{\geq 1}[1] \in \AA \subseteq \FF[1] * \AA$.
    \item $W \in \FF[1] * \AA$.
  \end{enumerate}
  
  (1) By the second row of \eqref{eq:octa}, we have the following long exact sequence in $\AA$:
  \[
    \begin{tikzcd}[row sep = 0]
      & 0 = H^{-2}(Y) \rar & H^{-1}(K) \rar & H^{-1}(X) \\
      \rar & H^{-1}(Y) \rar & H^0(K) \rar & H^0(X) \\
      \rar & H^0(Y) \rar & H^1(K) \rar & H^{1}(X) = 0.
    \end{tikzcd}
  \]
  It easily follows that $H^i(K) = 0$ for $i \notin \{-1,0,1\}$. Moreover, since $H^{-1}(X) \in \FF$, we have $H^{-1}(K) \in \FF$ because $\FF$ is closed under subobjects.
  Since the $i$-th cohomology of $K^{\leq 0}$ is zero for $i > 0$ and the same as $K$ for $i \leq 0$, we obtain that $K^{\leq 0}$ belongs to $\FF[1] * \AA$ by Lemma \ref{lem:intermediate} (1).

  (2)
  Because the $i$-th cohomology of $K$ vanishes for $i \geq 1$ except for $i=1$ by the proof of (1), we have $K^{\geq 1} \in \AA[-1]$. Thus $K^{\geq 1}[1] \in \AA$ holds.

  (3)
  Since $H^i(X) = 0$ and $H^i(K^{\leq 0}[1]) = H^{i+1}(K^{\leq 0}) = 0$ for $i \geq 1$, we obtain that $H^i(W) = 0$ for $i \geq 1$ by the second column of \eqref{eq:octa}. On the other hand, the third row of \eqref{eq:octa} shows the following exact sequence in $\AA$:
  \[
    \begin{tikzcd}
      0 = H^{-1}(K^{\geq 1}) \rar & H^{-1}(W) \rar & H^{-1}(Y) \rar
      & H^0(K^{\geq 1}) = 0
    \end{tikzcd}
  \]
  Thus $H^i (W) = 0$ holds for $i \leq -2$ and $H^{-1} (W) \iso H^{-1}(Y) \in \FF$ holds. Therefore, we obtain $W \in \FF[1] * \AA$ by Lemma \ref{lem:intermediate} (1).
\end{proof}

Actually, we can describe the localization of an intermediate subcategory as a usual Serre quotient of an abelian category as follows.
\begin{theorem}\label{thm:inter-serre-loc}
  Let $\AA$ be a skeletally small abelian category, $\FF$ a torsionfree class in $\FF$, and $\SS$ a Serre subcategory of $\AA$ with $\FF \subseteq \SS$. We have the following commutative diagram consisting of exact functors of extriangulated categories, where $Q_1$ and $Q_2$ are localization functors:
  \[
    \begin{tikzcd}
      \AA \rar["\iota"] \dar["Q_1"'] & \FF[1] * \AA \dar["Q_2"] \\
      \AA / \SS \rar["\Phi"] & (\FF[1] * \AA)\big / (\FF[1] * \SS)
    \end{tikzcd}
  \]
  Then $\Phi$ gives an equivalence of extriangulated categories (or equivalently, abelian categories).
\end{theorem}
\begin{proof}
  First, recall that an exact functor from an extriangulated category to an exact category (e.g. an abelian category) is precisely an additive functor preserving conflations by Lemma \ref{lem:exact-func-to-exact-cat}.
  It is easily checked from the universality of the localization that an exact functor $\Phi$ exists (see Proposition \ref{prop:ex-loc-univ2}). Furthermore, since both $\AA / \SS$ and $(\FF[1] * \AA)\big / (\FF[1] * \SS)$ are abelian categories by Proposition \ref{prop:inter-loc}, we only have to check that $\Phi$ is just an equivalence of an additive category.

  Actually, we will construct a quasi-inverse $\Psi \colon (\FF[1] * \AA)\big / (\FF[1] * \SS) \to \AA/\SS$ as follows. Consider the following diagram.
  \begin{equation}\label{eq:inter-eq}
    \begin{tikzcd}
      \FF[1] * \AA \dar["Q_2"'] \rar["H^0"] & \AA \dar["Q_1"]  \\
      (\FF[1] * \AA)\big / (\FF[1] * \SS) \rar[dashed, "\Psi"'] & \AA/\SS
    \end{tikzcd}
  \end{equation}
  We claim that the composition $Q_1 H^0$ is an exact functor. To show this, take any conflation
  \[
    \begin{tikzcd}
      X \rar["f"] & Y \rar["g"] & Z \rar[dashed] & \ 
    \end{tikzcd}
  \]
  in $\FF[1] * \AA$. Then since it is a triangle in $\D^\b(\AA)$, we obtain the following long exact sequence in $\AA$:
  \[
    \begin{tikzcd}
      H^{-1}(Z) \rar & H^0(X) \rar & H^0(Y) \rar & H^0(Z) \rar & H^1(X) = 0
    \end{tikzcd}
  \]
  Then Lemma \ref{lem:intermediate} (1) implies $H^{-1}(Z) \in \FF$. Now since $\FF$ is contained in $\SS$, we obtain the following short exact sequence in $\AA/\SS$ by applying an exact functor $Q_1$ to the above exact sequence:
  \[
    \begin{tikzcd}
      0 \rar & Q_1 H^0(X) \rar & Q_1 H^0(Y) \rar & Q_1 H^0(Z) \rar & 0.
    \end{tikzcd}
  \]
  Therefore, Lemma \ref{lem:exact-func-to-exact-cat} shows that $Q_1 H^0$ is an exact functor between extriangulated categories. Hence an exact functor $\Psi$ in \eqref{eq:inter-eq} is induced.
  Now we obtain the following strict commutative diagram consisting of additive functors between additive categories.
  \[
    \begin{tikzcd}
      \AA \rar["\iota"] \dar["Q_1"'] & \FF[1] * \AA \dar["Q_2"] \rar["H^0"]
      & \AA \dar["Q_1"] \\
      \AA / \SS \rar["\Phi"'] & (\FF[1] * \AA)\big / (\FF[1] * \SS) \rar["\Psi"'] & \AA/\SS
    \end{tikzcd}
  \]
  Since $H^0 \iota$ is naturally isomorphic to $\id_\AA$, we have $Q_1 H^0 \iota \simeq Q_1$. Thus the universal property of $Q_1$ as an exact 2-localization (see Remark \ref{rmk:2-loc}) implies that $\Psi \Phi$ is naturally isomorphic to the identity functor.

  Similarly, we have the following strict commutative diagram.
  \[
    \begin{tikzcd}
       \FF[1] * \AA \dar["Q_2"'] \rar["H^0"]
      &  \AA \dar["Q_1"] \rar["\iota"] & \FF[1] * \AA \dar["Q_2"]   \\
      (\FF[1] * \AA)\big / (\FF[1] * \SS) \rar["\Psi"'] & \AA/\SS \rar["\Phi"'] & (\FF[1] * \AA)\big / (\FF[1] * \SS)
    \end{tikzcd}
  \]
  We claim that $Q_2 \iota H^0$ is naturally isomorphic to $Q_2$.
  In fact, for any morphism $f \colon X \to Y$ in $\FF[1] * \AA$, we obtain the following morphism of conflations in $\FF[1] * \AA$:
  \[
    \begin{tikzcd}
      H^{-1}(X)[1] \rar \dar["{H^{-1}(f)[1]}"'] & X \rar\dar["f"] & \iota H^0(X) \dar["\iota H^0(f)"] \rar[dashed] & \ \\
      H^{-1}(Y)[1] \rar & Y \rar & \iota H^0(Y) \rar[dashed] & \ 
    \end{tikzcd}
  \]
  Since $Q_2$ is an exact functor which sends every object in $\FF[1]$ to $0$, we obtain the following exact commutative diagram in an abelian category $ (\FF[1] * \AA)\big / (\FF[1] * \SS)$:
  \[
    \begin{tikzcd}
      0 = Q_2 H^{-1}(X)[1] \rar \dar["{Q_2 H^{-1}(f)[1]}"'] & Q_2 (X) \rar\dar["Q_2 (f)"] & Q_2 \iota H^0(X) \dar["Q_2 \iota H^0(f)"] \rar & 0 \\
      0 = Q_2 H^{-1}(Y)[1] \rar & Q_2 (Y) \rar & Q_2 \iota H^0(Y) \rar & 0 
    \end{tikzcd}
  \]
  Therefore, $Q_2 \iota H^0$ is naturally isomorphic to $Q_2$. Thus the universal property of an exact 2-localization $Q_1$ (Remark \ref{rmk:2-loc}) implies that $\Phi \Psi$ is naturally isomorphic to the identity functor. Therefore, $\Phi$ and $\Psi$ are mutually quasi-inverse.
\end{proof}

\section{Questions}\label{sec:6}
In this section, we give some questions on the Grothendieck monoid of an extriangulated category. \emph{Throughout this section, $\CC$ denotes a skeletally small extriangulated category.}

\subsection{Invertible elements in the Grothendieck monoid}
We begin with considering $\M(\CC)^\times$, which consists of invertible elements in the Grothendieck monoid (Remark \ref{rmk:Face}). We can interpret this as the smallest Serre subcategory:
\begin{proposition}\label{prop:smallest-serre}
  In the bijection of $\Serre \CC$ and $\Face \M(\CC)$ in Proposition \ref{prop:serre-face-bij}, the face $\M(\CC)^\times$ corresponds to the unique smallest Serre subcategory of $\CC$.
\end{proposition}
\begin{proof}
  This follows from Proposition \ref{prop:serre-face-bij} and the fact that $\M(\CC)^\times$ is the unique smallest face of $\M(\CC)$ by Remark \ref{rmk:Face}.
\end{proof}
We are interested in the extreme case, that is, $\M(\CC)^\times = 0$ and $\M(\CC)^\times = \M(\CC)$. The former one holds when $\CC$ is an exact category.
Actually, we have the following implications.
\begin{proposition}
  Let $\CC$ be a skeletally small extriangulated category. Consider the following conditions.
  \begin{enumerate}
    \item $\CC$ is an exact category (with the usual extriangulated structure).
    \item If $X \to 0 \to Y \dto$ is a conflation in $\CC$, then $X \iso 0$ holds in $\CC$ (or equivalently, $Y \iso 0$ holds).
    \item $0$ is a Serre subcategory of $\CC$, where $0$ denotes the subcategory consisting of zero objects in $\CC$.
    \item If $[A] = 0$ in $\M(\CC)$, then $A \iso 0$ holds in $\CC$.
    \item $\M(\CC)^\times = 0$ holds, or equivalently, $\M(\CC)$ is a reduced monoid (see Definition \ref{def:mon-red-can}).
  \end{enumerate}
  Then {\upshape (1) $\Rightarrow$ (2) $\Leftrightarrow$ (3) $\Leftrightarrow$ (4) $\Rightarrow$ (5)} holds.
\end{proposition}
\begin{proof}
  The conditions in (2) are easily seen to be equivalent by considering the long exact sequence associated to the conflation.

  (1) $\Rightarrow$ (2):
  This is clear since an inflation $X \to 0$ in an exact category must be a monomorphism.

  (2) $\Rightarrow$ (3):
  This is immediate from the definition of Serre subcategories.

  (3) $\Rightarrow$ (4):
  Suppose that $\DD := 0$ is a Serre subcategory. Then $\M_\DD = 0$ holds, and Proposition \ref{prop:serre-face-bij} implies that $0 = \DD = \DD_{\M_\DD} = \DD_0 = \{ X \in \CC \mid [X] = 0 \}$ holds. This clearly shows (4).

  (4) $\Rightarrow$ (2):
  Suppose that we have a conflation $X \to 0 \to Y \dto$. Then we have $[X \oplus Y] = [X] + [Y] = [0] = 0$ in $\M(\CC)$. Thus (4) implies $X \oplus Y \iso 0$, so $X \iso Y \iso 0$ holds.

  (4) $\Rightarrow$ (5):
  Suppose that $x + y = 0$ holds in $\M(\CC)$. There are $X$ and $Y$ in $\CC$ satisfying $[X] = x$ and $[Y] = y$, and $[X \oplus Y] = [X] + [Y] = 0$ holds. Thus (4) implies $X \oplus Y \iso 0$, which shows $X \iso Y \iso 0$. Therefore, $x = y = 0$ holds.
\end{proof}
In the first version of this paper,
we posed the question whether all of the above conditions are equivalent.
However, there is a counterexample for (5) $\imply$ (1) and (3) $\imply$ (1)
(see \cite[Example 5.2, Remark 5.3]{BHST}).
Thus we modify this question to the following form.
\begin{question}\label{q:exact-reduced}
  Let $\CC$ be a skeletally small extriangulated category. Are the following conditions equivalent?
  \begin{enumerate}
    \item $0$ is a Serre subcategory of $\CC$.
    \item $\M(\CC)$ is a reduced monoid.
  \end{enumerate}
\end{question}

Next, we consider the condition $\M(\CC)^\times = \M(\CC)$. Clearly this is equivalent to that $\M(\CC)$ is a group. If $\CC$ is triangulated, then this is the case, but there are examples where $\CC$ is not triangulated and $\M(\CC)$ is a group (see Example \ref{ex:gro-mon} and Corollary \ref{cor:a1-a-k0}). Thus we have the following natural question.
\begin{question}
  When does $\M(\CC)$ become a group?
\end{question}
We remark that this is equivalent to $\Serre\CC = \{\CC\}$ by Proposition \ref{prop:smallest-serre}.

Finally, we discuss the ``decomposition'' of a monoid into a group and a reduced monoid. We have the following monoid-theoretic decomposition of an arbitrary monoid.
\begin{proposition}
  Let $M$ be a monoid. Then the following holds.
  \begin{enumerate}
    \item $M/M^\times$ is a reduced monoid.
    \item Let $L$ be a submonoid of $M$. Suppose that $L$ is a group and that $M/L$ is a reduced monoid. Then $L = M^\times$ holds.
  \end{enumerate}
\end{proposition}
\begin{proof}
  (1) This follows from the fact that $M^\times$ is a face of $M$ and Lemma \ref{lem:face loc 2} (1).

  (2)
  Since $L$ is a group, clearly $L \subseteq M^\times$ holds. Conversely, let $x \in M^\times$, so there is an element $y \in M$ satisfying $x + y = 0$. Then for the natural quotient map $\pi \colon M \defl M/L$, we have $\pi(x) + \pi(y) = 0$. This implies $\pi(x) = 0$ since $M/L$ is reduced. Therefore, there is some $l_1, l_2 \in L$ satisfying $x + l_1 = l_2$. Since $L$ is a group, $l_1$ is invertible, so we  obtain $x = l_2 - l_1 \in L$. 
\end{proof}
Roughly speaking, this claims the existence and the uniqueness of an ``short exact sequence'' of monoids
\[
  \begin{tikzcd}
    0 \rar & M^\times \rar & M \rar & M/M^\times \rar & 0
  \end{tikzcd}
\]
such that $M^\times$ is a group and $M/M^\times$ is a reduced monoid.
We are interested whether the corresponding exact localization exists for $M := \M(\CC)$, which we shall explain in detail.
Proposition \ref{prop:smallest-serre} shows that $\M(\CC)^\times$ corresponds to the smallest Serre subcategory of $\CC$, which we denote by $\SS_0$.
If $\SS_0$ satisfies some good properties like Corollary \ref{cor:mon loc ET}, then we  get the exact localization $\CC/\SS_0$ such that $\M(\CC/\SS_0) \iso \M(\CC)/\M_{\SS_0} = \M(\CC)/\M(\CC)^\times$. Moreover, if Question \ref{q:exact-reduced} is true, then $\CC/\SS_0$ should be an exact category.
This can be considered as a decomposition of an extriangulated category $\CC$ into the ``exact (reduced) part" $\CC/\SS_0$ and the ``triangulated-like (group) part" $\SS_0$.
Thus our question can be summarized as follows.
\begin{question}\label{q:cat-decomp}
  Let $\SS_0$ be the smallest Serre subcategory of $\CC$.
  \begin{enumerate}
    \item When does there exist the exact localization $\CC \to \CC/\SS_0$?
    \item If {\upshape (1)} holds, does it commute with the monoid quotient, that is, do we have $\M(\CC/\SS_0) \iso \M(\CC)/\M_{\SS_0} = \M(\CC)/\M(\CC)^\times$?
  \end{enumerate}
\end{question}
\begin{example}
  Question \ref{q:cat-decomp} has a positive answer for the following classes of extriangulated categories $\CC$.
  \begin{itemize}
    \item $\CC$ is an exact category. In this case, we have $\SS_0 = 0$, the exact localization is the identity functor, and $\M(\CC)$ is already reduced.
    \item $\CC$ is a triangulated category. In this case, we have $\SS_0 = \CC$, the exact localization is $\CC \to 0$, and $\M(\CC)$ is already a group.
    \item $\CC$ is an intermediate subcategory of $\D^\b(\AA)$ for an abelian category $\AA$. Let $\FF$ be a torsionfree class in $\AA$ with $\CC = \FF[1] * \AA$ and $\SS$ the smallest Serre subcategory of $\AA$ containing $\FF$. Then we have $\SS_0 = \FF[1] * \SS$ by Proposition \ref{prop:inter-serre-bij}. In this case, the exact localization exists and commutes with the monoid quotient by Proposition \ref{prop:inter-loc}. Note that $\CC/\SS_0$ is actually not only an exact category but an abelian category.
  \end{itemize}
\end{example}

We end this section with the following natural (but maybe difficult) inverse problem.
\begin{question}
  For a given commutative monoid $M$, does there exist an extriangulated category $\CC$ satisfying $\M(\CC) \iso M$?
\end{question}
Related to this, we remark that for any reduced monoid $M$, there is a split exact category $\CC$ satisfying $\M(\CC) \iso M$ by the argument similar to \cite[Section 2]{F}.

\subsection{Quasi-split extriangulated categories and c-closed subcategories}
In Section \ref{sec:3},
we have classified certain subcategories of $\catC$ via its Grothendieck monoid $\M(\catC)$.
In particular, we have introduced c-closed subcategories,
the largest class of subcategories which can be classified via $\M(\catC)$
(Proposition \ref{prp:c-closed bij}).
In this subsection,
we will discuss the conditions under which all subcategories are classified via $\M(\catC)$,
that is, all subcategories are c-closed.

Let us begin with a simple observation. Note that a subcategory is assumed to be closed under isomorphisms, but \emph{not} assumed to be additive here.
\begin{prp}\label{prp:c-closed quasi-split}
Let $\catC$ be a skeletally small extriangulated category.
Then the following conditions are equivalent.
\begin{enua}
\item
All subcategories of $\catC$ are c-closed.
\item
For any conflation $A\to B\to C\dto$ in $\catC$,
we have $B \iso A\oplus C$.
\end{enua}
\end{prp}
\begin{proof}
(1) $\imply$ (2):
Let $A\to B\to C\dto$ be a conflation in $\catC$.
Consider the subcategory $\catX$ consisting of objects isomorphic to $B$.
Since $\catX$ is c-closed, it contains $A\oplus C$.
This means $A\oplus C \iso B$.

(2) $\imply$ (1):
Let $\catX$ be a subcategory of $\catC$.
For any conflation $A\to B\to C\dto$,
we have $B\iso A\oplus C$ by the assumption (2).
Thus $B\in \catX$ if and only if $A\oplus C \in \catX$.
This proves that $\catX$ is c-closed.
\end{proof}
Let us call a conflation $A \to B \to C \dto$ \emph{quasi-split} if $B \iso A \oplus C$ holds.
There are subtle differences between split and quasi-split conflations.
\begin{ex}\label{ex:counter eq quasi-split}
There are two short exact sequences of abelian groups:
\[
\delta \colon 0\to \bbZ \xr{2} \bbZ \to \bbZ/2\bbZ \to0 
\quad \text{and} \quad
\epsilon \colon 0\to 0 \to (\bbZ/2\bbZ)^{\oplus \bbN} \xrightarrow{=} (\bbZ/2\bbZ)^{\oplus \bbN} \to0.
\]
Then the short exact sequence obtained by their direct sum
\[
  \begin{tikzcd}
    \delta \oplus \epsilon \colon 0\rar & \bbZ \rar & \bbZ \oplus (\bbZ/2\bbZ)^{\oplus \bbN} \rar & \bbZ/2\bbZ \oplus (\bbZ/2\bbZ)^{\oplus \bbN} \rar & 0
  \end{tikzcd}
\]
is quasi-split but not split.
Indeed, by the natural isomorphism $\Ext^1_{\bbZ}\left(\bbZ/2\bbZ \oplus (\bbZ/2\bbZ)^{\oplus \bbN}, \bbZ\right) \iso \Ext^1_{\bbZ}\left(\bbZ/2\bbZ, \bbZ\right) \oplus \Ext^1_{\bbZ}\left((\bbZ/2\bbZ)^{\oplus \bbN}, \bbZ\right)$,
the element $\delta \oplus \epsilon$ corresponds to $(\delta,0)$,
and it is nonzero since $\delta$ is not a split exact sequence.
\end{ex}
From this example,
a quasi-split conflation is not a split conflation in general.
However, there are the following examples where this holds.
\begin{ex}\label{ex:miyata}
In the following classes of extriangulated categories $\catC$,
every quasi-split conflation is a split conflation.
\begin{itemize}
\item 
$\catC$ is an exact $R$-category over a commutative ring $R$
such that $\catC(A,B)$ is an $R$-module of finite length for any $A,B \in\catC$.
Indeed, for any quasi-split conflation $0\to A\to B\xr{f} C\to 0$,
we have an exact sequence
\[
  \begin{tikzcd}
    0\rar & \catC(C,A) \rar & \catC(C,B) \rar["{\catC(C,f)}"] & \catC(C,C)
  \end{tikzcd}
\]
of $R$-modules.
Now we have $\catC(C,B) \iso \catC(C,A)\oplus \catC(C,C)$ by the assumption.
Considering the lengths of $R$-modules in the above exact sequence,
we conclude that $\catC(C,f)$ is surjective,
and this implies that the conflation $0\to A\to B\xr{f} C\to 0$ splits.
\item
$\catC$ is an extension-closed subcategory of $\catmod \Lambda$,
where $\Lambda$ is an algebra over a commutative noetherian ring $R$
with $\Lambda \in \mod R$.
This follows from the result in \cite[Theorem 1]{miyata}, which states that any quasi-split short exact sequence in $\mod \Lambda$ actually splits.
\end{itemize}
\end{ex}

The above examples lead us to the next question.
\begin{question}\label{q:quasi-split 1}
When does a quasi-split conflation become a split conflation?
More concretely, in the following classes of exact categories $\catC$,
does this question have a positive answer?
\begin{itemize}
\item 
$\catC$ is an exact $R$-category over a commutative noetherian ring $R$
such that $\catC(A,B) \in \mod R$ holds for any $A,B \in\catC$.
\item
$\catC$ is a noetherian exact category,
that is for any $X\in\catC$, the poset of admissible subobjects of $X$
satisfies the ascending chain condition 
(see \cite[Section 2]{JHP} for this poset).
\end{itemize}
\end{question}

We now return to considering when all subcategories are classified via $\M(\catC)$.
An extriangulated category $\catC$ is said to be \emph{quasi-split} if every conflation in $\catC$ is quasi-split.
Recall that this condition is equivalent to the condition under which 
all subcategories of $\catC$ can be classified via the Grothendieck monoid
by Proposition \ref{prp:c-closed quasi-split}. 
On the other hand, if every conflation of an extriangulated category $\catC$ is a split conflation,
then every inflation is a monomorphism, and every deflation is an epimorphism.
Thus $\catC$ becomes an exact category (see \cite[Corollary 3.18]{NP}).
Such an exact category $\catC$ is called a \emph{split exact category}.
It is clear that a split exact category is quasi-split.
Thus we have the following natural question.
\begin{question}\label{q:quasi-split 2}
When does a quasi-split extriangulated category $\catC$ become a split exact category?
\end{question}

\begin{ex}
Question \ref{q:quasi-split 2} has a positive answer for the following classes of extriangulated categories $\catC$.
\begin{itemize}
\item
$\catC$ is one of the exact categories in Example \ref{ex:miyata}. 
\item
More generally, $\catC$ is an exact category in which Question \ref{q:quasi-split 1} has a positive answer.
\item
$\catC$ is an extriangulated category with enough projectives.
See Proposition \ref{prp:quasi-split enough proj} below.
\end{itemize}
\end{ex}

We recall some terminology to prove Proposition \ref{prp:quasi-split enough proj}.
An object $P$ of $\catC$ is \emph{projective} if $\bbE(P,X)=0$ for any $X\in\catC$.
We denote by $\Proj \catC$ the category of projective objects of $\catC$.
Clearly $\Proj\catC$ is closed under extensions and direct summands.
An extriangulated category $\catC$ \emph{has enough projectives}
if for any $A\in\catC$, there exists a deflation $P\to A$ from a projective object $P$.
\begin{prp}\label{prp:quasi-split enough proj}
Suppose that an extriangulated category $\catC$ has enough projectives.
Then the following conditions are equivalent.
\begin{enua}
\item
Every subcategory of $\CC$ is c-closed.
\item
Every additive subcategory of $\CC$ is c-closed.
\item
Every extension-closed subcategory of $\CC$ is c-closed.
\item
$\catC$ is a split exact category.
\item
$\catC$ is a quasi-split extriangulated category.
\end{enua}
\end{prp}
\begin{proof}
The implications (1) $\imply$ (2) $\imply$ (3) and (4) $\imply$ (5) $\imply$ (1) are clear,
so we only prove (3) $\imply$ (4).
Let $A \to B\to C \dto$ be a conflation in $\catC$.
Since $\catC$ has enough projectives,
there is a conflation $X \to P \to C \dto$ in $\CC$ with $P\in \Proj \catC$.
The assumption (3) implies that $\Proj\catC$ is c-closed, and hence $C\oplus X \in \Proj\catC$.
Then $C$ is also projective since $\Proj \catC$ is closed under direct summands.
Therefore the conflation $A \to B\to C \dto$ splits.
This proves that $\catC$ is a split exact category.
\end{proof}

\appendix
\def\thesection{\Alph{section}}
\section{Basics on commutative monoids}\label{app:mon}
In this section, we collect some definitions and properties of commutative monoids.

A \emph{monoid} is a semigroup with a unit.
In this paper, every monoid is assumed to be \emph{commutative}.
Hence we use the additive notation, that is,
the binary operation is denoted by $+$, and the unit is denoted by $0$.
A \emph{homomorphism} of monoids is a map $f\colon M \to N$ 
satisfying $f(x+y)=f(x)+f(y)$ and $f(0_M)=0_N$.
We denote $\Mon$ by the category of (commutative) monoids and homomorphisms of them.
The category $\Mon$ has arbitrary small limits and colimits
(see \cite[Section I.1.1]{Og}). 
We can define
the product $\prod_{i\in I}M_i$ and 
the direct sum (= coproduct) $\bigoplus_{i\in I}M_i$ of monoids
similarly to abelian groups. 
In particular, finite products and finite direct sums coincide.

A basic example of monoids is
the set $\bbN$ of non-negative integers with the arithmetic addition.
A monoid $M$ is said to be \emph{free} 
if it is isomorphic to $\bbN^{\oplus I}$ for some index set $I$.
In this case, the cardinality of $I$ is called the \emph{rank} of $M$.
A \emph{basis} of a free monoid is defined by a similar way to abelian groups.

The notion of quotients of monoids slightly differs from that of abelian groups.
We introduce a class of equivalence relations $\sim$ on a monoid $M$
to guarantee that the quotient set $M/{\sim}$ becomes a monoid.
\begin{dfn}\label{def:mon-congr}
The equivalence relation $\sim$ on a monoid $M$ is called a \emph{congruence}
if $x\sim y$ implies $a+x \sim a+y$ for every $a, x, y\in M$.
\end{dfn}
We can check easily that the quotient set $M/{\sim}$ of a monoid $M$ 
by a congruence $\sim$ has the unique monoid structure such that
the quotient map $M\to M/{\sim}$ is a monoid homomorphism.
To each submonoid of a monoid, we can associate the congruence as follows.
\begin{definition}\label{def:quot-submon}
  Let $M$ be a monoid and $N$ its submonoid. Define a congruence on $M$ as follows:
  \[
  x\sim y :\equi \text{there exist $n,n'\in N$ such that $x+n=y+n'$.}
  \]
  Then the monoid $M/N:=M/{\sim}$ is called the \emph{quotient monoid} of $M$ by $N$.
  We write $x\equiv y \bmod N$ if $x \sim y$ holds.
  The equivalence class of $x\in M$ is denoted by $x \bmod N$. 
\end{definition}

It is easily seen that the quotient monoids have the following universal property.
\begin{prp}\label{prp:mon quot}
Let $N$ be a submonoid of a monoid $M$, and let $\pi \colon M\to M/N$ be the quotient homomorphism.
Then $\pi(N) = 0$ holds, and for any monoid homomorphism $f\colon M \to X$ such that $f(N)=0$,
there exists a unique monoid homomorphism $\ol{f}\colon M/N \to X$ satisfying $\ol{f}\pi=f$.
This means that the diagram
\[
  \begin{tikzcd}
    N \rar[shift left, "\iota"] \rar[shift right, "0"'] & M \rar["\pi"] & M/N
  \end{tikzcd}
\]
is a coequalizer diagram in $\Mon$,
where $\iota$ is the inclusion map.
\end{prp}

Next, we introduce a class of submonoids which corresponds to
Serre subcategories (see Section \ref{s:Serre face}).
\begin{dfn}\label{def:face}
Let $M$ be a monoid.
\begin{enua}
\item
A submonoid $F$ of $M$ is called a \emph{face}
if for all $x,y\in M$, we have that $x+y\in F$ if and only if both $x \in F$ and $y\in F$.
\item
$\Face M$ denotes the set of faces of $M$.
\end{enua}
\end{dfn}

\begin{rmk}\label{rmk:Face}
Let $M$ be a monoid.
An element $x\in M$ is called a \emph{unit} if there exists $y\in M$ such that $x+y=0$.
The set $M^{\times}$ of units is the smallest face,
and $M$ itself is the largest face.
In particular, $\Face M$ is a singleton if and only if $M$ is a group.
\end{rmk}

\begin{ex}
Let $M$ be a free monoid of rank $2$ with a basis $e_1$ and $e_2$.
\begin{enua}
\item
$\bbN(e_1+e_2)$ is a submonoid of $M$ but not a face.
\item
$\Face M = \{ M, \bbN e_1, \bbN e_2, 0 \}$ holds.
\end{enua}
\end{ex}

Let us give an explicit description of the face generated by a subset,
which is useful to study faces.
\begin{fct}[{cf. \cite[Proposition I.1.4.2]{Og}}]
Let $S$ be a subset of a monoid $M$.
\begin{enua}
\item
The \emph{submonoid $\la S \ra_{\bbN}$ of $M$ generated by $S$} is the smallest submonoid of $M$ containing $S$. We can describe it as follows:
\[
\la S \ra_{\bbN}:=\left\{\sum_{i=1}^m n_ix_i \middle|\; m,n_i \in \bbN, x_i\in S  \right\}.
\]
\item
The \emph{face $\la S \ra_{\face}$ of $M$ generated by $S$} is the smallest face of $M$ containing $S$. We can describe it as follows:
\[
\la S \ra_{\face}:=\left\{x\in M \mid
\text{there exists $y\in M$ such that $x+y\in \la S \ra_{\bbN}$} \right\}.
\]
\end{enua}
\end{fct}

Unlike the case of abelian groups,
submonoids of $M/N$ do not correspond to 
those of $M$ containing $N$:
\begin{ex}
Let $M:=\bbN^{\oplus 2}$ and $N:=\bbN(1,0)+\bbN(1,1) \subseteq M$.
Then we have $M/N = 0$ but $M$ and $N$ are distinct submonoids of $M$ containing $N$.
\end{ex}

However, we have a bijection for faces.
\begin{prp}\label{prp:face quot}
Let $N$ be a submonoid of a monoid $M$,
and let $\pi \colon M \to M/N$ be the quotient homomorphism.
\begin{enua}
\item
If $X$ is a face of $M$ containing $N$,
then $X/N:=\pi(X)$ is also a face of $M/N$.
\item
If $X'$ is a face of $M/N$,
then $\pi^{-1}(X')$ is also a face of $M$ containing $N$.
\item
The assignments in {\upshape (1)} and {\upshape (2)} give inclusion-preserving bijections
between $\{ X \in \Face M \mid X \supseteq N\} $ and $\Face (M/N)$.
\end{enua}
\end{prp}
\begin{proof}
(1)
It is clear that $X/N$ is a submonoid of $M/N$.
Let $a$ and $b$ elements in $M$ such that $\pi(a)+\pi(b) \in X/N$.
Then there exist $x\in X$ and $n, n'\in N$ such that $x+n=(a+b)+n'$.
Since $x+n$ belongs to the face $X$ of $M$, both $a \in X$ and $b\in X$ hold.
Thus both $\pi(a)$ and $\pi(b)$ belong to $X/N$,
which shows that $X/N$ is a face of $M/N$.

(2)
This is clear from the definitions.

(3)
We have that $\pi^{-1}(X')/N=\pi(\pi^{-1}(X'))=X'$ for any face $X'$ of $M/N$
since $\pi$ is surjective.
Let $X$ be a face of $M$ containing $N$.
It is easily seen that $\pi^{-1}(X/N)\supseteq X$.
Take $a \in \pi^{-1}(X/N)$.
Then we have $\pi(a)\in X/N$,
which implies that there exist $x\in X$ and $n, n'\in N$ such that $x+n=a+n'$.
Since $a + n' = x + n$ belongs to $X$ the face of $M$, we have $a\in X$.
Therefore we conclude that $\pi^{-1}(X/N) = X$.
\end{proof}

\begin{cor}\label{cor:face unit}
Let $N$ be a submonoid of a monoid $M$.
There is an inclusion-preserving bijection between $\Face(M/N)$ and $\Face(M/\la N \ra_{\face})$.
In particular,
we have an inclusion-preserving bijection $\Face(M) \iso \Face(M/M^{\times})$,
where $M^{\times}$ is the set of units of $M$.
\end{cor}
\begin{proof}
  Proposition \ref{prp:face quot} shows that both $\Face (M/N)$ and $\Face(M/\la N \ra_{\face})$ are in bijection with $\{ X \in \Face M \mid X \supseteq N \}$ by the definition of $\la N \ra_{\face}$. Thus the former assertion holds.
  The latter assertion follows from Remark \ref{rmk:Face} and the former one by putting $N := 0$.
\end{proof}

Next, we recall the localization of a monoid,
which is a monoid analogue of the localization of a commutative ring.
\begin{dfn}\label{def:mon-loc}
Let $M$ be a monoid and $S$ a subset of $M$.
The \emph{localization of $M$ with respect to $S$} is 
a monoid $M_S$ together with a monoid homomorphism $\rho \colon M\to M_S$
which satisfies the following universal property:
\begin{enur}
\item
$\rho(s)$ is invertible in $M_S$ for each $s\in S$.
\item
For any monoid homomorphism $\phi \colon M \to X$ 
such that $\phi(s)$ is invertible for each $s\in S$,
there is a unique monoid homomorphism $\ol{\phi} \colon M_S \to X$ satisfying $\phi=\ol{\phi}\rho$.
\end{enur}
\end{dfn}
The localization of a monoid $M$ with respect to a subset $S \subseteq M$ actually exists,
which is constructed as follows:
Define a binary relation on $M \times \la S \ra_{\bbN}$ by
\[
(x,s) \sim (y,t) :\equi \text{there exist $u\in \la S \ra_{\bbN}$ such that $x+t+u=y+s+u$ in $M$.}
\]
It is a congruence on the monoid $M \times \la S \ra_{\bbN}$,
and hence the quotient set $M_S:=M\times \la S \ra_{\bbN} /{\sim}$ becomes a monoid.
We denote by $[x,s]$ the equivalence class of $(x,s)\in M\times \la S \ra_{\bbN}$.
We can think of $[x,s]$ as ``$x-s$.''
Then it is straightforward to check that
the monoid $M_S$ together with a monoid morphism $\rho \colon M \to M_S$ defined by $\rho(m)=[m,0]$
is the localization of $M$ with respect to $S$.
We call $\rho \colon M \to M_S$ the \emph{localization homomorphism} of $M$ with respect to $S$.

We reveal the relationship between faces of $M_S$ and those of $M$
in Proposition \ref{prp:face loc} below.
Let us prove two lemmas for this purpose.
\begin{lem}\label{lem:face loc 1}
Let $S$ be a subset of a monoid $M$.
Then the monoid homomorphism
\[
M_S \to M_{\la S \ra_{\face}},\quad
[x,s] \mapsto [x,s]
\]
is an isomorphism.
\end{lem}
\begin{proof}
Let $\rho\colon M \to M_S$ be the localization homomorphism.
We have $\rho^{-1}(M_S^{\times}) = \la S \ra_{\face}$ 
by \cite[the text following Proposition 1.4.4]{Og}.
Then we can check that $\rho$ also satisfies the universal property of the localization of $M$ with respect to $\la S \ra_{\face}$, so the assertion holds.
\end{proof}

\begin{lem}\label{lem:face loc 2}
Let $S$ be a face of a monoid $M$.
\begin{enua}
\item
$M/S$ is reduced (see Definition \ref{def:mon-red-can}).
\item
The monoid homomorphism
\[
\phi \colon M_S/M_S^{\times} \to M/S,\quad
[x,s] \bmod M_S^{\times} \to x \bmod S
\]
is an isomorphism.
\end{enua}
\end{lem}
\begin{proof}
(1)
Let $x,y\in M$ such that $x+y \equiv 0 \bmod S$.
There are elements $s,t \in S$ such that $x+y+s=t$ in $M$.
Since $S$ is a face, both $x$ and $y$ belong to $S$,
which implies both $x \equiv 0 \bmod S$ and $y \equiv 0 \bmod S$.
Therefore $M/S$ is reduced.

(2)
The quotient homomorphism $M\to M/S$ induces 
a monoid homomorphism $\phi' \colon M_S \to M/S$ by the universal property of $M_S$, which satisfies $\phi'([x,s]) = x \bmod S$.
Then $\phi'(M_S^{\times})=0$ since $M/S$ is reduced.
Thus $\phi'$ induces a monoid homomorphism $\phi\colon M_S/M_S^{\times} \to M/S$
by the universal property of $M_S/M_S^{\times}$, which satisfies $\phi([x,s] \bmod M_S^\times) = x \bmod S$.
The homomorphism $\phi$ is clearly surjective.
We prove that $\phi$ is injective.
Let $[x,s], [y,t]\in M_S$ such that $x\equiv y \bmod S$ in $M/S$.
Then there are $n,n' \in S$ such that $x+n=y+n'$ in $M$.
Hence we have the following equalities in $M_S$:
\[
[x,s]+[s+n,0]=[x+s+n,s]=[x+n,0]=[y+n',0]=[y,t]+[t+n',0].
\]
We conclude that $[x,s]\equiv [y,t] \bmod M_S^{\times}$
because $[s+n,0], [t+n',0] \in M_S^{\times}$.
Therefore $\phi$ is injective.
\end{proof}

\begin{prp}\label{prp:face loc}
Let $M$ be a monoid and $S$ a subset of $M$.
Then there is an inclusion-preserving bijection between the following sets:
\begin{enua}
\item
$\Face(M_S)$.
\item
$\{ F \in \Face M \mid F \supseteq S \}$.
\end{enua}
\end{prp}
\begin{proof}
We have the following inclusion-preserving bijections:
\[
\Face(M_S) \iso \Face\left(M_{\la S \ra_{\face}}\right) 
\iso \Face\left(M_{\la S \ra_{\face}}/M_{\la S \ra_{\face}}^{\times}\right)
\iso \Face(M/\la S \ra_{\face})
\]
by Lemma \ref{lem:face loc 1}, \ref{lem:face loc 2}, and Corollary \ref{cor:face unit}.
The set $\Face(M/\la S \ra_{\face})$ corresponds bijectively to
the set (2) by Proposition \ref{prp:face quot}.
\end{proof}

We now discuss the relationship between monoids and groups.
Let $M$ be a monoid.
The \emph{group completion} $\gp M = (\gp M, \rho \colon M \to \gp M)$ of $M$
is the localization of $M$ with respect to $S=M$.
Note that $\gp M$ is an abelian group.
It has the following universal property:
\begin{prp}\label{prop:gp-compl-univ}
Let $M$ be a monoid.
The group completion $(\gp M, \rho \colon M \to \gp M)$ of $M$
satisfies the following universal property:
\begin{itemize}
\item 
For every monoid homomorphism $f\colon M \to G$ into a group $G$,
there exists a unique group homomorphism $\ol{f}\colon \gp M \to G$
satisfying $f=\rho \ol{f}$
\end{itemize}
\end{prp}

The assignment $M \mapsto \gp M$ gives rise to a functor $\gp \colon \Mon \to \Ab$ by the universal property.
Also,
the assignment $M\mapsto M^{\times}$ gives rise to a functor $(-)^{\times} \colon \Mon \to \Ab$
since units are preserved by monoid homomorphisms.
\begin{prp}\label{prp:adj grp}
The forgetful functor $\Ab \to \Mon$
has both the left adjoint $\gp \colon \Mon \to \Ab$
and the right adjoint $(-)^{\times}\colon \Mon \to \Ab$.
\end{prp}
\begin{proof}
This follows from Proposition \ref{prop:gp-compl-univ} and 
the fact that units are preserved by monoid homomorphisms.
\end{proof}

Finally, we list properties of monoids which we use in this paper.
\begin{dfn}\label{def:mon-red-can}
A monoid $M$ is \emph{reduced}
if $a+b=0$ implies $a=b=0$ for any $a,b\in M$. This is equivalent to $M^\times = 0$.
\end{dfn}


\begin{dfn}\label{def:subt-cofin}
Let $S$ be a subset of a monoid $M$.
\begin{enua}
\item
$S$ is \emph{subtractive} 
if $x+y\in S$ and $x\in S$ imply $y\in S$ for any $x,y\in M$.
\item
$S$ is \emph{cofinal} 
if, for any $x\in M$, there exists $y\in M$ satisfying $x+y\in S$. 
\end{enua}
\end{dfn}

\begin{rmk}\label{rem:grp-subt-cof}
Let $M$ be a monoid.
\begin{enua}
\item
If $M$ is a group, then a subtractive submonoid is nothing but a subgroup.
\item
If $M$ is a group,
then any submonoid $N$ of $M$ is cofinal since $x+(-x)=0\in N$ for all $x\in M$. 
\item
We can define a pre-order $\le$ on any monoid $M$ by
\[
x\le y :\equi \text{there exists some $a\in M$ such that $y=x+a$.}
\]
A cofinal subset of $M$ defined as above is nothing but a cofinal subset of $M$
with respect to this pre-order $\le$.
\end{enua}
\end{rmk}

\medskip
\noindent
{\bf Acknowledgement.}
The authors would like to thank Hiroyuki Nakaoka and Arashi Sakai for valuable discussions and comments.
They would also like to thank Yasuaki Ogawa for interesting information.
H. E. is supported by JSPS KAKENHI Grant Number JP21J00299.
S. S. is supported by JSPS KAKENHI Grant Number JP21J21767.

\end{document}